\documentclass[reqno,12pt]{amsart}

\usepackage{latexsym}
\usepackage{amsthm,amsmath, amssymb,amsfonts}
\usepackage{mathrsfs}
\def\R{\mathbb R}
\def\RP{\mathbb {RP}}
\def\N{\mathbb N}

\theoremstyle{plain}

\theoremstyle{plain}
\newtheorem{thm}{Theorem}[section]
\newtheorem{cor}[thm]{Corollary}
\newtheorem{lem}[thm]{Lemma}
\newtheorem{prop}[thm]{Proposition}

\newtheorem*{thmA}{Theorem A}
\newtheorem*{thmB}{Theorem B}
\newtheorem*{thmC}{Theorem C}
\newtheorem*{thmD}{Theorem D}
\newtheorem*{thmE}{Theorem E}
\newtheorem*{lemA1}{Lemma A.1}
\newtheorem*{corA2}{Corollary A.2}

\theoremstyle{definition}
 
\newtheorem{definition}[thm]{Definition}
\newtheorem{rmk}[thm]{Remark}

\newcommand{\norm}[1]{\left\Vert#1\right\Vert}  
\def\Xint#1{\mathchoice
   {\XXint\displaystyle\textstyle{#1}}%
   {\XXint\textstyle\scriptstyle{#1}}%
   {\XXint\scriptstyle\scriptscriptstyle{#1}}%
   {\XXint\scriptscriptstyle\scriptscriptstyle{#1}}%
   \!\int}
\def\XXint#1#2#3{{\setbox0=\hbox{$#1{#2#3}{\int}$}
     \vcenter{\hbox{$#2#3$}}\kern-.5\wd0}}

\def\dashint{\Xint-}

\begin{document}

\title[]{Riemannian metrics on the sphere with Zoll families of minimal hypersurfaces}

\author{Lucas Ambrozio, Fernando C. Marques and Andr\'e Neves}

\address{IMPA - Instituto Nacional de Matematica Pura e Aplicada, Rio de Janeiro, RJ, Brasil, 22460-320.}
\email{\texttt{l.ambrozio@impa.br}}

\address{Princeton University, Princeton, NJ, USA, 08544.}
\email{\texttt{coda@math.princeton.edu}}

\address{University of Chicago, Chicago, IL, USA, 60637.}
\email{\texttt{aneves@uchicago.edu}}

\thanks{The first author was funded by the Ambrose Monell Foundation while member of the IAS in 2018/2019. The second author is partly supported by NSF-DMS-2105557 and a Simons Investigator Grant. The third author is partly supported by NSF-DMS-2005468 and a Simons Investigator Grant.}

\begin{abstract}
In this paper we construct smooth Riemannian metrics on the sphere which admit  smooth Zoll families of minimal hypersurfaces.
This generalizes a theorem of  Guillemin  for the case of geodesics. The proof uses the Nash-Moser Inverse Function Theorem in the tame maps setting of Hamilton. This answers a question of Yau on   perturbations of minimal hypersurfaces   in  positive Ricci curvature. We also consider the case of  the projective space and characterize those metrics on the sphere with minimal equators.
\end{abstract}

\maketitle

\begin{tiny}
\setcounter{tocdepth}{1}
\tableofcontents
\thispagestyle{empty}
\end{tiny}

\section{Introduction}

In this paper we construct smooth Riemannian metrics on the sphere which admit $n$-parameter families of $(n-1)$-dimensional minimal hypersurfaces as the canonical family of equators. These are the intersections of $n$-dimensional linear hyperplanes of $\mathbb{R}^{n+1}$ with $S^n$ (the unit sphere) and are totally geodesic for the canonical metric $can$. They  form a family of embedded $(n-1)$-dimensional spheres smoothly parametrized by the projective space ${\mathbb{RP}}^n$. In addition there is a unique equator tangent  to any $(n-1)$-dimensional space in the tangent space of any point.  
We say that a Riemannian metric $g$ on  $S^n$  is in $\mathcal{Z}$  if there exists a smooth family $\{\Sigma_\sigma\}_{\sigma\in \RP^{n}}$ of smoothly embedded $(n-1)$-dimensional minimal spheres  in $(S^n,g)$  such that for each $x\in S^n$ and each $(n-1)$-dimensional space $\pi\subset T_xS^n$ there exists a unique $\sigma\in \RP^{n}$ such that $T_x\Sigma_\sigma=\pi$  ($\{\Sigma_\sigma\}_\sigma$ is called a Zoll family), and that $g$ is in $\mathcal{Z}'$ if moreover $area(\Sigma_\sigma,g)= area(S^{n-1},can)$ for all $\sigma \in \mathbb{RP}^n$.

\begin{thmA} \label{thmintromain} 
	Let $\dot{\rho}$ be a smooth odd function on the sphere $S^n$, $n\geq 3$. Then there exists a smooth one-parameter family of smooth functions $\rho(t)$ on $S^n$, $-\delta < t< \delta$, with $\rho(0)=0$ and $\rho'(0)=\dot{\rho}$
such that $e^{2\rho(t)}can\in \mathcal{Z}'$ for every $t$.  In fact there exists a neighborhood $W$ of the origin in the space of smooth odd functions and  a smooth  map 
$$
\lambda: W \rightarrow \mathcal{Z}'
$$
with $\lambda(0)=can$ and $D\lambda(0)\cdot \psi  =2\psi can.$
\end{thmA}

The existence of metrics in the space $\mathcal{Z}$  (other than the constant curvature metrics) was proven by Zoll \cite{Zol} in the case of the two-sphere. These examples are rotationally symmetric. Riemannian metrics with all geodesics closed and of the same length are called Zoll metrics (\cite{Bes}).  

Our  Theorem A was proven by Guillemin \cite{Gui} for the two-sphere. Funk \cite{Fun} knew that  the function $\dot{\rho}$ must be odd for the deformation to exist. Funk used his transform that sends a function $\psi$ on the sphere to the function
$$
\gamma \mapsto \int_\gamma \psi ,
$$
where $\gamma$ denotes a closed geodesic in the round metric. A generalized version of this map will be key in our argument.

The proof of  Theorem A will use the Nash-Moser  Inverse Function Theorem  stated by Hamilton \cite{Ham}. We will make use of a
variational interpretation of the problem, differing from the dynamical approach for geodesics of Guillemin. Our proof also recovers Guillemin's result.

The methods used in the proof of Theorem A can be applied to the case of general metric variations (see Theorem E).

We also prove that metrics like in Theorem A do not exist on $\mathbb{RP}^3$:

\begin{thmB}\label{thmb} The  canonical metric in $S^3$ has an open neighborhood $\mathcal{W}$ in the smooth topology  so that if $\rho$ is an even function with   $e^{2\rho}can \in \mathcal{W} \cap \mathcal{Z}$,  then $\rho$ is a constant function.
\end{thmB}

In the discussion of  three-manifolds with positive Ricci curvature in \cite{Yau} (page 128)  Yau posed the question of
whether there can be continuous families of minimal surfaces when the ambient has no symmetry. The next theorem proves that this is actually possible.

{\begin{thmC} For every $n\geq 3$, there exist metrics $e^{2\rho}can \in \mathcal{Z}$ on $S^n$ which are arbitrarily close to the canonical metric and have trivial isometry group.
\end{thmC}}

Green showed that Zoll metrics on $\RP^2$ have constant curvature \cite{Green}.  The next theorem is a consequence of the classification of metrics on $S^n$ with minimal equators (Section \ref{proof.theorem.D}). This uses the method of Hangan \cite{Han1} who classified metrics on $\mathbb{R}^n$ with minimal hyperplanes.

\begin{thmD} \label{thmcorrespondenceA}
	There exist metrics on $\mathbb{RP}^3$ with minimal equators and discrete isometry group.
	\end{thmD}

We also obtain metrics on $\mathbb{RP}^n$, for each $n\geq 3$, with minimal equators,   that do not have constant sectional curvature.

Theorem A can be generalized to include nonconformal variations:

\begin{thmE} 
	Let $h$ be a smooth symmetric two-tensor on the sphere $S^n$, $n\geq 2$, of the form 	\begin{equation} \label{eqsomeelementsofthekernel}
		h=fcan + \mathcal{L}_X(can) + \overline{h},
\end{equation}	 
where $f$ is a smooth odd function on $S^n$, $\mathcal{L}_X$ is the Lie derivative by a smooth vector field $X$ on $S^n$, and $\overline{h}$ is a transverse-traceless symmetric two-tensor on $(S^n,can)$. Then there exists a smooth one-parameter family of Riemannian metrics $g(t)$ on $S^n$, $-\delta < t< \delta$, with $g(0)=can$ and $g'(0)=h$,
	such that $g(t)\in \mathcal{Z}'$ for every $t$. In fact there exists a neighborhood $W$ of the origin in the space of smooth tensors $h$ as in \eqref{eqsomeelementsofthekernel} and a smooth map 
$$
\lambda: W \rightarrow \mathcal{Z}'
$$
with $\lambda(0)=can$ and $D\lambda(0)\cdot h =h.$ 
\end{thmE}

For the case of the three-sphere, Theorem E (or Theorem A) provides examples of nonhomogeneous metrics for which the uniqueness theorem of  G\'alvez and Mira \cite{GalMir} applies: for such metrics any minimal sphere belongs to the Zoll family.

 For $n\geq 3$, Riemannian metrics with  Zoll families of totally geodesic hypersurfaces have constant sectional curvature.


\section{Graphical perturbations of equators}\label{prelim}

We will start by describing the families of spheres modeled after the family of equators.
Let $S^n$, $\mathbb{RP}^n$ be the sphere and projective space respectively. 
Given $v \in S^n$, let
\begin{equation*}
	\Sigma_{v} = \{p\in S^n:\, \langle p, v\rangle =0\}
\end{equation*}
denote the equator that is orthogonal to $v$. As sets,  $\Sigma_{-v}=\Sigma_v$. 
The assignment
\begin{equation*}
	\sigma = [v] \in {\mathbb{RP}}^n \mapsto \Sigma_{\sigma}= \Sigma_{v} \subset S^n
\end{equation*}
is a  bijection between ${\mathbb{RP}}^n$ and the set of equators of $S^n$.  \\
\indent The standard metric on $S^n$ is denoted by $can$. The unit tangent bundle of $S^n$ is
\begin{equation*}
	T_{1}S^n = \{ (p,v)\in S^n\times S^n:\, \langle p,v \rangle =0 \}.
\end{equation*}

\indent Consider the space
$
	C^{\infty}_{*,odd}(T_{1}S^n)
$
of smooth functions $\Phi$ defined on  $T_1S^n$ that are odd with respect to the second variable:
\begin{equation*}
	\Phi(p,-v) = - \Phi(p,v).
\end{equation*}
\indent Given $\Phi \in C^{\infty}_{*,odd}(T_{1}S^n)$, for each $v\in S^n$ we define the set
\begin{equation}\label{graph}
	\Sigma_{v}(\Phi) = \{\cos(\Phi(x,v))x+\sin(\Phi(x,v))v\in S^n:\, x\in \Sigma_v\}.
\end{equation}
Clearly, $\Sigma_v(0)=\Sigma_v$. We will also use $\Sigma_{v}(\Phi)$ to denote the map
\begin{equation*}
	\Sigma_v(\Phi) : x \in \Sigma_v \mapsto \cos(\Phi(x,v))x+\sin(\Phi(x,v))v \in S^n.
\end{equation*}
We think of $\Sigma_v(\Phi)$ as the normal graph over the equator $\Sigma_{v}$ of the function 
\begin{equation*}
	\Phi_v=\Phi(-,v) \in C^{\infty}(\Sigma_v).
\end{equation*}
Notice that $\Sigma_{-v}(\Phi)=\Sigma_v(\Phi)$ both as sets and as maps.
Hence we can assign without ambiguity to each point $\sigma=[v]\in {\mathbb{RP}}^n$ the hypersurface $\Sigma_{\sigma}(\Phi)=\Sigma_{v}(\Phi)$ of $S^n$.

Assume that $\Phi$ has sufficiently small $C^{1}$ norm so that all maps $\Sigma_v(\Phi)$ are smooth embeddings of the $(n-1)$-dimensional sphere into $S^n$.  We will require more conditions on $\Phi$ later as needed.

\begin{lem} \label{lembasicformulae}
The following formulas hold:
\begin{itemize}
	\item[$i)$] The tangent space of $\Sigma_v(\Phi)$ at the point $y=\Sigma_v(\Phi)(x)$ consists of the vectors
	\begin{equation*}
		\cos(\Phi(x,v))u + D\Phi_{(x,v)}\cdot (u,0)\, \Sigma^{\perp}_{v}(\Phi)(x), \quad u\in T_{x}\Sigma_v,
	\end{equation*}
		where
	\begin{equation*}
		\Sigma^{\perp}_v(\Phi)(x)= -\sin(\Phi(x,v))x + \cos(\Phi(x,v))v.
	\end{equation*}			
	\item[$ii)$] The unit normal of $\Sigma_v(\Phi)$ at the point $y=\Sigma_v(\Phi)(x)$ that points towards $v$ is the vector
	\begin{equation*}
	N_{v}(\Phi)(x) = \frac{\cos(\Phi(x,v))\Sigma^{\perp}_v(\Phi)(x) - \nabla^{\Sigma_v}\Phi_v(x)}{\sqrt{\cos(\Phi(x,v))^2+|\nabla^{\Sigma_v}\Phi_v|^2(x)}},
	\end{equation*}
	where $\nabla^{\Sigma_v}\Phi_v$ is the gradient of $\Phi_v$ in $(\Sigma_{v},can)$. In particular,
	\begin{equation*}
		N_{-v}(\Phi)(x) = - N_{v}(\Phi)(x) \quad \text{for all} \quad x\in \Sigma_{v}.
	\end{equation*}
	\item[$iii)$] The Jacobian determinant of the map $\Sigma_v(\Phi)$ at $x\in \Sigma_v$ is
	\begin{equation*}
		|Jac\, \Sigma_v(\Phi)|(x)= \cos^{n-2}(\Phi(x,v))\sqrt{\cos(\Phi(x,v))^2 + |\nabla^{\Sigma_v}\Phi_v|^{2}(x)}.
	\end{equation*}
\end{itemize}
\end{lem}

\begin{proof}
Let $u\in T_x\Sigma_v$. Then
\begin{multline*}
D\Sigma_v(\Phi)_x\cdot u=\cos(\Phi(x,v))u\\
+D\Phi_{(x,v)}\cdot (u,0) \big(-\sin(\Phi(x,v))x + \cos(\Phi(x,v))v\big),
\end{multline*}
which proves $i)$.

Note that $\langle \Sigma^{\perp}_v(\Phi)(x),u\rangle =0$ for every $u\in T_x\Sigma_v$ and
$|\Sigma^{\perp}_v(\Phi)(x)|^2=1$. Hence
$$
|\cos(\Phi(x,v))\Sigma^{\perp}_v(\Phi)(x) - \nabla^{\Sigma_v}\Phi_v(x)|^2=\cos(\Phi(x,v))^2+|\nabla^{\Sigma_v}\Phi_v|^2(x).
$$

Now 
\begin{multline*}
\langle \cos(\Phi(x,v))\Sigma^{\perp}_v(\Phi)(x) - \nabla^{\Sigma_v}\Phi_v(x),\Sigma_v(\Phi)(x) \rangle \\
=\cos(\Phi(x,v)) \langle \Sigma^{\perp}_v(\Phi)(x),\cos(\Phi(x,v))x+\sin(\Phi(x,v))v \rangle=0.
\end{multline*}

For $u\in T_x\Sigma_v$,
\begin{multline*}
\langle \cos(\Phi(x,v))\Sigma^{\perp}_v(\Phi)(x) - \nabla^{\Sigma_v}\Phi_v(x),\\
\cos(\Phi(x,v))u
+D\Phi_{(x,v)}\cdot (u,0) \Sigma^{\perp}_v(\Phi)(x)\rangle\\
=\cos(\Phi(x,v))D\Phi_{(x,v)}\cdot (u,0)-\cos(\Phi(x,v))\langle \nabla^{\Sigma_v}\Phi_v(x), u\rangle=0,
\end{multline*}
finishing the proof of $ii)$.

Let $\{u_i\}\subset T_x\Sigma$ be an orthonormal basis.  Then
\begin{align*}
g_{ij} & = \langle \cos(\Phi(x,v))u_i + D\Phi_{(x,v)}(u_i,0)\Sigma^{\perp}_{v}(\Phi)(x),\\
& \quad \quad \cos(\Phi(x,v))u_j + D\Phi_{(x,v)}(u_j,0)\Sigma^{\perp}_{v}(\Phi)(x)\rangle \\
& =\cos(\Phi(x,v))^2\delta_{ij}+\langle \nabla^{\Sigma_v}\Phi_v(x), u_i\rangle \langle \nabla^{\Sigma_v}\Phi_v(x), u_j\rangle.
\end{align*}
Hence
\begin{equation*}
\det g_{ij}=\cos(\Phi(x,v))^{2(n-2)}\big(\cos(\Phi(x,v))^2+ |\nabla^{\Sigma_v}\Phi_v(x)|^2\big),
\end{equation*}
proving $iii)$.
\end{proof}

\subsection{Incidence sets}\label{section.incidence} Given the family of hypersurfaces $\{\Sigma_\sigma(\Phi)\}$, define its incidence set by
\begin{equation*}
	[F(\Phi)]= \{ (y,\sigma) \in S^{n}\times {\mathbb{RP}}^n:\, y\in \Sigma_\sigma(\Phi)\}
\end{equation*}
and its orientation-inducing incidence set by
\begin{equation*}
	F(\Phi)= \{ (y,v) \in S^{n}\times S^n:\, y\in \Sigma_{v}(\Phi)\}.
\end{equation*}
The 
 projection of $S^n\times S^n$ onto $S^n\times {\mathbb{RP}}^n$ induces a two-to-one map $F(\Phi) \rightarrow [F(\Phi)]$. 
Clearly, $F(0)=T_{1}S^n$. 

The set $F(\Phi)$ is the image of the smooth map
\begin{equation} \label{eqparamincidset}
	(x,v) \in T_1S^n \mapsto (\Sigma_v(\Phi)(x),v) \in S^n\times S^n.
\end{equation}
Since this is a $C^1$-perturbation of the inclusion of $T_1S^n$ into $S^n\times S^n$, we obtain

\begin{prop} \label{propgelfand1}
	The sets $F(\Phi)$ and $[F(\Phi)]$ are smooth, embedded, compact hypersurfaces of $S^n\times S^n$ and $S^n\times {\mathbb{RP}}^n$ respectively.
\end{prop}

 and

\begin{prop} \label{propgelfand2}
	The projections
	\begin{equation*}
		\pi_1:(p,v) \in F(\Phi) \mapsto p\in S^n, \quad \pi_2: (p,v)\in F(\Phi) \mapsto v\in S^n
	\end{equation*}
	and
	\begin{equation*}
		\pi_1: (p,\sigma) \in [F(\Phi)] \mapsto p\in S^n, \quad \pi_2: (p,\sigma)\in [F(\Phi)] \mapsto \sigma\in {\mathbb{RP}}^n
	\end{equation*}
	are smooth submersions.
\end{prop}

For $p\in S^n$, we define the dual hypersurface
$$
	\Sigma^*_p(\Phi)=\{\sigma \in \mathbb{RP}^n: (p,\sigma) \in [F(\Phi)]\}.
	$$
	If  $\Phi=0$ these sets are the linear projective hyperplanes in ${\mathbb{RP}}^n$:
\begin{equation*}
	\Sigma^{*}_p=\Sigma^*_{p}(0)=\{ [v]\in {\mathbb{RP}}^n:\, \langle v,p\rangle =0\},
\end{equation*}
parametrized by points $p \in S^n$.

	Notice that
\begin{equation*}
	p \in \Sigma_{\sigma}(\Phi)  \quad \Leftrightarrow \quad \sigma \in \Sigma^{*}_{p}(\Phi).
	\end{equation*}
	We have 
	$
	\Sigma_{\sigma}(\Phi) = \pi_{1}(\pi_2^{-1}(\sigma))\subset S^n
$
and
$
	\Sigma^{*}_{p}(\Phi) := \pi_{2}(\pi_1^{-1}(p)) \subset {\mathbb{RP}}^n.
$

\subsection{Generalized Gauss map}\label{section.gauss.map} It will be useful to consider $\Sigma(\Phi)$ as the map 
\begin{equation*}
	\Sigma(\Phi): (x,v)\in T_{1}S^n \mapsto \Sigma_{v}(\Phi)(x) \in S^n,
\end{equation*}
and similarly the unit normal (with respect to $can$) as the map 
\begin{equation*}
	N(\Phi) : (x,v)\in T_{1}S^n \mapsto N_{v}(\Phi)(x) \in S^n.
\end{equation*}
The generalized Gauss map of the family $\{\Sigma_{\sigma}(\Phi)\}$ is the map
\begin{equation*}
	\mathcal{G}(\Phi) : (x,v) \in T_{1}S^n \mapsto (\Sigma(\Phi)(x,v),N(\Phi)(x,v)) \in T_{1}S^n.
\end{equation*}
The map $\mathcal{G}(0)$ is the identity. 

Assume that $\Phi$ has sufficiently small $C^2$ norm. Then Lemma \ref{lembasicformulae} implies:
\begin{prop} \label{propGaussmap}
	The generalized Gauss map $\mathcal{G}(\Phi)$ is a smooth diffeomorphism of $T_1S^n$.
	\end{prop}
	
	\begin{proof}
	$\mathcal{G}(\Phi)$ is $C^1$ close to the identity map.
	\end{proof}

	
\section{The area functional}

\indent We denote by $C^{\infty}_{even}(S^n)$ the set of smooth even functions $h$ on $S^n$ ($h(-v)=h(v)$) and by $C^{\infty}_{odd}(S^n)$ the set of smooth odd functions $h$ on $S^n$ ($h(-v)=-h(v)$).

 Let $\rho$ be a smooth function on $S^n$, and $\Phi\in C^\infty_{*,odd}(T_1S^n)$. Let $\mathcal{A}(\rho,\Phi)$ denote the map that assigns to each $\sigma\in {\mathbb{RP}}^n$ the area of the surface $\Sigma_{\sigma}(\Phi)$ computed with respect to the conformal metric $e^{2\rho}can$. Explicitly, for all $\sigma=[v]\in {\mathbb{RP}}^n$,
\begin{align} \label{eqformulaA}
	\mathcal{A}(\rho,\Phi)(\sigma) & = area(\Sigma_{\sigma}(\Phi),e^{2\rho}can)\nonumber  \\ 
	  & = \int_{\Sigma_v} e^{(n-1)\rho(\Sigma_v(\Phi)(x))}|Jac(\Sigma_v(\Phi))|(x)dA_{can}(x).
\end{align}
Hence $\mathcal{A}(\rho,\Phi)\in C^{\infty}({\mathbb{RP}}^n)$.
It will be convenient to think of $\mathcal{A}(\rho,\Phi)$ as a function in $C^{\infty}_{even}(S^n)$ as well, and the notation will change accordingly.

\subsection{The first variation and the Euler-Lagrange operator} From the  identity  \eqref{eqformulaA} we have
\begin{equation*}
	\mathcal{A}(\rho,\Phi)(v) = \int_{\Sigma_v} A(\rho)((x,v),\Phi(x,v),\nabla^{\Sigma_v}\Phi_v(x))dA_{can}(x)
\end{equation*} 
for the smooth function
$
	A(\rho): T_{1}S^n \times \mathbb{R} \times \mathbb{R}^{n+1} \rightarrow \mathbb{R}
$
given explicitly by
\begin{equation*}
 A(\rho)((x,v),U,L) = e^{(n-1)\rho(\cos(U)x+\sin(U)v)}\cos^{n-2}(U)\sqrt{\cos^{2}(U)+|L|^2}.
\end{equation*}

Hence, for every $\phi\in C^{\infty}_{*,odd}(T_1S^n)$ and $v\in S^n$, the identity 
\begin{equation} \label{eqD2A}
	\frac{d}{dt}_{\vert t=0}\mathcal{A}(\rho,\Phi+t\phi)(v) = \int_{\Sigma_v}\mathcal{H}(\rho,\Phi)(x,v)\phi(x,v)dA_{can}(x)
\end{equation}
holds for the map
\begin{multline} \label{eqH}
	\mathcal{H}(\rho,\Phi)(x,v) = -div_{(\Sigma_v,can)} \left(D_3A(\rho)((x,v),\Phi(x,v),\nabla^{\Sigma_v}\Phi_v(x))\right) 
	\\ +  D_2A(\rho)((x,v),\Phi(x,v),\nabla^{\Sigma_v}\Phi_v(x))),
\end{multline}
which we call the Euler-Lagrange operator. (In the first term on the right-hand side, $v$ is frozen and the divergence is the divergence, at the point $x$, of a vector field that is tangent to $\Sigma_v$). \\
\indent In view of \eqref{eqH}, $\mathcal{H}(\rho,\Phi)$ is a smooth function on $T_{1}S^n$ depending on up to one derivative of $\rho$ and up to two derivatives of $\Phi$. Note that  $\mathcal{H}(\rho,\Phi)\in C^\infty_{*,odd}(T_1S^n).$ 

 The first variation formula for the area says that
\begin{align*}
&\frac{d}{dt}_{\vert t=0}\mathcal{A}(\rho,\Phi+t\phi)(v) & \\
& \quad \quad =\int_{\Sigma_v(\Phi)} \hat{\mathcal{H}}(\rho,\Phi)(x,v) & \\
& \hspace{2cm} \cdot e^{2\rho(y)} \langle \frac{d}{dt}_{\vert t=0}\Sigma_v(\Phi+t\phi)(x), e^{-\rho(y)}N_v(\Phi)(x)\rangle dA_{e^{2\rho}can}(y) & \\
& \quad \quad =\int_{\Sigma_v} \hat{\mathcal{H}}(\rho,\Phi)(x,v) e^{n\rho(\Sigma_v(\Phi)(x))} \cos^{n-1}(\Phi(x,v)) \phi(x,v) dA_{can}(x), &
\end{align*}
where $\hat{\mathcal{H}}(\rho,\Phi)(x,v)$ is the mean curvature of $\Sigma_v(\Phi)$ in $(S^n,e^{2\rho}can)$ at $y=\Sigma_v(\Phi)(x)$. Hence
$$
\mathcal{H}(\rho,\Phi)(x,v) =\hat{\mathcal{H}}(\rho,\Phi)(x,v) e^{n\rho(\Sigma_v(\Phi)(x))} \cos^{n-1}(\Phi(x,v)).
$$

\indent The geometric meaning of the equation  
\begin{equation*}
	\mathcal{H}(\rho,\Phi) = 0
\end{equation*}
is therefore that it holds if and only if all hypersurfaces $\Sigma_{\sigma}(\Phi)$ are critical points of the $(n-1)$-dimensional area functional of $(S^n,e^{2\rho}can)$. In other words, $\mathcal{H}(\rho,\Phi)=0$ if and only if the family $\{\Sigma_\sigma(\Phi)\}_{\sigma\in {\mathbb{RP}}^n}$ consists of minimal hypersurfaces of $(S^n,e^{2\rho}can)$.

\subsection{The linearization of $\mathcal{H}(\rho,\Phi)$} 

\begin{prop} \label{propDH_1}
	The following formulas hold:
\begin{itemize} 
	\item[$i)$] For every $f\in C^{\infty}(S^n)$,
	\begin{eqnarray*}
			&&\frac{d}{dt}_{|t=0}\mathcal{H}(\rho+tf,\Phi) = (n-1)(f\circ \Sigma(\Phi))\mathcal{H}(\rho,\Phi)  \\ 
			&& \hspace{1cm}+ (n-1)(\cos\circ \Phi)^{n-1} \left\langle \nabla f \circ \Sigma(\Phi), N(\Phi) \right\rangle e^{(n-1)\rho\circ \Sigma(\Phi)}.
		\end{eqnarray*}
		\item[$ii)$] For every $\phi \in C^{\infty}_{*,odd}(T_1S^n)$ and  $v\in S^n$,
		$$
		\left(\frac{d}{dt}_{|t=0}\mathcal{H}(\rho,\Phi+t\phi)\right)_{v}=\mathcal{J}_v(\rho,\Phi) (\phi_v)
		$$
		 for some  symmetric second-order linear  partial differential operator $\mathcal{J}_v(\rho,\Phi): C^{\infty}(\Sigma_v) \mapsto  C^{\infty}(\Sigma_v)$.
		  Moreover, $\mathcal{J}_{-v}(\rho,\Phi)=\mathcal{J}_v(\rho,\Phi)$.
\end{itemize}
\end{prop}

\begin{proof}
From the formula of the conformal change of the mean curvature we have
$$
\hat{\mathcal{H}}(\rho+tf,\Phi)=e^{-tf \circ \Sigma(\Phi)} \big(\hat{\mathcal{H}}(\rho,\Phi)+(n-1)e^{-\rho \circ \Sigma(\Phi)}N(\Phi)(tf)\big).
$$
Hence
\begin{equation*}
\frac{d}{dt}_{|t=0}\hat{\mathcal{H}}(\rho+tf,\Phi) =-(f \circ \Sigma(\Phi)) \hat{\mathcal{H}}(\rho,\Phi)+(n-1)e^{-\rho \circ \Sigma(\Phi)}N(\Phi)(f).
\end{equation*}
Since $$
\mathcal{H}(\rho,\Phi)(x,v) =\hat{\mathcal{H}}(\rho,\Phi)(x,v) e^{n\rho(\Sigma_v(\Phi)(x))} \cos^{n-1}(\Phi(x,v)),
$$
we have
\begin{multline*}
 \frac{d}{dt}_{|t=0}\mathcal{H}(\rho+tf,\Phi) =\frac{d}{dt}_{|t=0}\hat{\mathcal{H}}(\rho+tf,\Phi) e^{n\rho(\Sigma_v(\Phi)(x))} \cos^{n-1}(\Phi(x,v))\\
+ nf(\Sigma_v(\Phi)(x)) \hat{\mathcal{H}}(\rho,\Phi)(x,v) e^{n\rho(\Sigma_v(\Phi)(x))} \cos^{n-1}(\Phi(x,v)).
\end{multline*}
Therefore
\begin{multline*}
 \frac{d}{dt}_{|t=0}\mathcal{H}(\rho+tf,\Phi) =(n -1) (f \circ \Sigma(\Phi)) \mathcal{H}(\rho,\Phi) \\
+(n-1)  N(\Phi)(f)e^{(n-1)\rho(\Sigma_v(\Phi)(x))} \cos^{n-1}(\Phi(x,v)).
\end{multline*}

The fact that $\mathcal{J}_v(\rho,\Phi)$ is a second-order linear differential operator follows from  equation (\ref{eqH}). 
Symmetry of $\mathcal{J}_v(\rho,\Phi)$ follows from equation (\ref{eqD2A}), since
\begin{equation*}
	\frac{\partial }{\partial t}_{\vert t=0}\frac{\partial}{\partial s}_{\vert s=0}\mathcal{A}(\rho,\Phi+t\phi+s\psi)(v) = \int_{\Sigma_v}\psi_v(x)\mathcal{J}_v(\rho,\Phi)\cdot \phi_v(x) dA_{can}(x)
\end{equation*}
and partial derivatives commute. (See also \cite{Whi}, Proposition 1.1). The identity $\mathcal{J}(\rho,\Phi)_{-v}=\mathcal{J}(\rho,\Phi)_v$
follows from 
$$
\mathcal{A}(\rho,\Phi+t\phi+s\psi)(-v)=\mathcal{A}(\rho,\Phi+t\phi+s\psi)(v).
$$
\end{proof}

\begin{rmk} \label{rmkJacobicanonico}
When $(\rho,\Phi)=(0,0)$, the operator $\mathcal{J}_v(0,0)$ coincides with the Jacobi operator of the equator $\Sigma_v$ as a minimal hypersurface of $(S^n,can)$, that is
\begin{equation*}
	 \mathcal{J}_v(0,0) = -\Delta_{(\Sigma_v,can)} - (n-1).
\end{equation*}
In fact the graphical perturbation $t\mapsto \Sigma_{v}(t\phi)$ of $\Sigma_v(0)$ has normal speed $\phi_v$ at $t=0$.  It is well-known that the kernel of the operator $\mathcal{J}(0,0)_v$ consists precisely of the linear functions 
\begin{equation*}
	x \in \Sigma_v \mapsto \langle x,u \rangle \in \mathbb{R} \quad \text{for all} \quad u\in T_{v}S^n.
\end{equation*}
	By  Fredholm alternative and regularity theory, for every 
	 $$\psi\in ker(\mathcal{J}_v(0,0))^{\perp}\cap C^{\infty}(\Sigma_v)$$ (i.e. $L^2$-orthogonal to the linear functions) there exists a unique $\phi\in ker(\mathcal{J}_v(0,0))^{\perp}\cap C^{\infty}(\Sigma_v)$ that solves the equation $\mathcal{J}_v(0,0)\phi=\psi$.
\end{rmk}

\indent Denote by $P_v : L^{2}(\Sigma_v,can)\rightarrow ker(\mathcal{J}_v(0,0))^{\perp}$ the $L^2$-orthogonal projection. The adjoint $P_v^*$ is  the inclusion map.  We denote by $C^{\infty}_{0,odd}(T_{1}S^n)$ the space of $\phi\in  C^{\infty}_{*,odd}(T_{1}S^n)$ such that $\phi_v \in ker(\mathcal{J}_v(0,0))^{\perp}$ for every $v\in S^n$.

Assume the $C^3$ norms of $\rho$ and $\Phi$ are sufficiently small.

\begin{prop}\label{propsoljacobiprojected}
The linear operator $\mathcal{J}_v(\rho,\Phi)$ is elliptic  Fredholm of index zero, and 
{
$$
\mathcal{P}_{v}(\rho,\Phi) : ker(\mathcal{J}_v(0,0))^{\perp}\cap C^{\infty}(\Sigma_v)\rightarrow  ker(\mathcal{J}_v(0,0))^{\perp}\cap C^{\infty}(\Sigma_v)
 $$
 given by $\mathcal{P}_{v}(\rho,\Phi)=P_v\circ \mathcal{J}_v(\rho,\Phi)\circ P_v^*$ 
is bijective.}
 Moreover, for any fixed $\alpha\in (0,1)$ {there is a positive constant $c=c(n,\alpha)$ so that } 
	\begin{equation} \label{eqgard0}
		\norm{\mathcal{P}_v(\rho,\Phi)\psi_v}_{C^{0,\alpha}(\Sigma_v)} \geq c \norm{\psi_v}_{C^{2,\alpha}(\Sigma_v)}
	\end{equation}
	for every $\psi\in C^{\infty}_{0,odd}(T_{1}S^n)$ and every $v\in S^n$. 
	\end{prop}
	
	\begin{proof}

The coefficients of the second-order part of  $\mathcal{J}_v(\rho,\Phi)$
 are close in $C^0$ norm to the coefficients of $\mathcal{J}_v(0,0)$  uniformly in $v\in S^n$.  This implies $\mathcal{J}_v(\rho,\Phi)$
 is elliptic Fredholm and by symmetry it has index zero.

	 When $(\rho,\Phi)=(0,0)$, by Fredholm alternative and regularity theory, 
	 $$
	 \mathcal{J}_v(0,0): ker(\mathcal{J}_v(0,0))^{\perp}\cap C^{2,\alpha}(\Sigma_v) \rightarrow ker(\mathcal{J}_v(0,0))^{\perp}\cap C^{0,\alpha}(\Sigma_v)
	 $$
	 is a continuous bijection. There is a constant $c>0$ depending only on $n$ and $\alpha$ such that
	 $$
	 c\norm{\psi}_{C^{2,\alpha}(\Sigma_v)}  \leq \norm{\mathcal{J}_v(0,0)\psi}_{C^{0,\alpha}(\Sigma_v)}
	 $$
	 for every $\psi \in ker(\mathcal{J}_v(0,0))^{\perp}\cap C^{2,\alpha}(\Sigma_v)$. We have
	  \begin{align*}
		c\norm{\psi}_{C^{2,\alpha}(\Sigma_v)} & \leq \norm{\mathcal{J}_v(0,0)\psi}_{C^{0,\alpha}(\Sigma_v)} \\
		& \leq \norm{\mathcal{P}_v(\rho,\Phi)\psi}_{C^{0,\alpha}(\Sigma_v)} + \norm{\mathcal{P}_v(\rho,\Phi)-\mathcal{P}_v(0,0)}\norm{\psi}_{C^{2,\alpha}(\Sigma_v)}.
		\end{align*}
		The assumptions on $\rho$ and $\Phi$ allow us to absorb the last term into the left-hand side  and prove \eqref{eqgard0} with a different constant.
		Elliptic regularity theory 
	implies that 
	$$
	 \mathcal{P}_{v}(\rho,\Phi): ker(\mathcal{J}_v(0,0))^{\perp}\cap C^{\infty}(\Sigma_v) \rightarrow ker(\mathcal{J}_v(0,0))^{\perp}\cap C^{\infty}(\Sigma_v)
	 $$
	 is a bijection.

	 \end{proof}

\subsection{The center and solution maps}  Let $\Omega_{even}^{1}(S^n)$ denote the set of smooth differential one-forms $\omega$ on $S^n$ ($\omega \in \Omega^1(S^n)$) that are even with respect to the antipodal map $A$, \textit{i.e.}  $A^*\omega=\omega$. The following definition is motivated by Remark \ref{rmkJacobicanonico}.
\begin{definition}
	The center map is the linear map
\begin{equation*}
	C : C^\infty_{*,odd}(T_1S^n) \rightarrow \Omega^{1}_{even}(S^n)
\end{equation*}
that assigns to each $\Psi$ the one-form
\begin{equation*}
	C(\Psi)_{v}(u) = \int_{\Sigma_v} \Psi(x,v)\langle x,u \rangle dA_{can}(x) \quad \text{for all} \quad u\in T_vS^n. 
\end{equation*}
\end{definition}

We call $C(\Psi)$ the center of $\Psi$.

\indent We have that
\begin{equation*}
	Ker \, C = \{\Psi\in C^{\infty}_{*,odd}(T_1S^n):\, C(\Psi)=0\} = C^{\infty}_{0,odd}(T_1S^n).
	\end{equation*}
	
	\indent The space $C^\infty_{0,odd}(T_{1}S^n)$ is complemented in $C^{\infty}_{*,odd}(T_{1}S^n)$ by a set of functions that have a simple description. 

\begin{prop} \label{propdecomposition}
	There is a direct sum decomposition
	\begin{equation*} 
	C_{*,odd}^{\infty}(T_1S^n) = C_{0,odd}^{\infty}(T_1S^n) \oplus j\Omega^{1}_{even}(S^n),
	\end{equation*}
	where $j : \Omega^{1}_{even}(S^n) \rightarrow C^{\infty}_{*,odd}(T_{1}S^n)$ is a right-inverse of the center map that satisfies the following property: for every $\phi\in C^{\infty}_{0,odd}(T_{1}S^n)$ and every $\omega\in \Omega^{1}_{even}(S^n)$,
	\begin{equation*}
		\int_{\Sigma_v}\phi_v(x)(j\omega)_v(x)dA_{can}(x) = 0 \quad \text{for all} \quad v\in S^n.
	\end{equation*}
\end{prop}

\begin{proof}
\indent Let 
\begin{equation*}
	j: \Omega^{1}_{even}(S^n) \rightarrow C^\infty_{*,odd}(T_1S^n)
\end{equation*}
denote the linear map that assigns to each $\omega\in \Omega^{1}_{even}(S^n)$ the smooth function  $j\omega$ on $T_{1}S^n$ defined by
\begin{equation*}
	j\omega(x,v) = \alpha_n \omega_{v}(x) \quad \text{for every $(x,v)\in T_1S^n$},
\end{equation*}
where $\alpha_n$ is a positive dimensional constant defined so that 
\begin{equation*}
	\langle X,Y \rangle = \alpha_n\int_{\Sigma_v} \langle X,x\rangle \langle x, Y\rangle dA_{can}(x) \quad \text{for all} \quad X,Y\in T_{v}S^n.
\end{equation*}
We have
$$
j\omega(x,-v)=\alpha_n \omega_{-v} (x)=-\alpha_n \omega_{-v}(-x)=-\alpha_n(A^*\omega)_v (x),
$$
hence $j\omega$ is odd in the second variable because $\omega$ is even.

\indent By duality, a one-form $\omega$ on $S^n$ corresponds to a tangent vector field $X$ on $S^n$ such that $\omega_{v}(u)=\langle X(v), u\rangle$ for all $u\in T_vS^n$. For every $(x,v)\in T_1S^n$, we have $(j\omega)(x,v)=\alpha_n\omega_v(x)= \alpha_n\langle X(v), x\rangle$. Hence, for every $u\in T_vS^n$,
\begin{equation*}
	C(j\omega)_v(u) = \alpha_n\int_{\Sigma_v} \langle X(v), x\rangle\langle x,u\rangle dA_{can}(x) = \langle X(v), u\rangle = \omega_v(u),
\end{equation*}
that is
\begin{equation*}
	C(j\omega) = \omega \quad \text{for every} \quad \omega \in \Omega^{1}_{even}(S^n). \\
\end{equation*}
\indent  Moreover, for every $\phi \in C^{\infty}_{0,odd}(T_{1}S^n)$ and every $\omega\in \Omega ^1_{even}(S^n)$,
\begin{equation*}
	\int_{\Sigma_v}\phi_v(x)(j\omega)_v(x)dA_{can}(x) = \alpha_n C(\phi)_{v}(X(v))=0 \quad \text{for all} \quad v\in S^n.
\end{equation*}
\indent Since $j$ is a right-inverse of the centre map $C$, we can write 
$$
\phi=\big(\phi-j C(\phi)\big)+j C(\phi)
$$
for $\phi\in C^{\infty}_{*,odd}(T_1S^n)$ and  $C\big(\phi-j C(\phi)\big)=0$,
which induces a decomposition of $C^{\infty}_{*,odd}(T_{1}S^n)$ with the required properties. 
\end{proof}

\indent Notice that $\psi\in C^{\infty}_{0,odd}(T_{1}S^n)$ if and only if $\psi_v\in C^\infty(\Sigma_v)$ lies in $Ker(\mathcal{J}_v(0,0))^{\perp}$ for all $v\in S^n$. Hence, by Remark \ref{rmkJacobicanonico}, there exists a unique $\phi\in C^{\infty}_{0,odd}(T_{1}S^n)$ so that $\mathcal J_v(0,0)(\phi_v)=\psi_v$ for all $v\in S^n$. The smoothness of $\phi$ in the combined variable $(x,v)$ follows by elliptic theory and differentiation of the equation. 

\indent More generally, we consider the map
$$
\mathcal{P}(\rho,\Phi):C^{\infty}_{0,odd}(T_{1}S^n) \rightarrow C^{\infty}_{0,odd}(T_{1}S^n)$$
given by
\begin{equation}\label{eqP}
\mathcal{P}(\rho,\Phi)(\phi)= \frac{d}{dt}_{|t=0} (\mathcal{H}(\rho,\Phi+t\phi)-j C(\mathcal{H}(\rho,\Phi+t\phi))). 
\end{equation}
Proposition \ref{propsoljacobiprojected} implies
\begin{equation*}
	\mathcal{P}(\rho,\Phi)(\phi)(x,v)=\mathcal{P}_v(\rho,\Phi)(\phi_v)(x).
\end{equation*}
 The inverse of $\mathcal{P}(\rho,\Phi)$ will be denoted by  
\begin{equation} \label{eqS}
		\mathcal{S}(\rho,\Phi) : C^{\infty}_{0,odd}(T_{1}S^n) \rightarrow C^{\infty}_{0,odd}(T_{1}S^n),
	\end{equation}
and is called the solution map.
A priori we obtain that $\mathcal{S}(\rho,\Phi)(\psi)$ is smooth in $x$  and continuous in $v$.  Again by   implicit differentiation, using the uniqueness of the solution and elliptic theory, we get that  $\mathcal{S}(\rho,\Phi)(\psi)\in C^{\infty}_{0,odd}(T_{1}S^n)$.

\subsection{The variational constraint}\label{section.var.constraint}

 Given $(x,v)\in T_1S^n$, $u\in T_vS^n$, define
\begin{align} \label{eqeta2}
 \eta(\Phi)(x,v,u) = & -\langle x,u\rangle +D\Phi_{(x,v)}\cdot  (-\langle x,u\rangle v,u) \nonumber \\
 & -\tan\Phi(x,v)\langle \nabla^{\Sigma_v}\Phi_v(x),u\rangle. 
\end{align}
Note that $\eta(\Phi)(x,v,u)$ depends linearly on $u$ and
\begin{equation*}
\eta(\Phi)(x,-v,u)=\eta(\Phi)(x,v,u).
\end{equation*}

\begin{definition} The constraint map
\begin{equation*}
	\mathcal{K} : C^{\infty}_{*,odd}(T_1S^n)\times C^{\infty}_{*,odd}(T_1S^n) \rightarrow \Omega^1_{even}(S^n)
\end{equation*}
is defined by 
\begin{equation*}
	\mathcal{K}(\Phi,\Psi)_{v}(u) = \int_{\Sigma_v} \Psi(x,v)\eta(\Phi)(x,v,u)dA_{can}(x) \quad \text{for all} \quad u\in T_{v}S^n.
\end{equation*}
\end{definition}

The map $\mathcal{K}$ generalizes the center map. In fact by \eqref{eqeta2}, 
\begin{equation*}
	C(\Psi) = - \mathcal{K}(0,\Psi).
\end{equation*}
 
\begin{thm}[Variational constraint] \label{thmconstraint}
	We have
\begin{equation*}
	\mathcal{K}(\Phi,\mathcal{H}(\rho,\Phi)) = d\mathcal{A}(\rho,\Phi).
\end{equation*}
\end{thm}

\begin{proof}
To use identity (\ref{eqD2A}), we will write $\Sigma_{v(t)}(\Phi)$ (where $v(t)$ is a variation of $v$) as a graph over $\Sigma_v$
of the type $x\mapsto \cos h_t(x)x +\sin h_t(x) v$ and differentiate.

Suppose $|u|=1$, $u\in T_vS^n$. Then 
$$
v(t) = \frac{1}{\sqrt{1+t^2}} v + \frac{t}{\sqrt{1+t^2}} u \in S^n
$$
satisfies $v(0)=v$ and $v'(0) = u$. For $x\in \Sigma_v$, the curve
$$
\tilde{x}(t) = \frac{1}{\sqrt{1+t^2\langle x,u\rangle^2}} x - \frac{t \langle x,u\rangle}{\sqrt{1+t^2\langle x,u\rangle^2}}v\in S^n
$$
satisfies $\langle \tilde{x}(t),v(t)\rangle =0$  and so $\tilde{x}(t)\in \Sigma_{v(t)}$. Note that $\tilde{x}(0)=x$ and 
 $$
 \tilde{x}'(0) = -\langle x,u\rangle v.
 $$

Then
$$
y(t)= \cos\Phi(\tilde{x}(t),v(t))\, \tilde{x}(t) + \sin \Phi(\tilde{x}(t),v(t)) \, v(t) \in \Sigma_{v(t)}(\Phi),
$$
with $y(0)=y=\Sigma_v(\Phi)(x)$.

  Now we write
 $$
 y(t)=\cos\Gamma(t, x(t))x(t)+\sin \Gamma(t,x(t))v
 $$
 with $\langle x(t),v\rangle=0$, so $\Sigma_{v(t)}(\Phi)$ is the normal graph over $\Sigma_v$ of $\Gamma(t,\cdot)$.  Hence $\Gamma(0,x)=\Phi(x,v)$.
 
  We want to compute $\eta=\frac{\partial}{\partial t}_{|t=0}\Gamma(0,\cdot)$. Since
 \begin{align*}
 \sin \Gamma(t,x(t)) = & -\cos\Phi(\tilde{x}(t),v(t))\frac{t \langle x,u\rangle}{\sqrt{1+t^2\langle x,u\rangle^2}}\\
  & + \sin \Phi(\tilde{x}(t),v(t))\frac{1}{\sqrt{1+t^2}},
 \end{align*}
 it follows  by differentiating  that
 \begin{multline*}
 \cos \Phi(x,v)\big(\eta(x)+D_2\Gamma_{(0,x)}\cdot x'(0)\big)=-\cos \Phi(x,v)\langle x,u\rangle \\
  +\cos \Phi(x,v)D\Phi_{(x,v)}\cdot  (-\langle x,u\rangle v,u).
 \end{multline*}
 Hence
  \begin{equation*}
 \eta(x)+D_2\Gamma_{(0,x)}\cdot x'(0)=-\langle x,u\rangle +D\Phi_{(x,v)}\cdot  (-\langle x,u\rangle v,u).
 \end{equation*}
 Since $\Gamma(0,x)=\Phi(x,v)$, we have $D_2\Gamma_{(0,x)}\cdot x'(0)=D\Phi_{(x,v)}(x'(0),0)$.
 
  Now we differentiate
 \begin{align*}
 \cos\Gamma(t, x(t))x(t) = & \cos\Phi(\tilde{x}(t),v(t))  \frac{1}{\sqrt{1+t^2\langle x,u\rangle^2}} x\\
 & + \sin \Phi(\tilde{x}(t),v(t))\frac{t}{\sqrt{1+t^2}} u
 \end{align*}
 at $t=0$. Hence
 \begin{multline*}
 -\sin\Phi(x,v)\big(\eta(x)+D_2\Gamma_{(0,x)}\cdot x'(0)\big)x+\cos\Phi(x,v)x'(0)\\
 =-\sin\Phi(x,v)D\Phi_{(x,v)}\cdot  (-\langle x,u\rangle v,u)x+\sin\Phi(x,v)u.
 \end{multline*}
 
  Therefore
  \begin{equation*}
 \sin\Phi(x,v)\big(\langle x,u\rangle x-u\big)+\cos\Phi(x,v)x'(0)=0,
 \end{equation*}
 so
 \begin{equation*}
 x'(0)=\tan\Phi(x,v) \big(-\langle x,u\rangle x+u\big).
 \end{equation*}
 
  We conclude that
 \begin{align*}
  \eta(x) & = -\langle x,u\rangle +D\Phi_{(x,v)}\cdot  (-\langle x,u\rangle v,u)\\
  & \quad \, -\tan\Phi(x,v)D\Phi_{(x,v)}( -\langle x,u\rangle x+u,0)\\
  & =\eta(\Phi)(x,v,u).
 \end{align*}

The result follows from  identity (\ref{eqD2A}), since $\Sigma_{v(t)}(\Phi)$ is the normal graph over $\Sigma_v$ of $\Gamma(t,\cdot)$ and $\frac{\partial}{\partial t}\Gamma(0,\cdot)=\eta(\Phi)(\cdot, v, u)$:
$$
d\mathcal{A}(\rho,\Phi)_v(u)=\int_{\Sigma_v}\mathcal{H}(\rho,\Phi)(x,v)\eta(\Phi)(x,v,u)dA_{can}(x) .
$$

\end{proof}

We have:

\begin{prop} \label{propcriterion}  The following assertions are equivalent:
	\begin{itemize}
		\item[$i)$] The hypersurfaces $\Sigma_{\sigma}(\Phi)$ are minimal in $(S^n,e^{2\rho}can)$, \textit{i.e.}
		\begin{equation*}
			\mathcal{H}(\rho,\Phi)=0.
		\end{equation*}
		\item[$ii)$] The hypersurfaces $\Sigma_{\sigma}(\Phi)$ have the same area in $(S^n,e^{2\rho}can)$ and there exists $\omega\in \Omega^{1}_{even}(S^n)$ such that
		\begin{equation*}
			\mathcal{H}(\rho,\Phi)=j \omega.
		\end{equation*}
	\end{itemize}
\end{prop}

\begin{proof}
	If the hypersurfaces $\Sigma_{\sigma}(\Phi)$ are minimal in $(S^n,e^{2\rho}can)$, Theorem \ref{thmconstraint} implies $\mathcal{A}(\rho,\Phi)$ is constant.  Then i) implies ii) with $\omega=0$. 
	
	Suppose $ii)$. By the Variational Constraint Theorem \ref{thmconstraint},
	\begin{equation*}
		\mathcal{K}(\Phi,j \omega)=\mathcal{K}(\Phi,\mathcal{H}(\rho,\Phi))=d\mathcal{A}(\rho,\Phi)=0.
	\end{equation*}

	Hence, writing $j\omega(x,v)=\alpha_n\langle X(v),x\rangle$,
	\begin{multline*}
	\int_{\Sigma_v} -\alpha_n \langle X(v),x\rangle \langle x,u\rangle dA_{can}(x)=
	\int_{\Sigma_v} j\omega(x,v) \eta(\Phi)(x,v,u)dA_{can}(x)\\
	-\int_{\Sigma_v} \alpha_n \langle X(v),x\rangle \big(\eta(\Phi)(x,v,u)+\langle x,u\rangle \big)dA_{can}(x).
	\end{multline*}

	Choosing $u=X(v)$, we have
	\begin{equation*}
	\langle X(v),X(v)\rangle =\int_{\Sigma_v} \alpha_n \langle X(v),x\rangle \big(\eta(\Phi)(x,v,X(v))+\langle x,X(v)\rangle \big)dA_{can}(x).
	\end{equation*}
	Hence, by (\ref{eqeta2}), we have $|X(v)|^2\leq C\norm{\Phi}_{C^1}|X(v)|^2$. If $\norm{\Phi}_{C^1}$ is sufficiently small, $X=0$ and hence $\mathcal{H}(\rho,\Phi)=j \omega=0$.

\end{proof}


\section{The Funk transform}

 The Funk transform (or Funk-Radon transform) associated to the family $\{\Sigma_{\sigma}(\Phi)\}$ in $(S^n,e^{2\rho}can)$ is the linear map that sends each smooth function $f$ on $S^n$ to the function $\mathcal{F}(\rho,\Phi)(f)$ on ${\mathbb{RP}}^n$ defined by
\begin{multline} \label{eqdeffunk}
	 \mathcal{F}(\rho,\Phi)(f)(\sigma)  = \int_{\Sigma_{\sigma}(\Phi)} f(y)dA_{e^{2\rho}can}(y) \\
						    = \int_{\Sigma_{v}} f(\Sigma_v(\Phi)(x))e^{(n-1)\rho(\Sigma_{v}(\Phi)(x))}|Jac(\Sigma_v(\Phi))|(x)dA_{can}(x)
\end{multline}
for every $\sigma=[v]\in {\mathbb{RP}}^n$. The Funk transform is well-defined as a linear map
\begin{equation*}
	\mathcal{F}(\rho,\Phi) : C^{\infty}(S^n) \mapsto C^{\infty}({\mathbb{RP}}^n).
\end{equation*}

 \indent The transform $\mathcal{F}=\mathcal{F}(0,0)$ maps a smooth function $f$ on $S^n$ to the function 
\begin{equation*}
	\sigma \in {\mathbb{RP}}^n \mapsto \int_{\Sigma_\sigma} f(x)dA_{can}(x)
\end{equation*}
that computes the integral of the restriction of $f$ to equators of $S^n$. %
This functional was used by Funk \cite{Fun} and Guillemin \cite{Gui} in their works on deformations of Zoll metrics for $n=2$.

\indent Then:

\begin{prop} \label{propD_1A}
	\begin{equation*}
		\mathcal{F}(\rho,\Phi)(f)(\sigma) = \frac{1}{n-1}\frac{d}{dt}_{|t=0}\mathcal{A}(\rho+tf,\Phi)(\sigma) \quad \text{for all} \quad \sigma\in \mathbb{RP}^n.
	\end{equation*}
\end{prop}

\begin{proof}
	This follows from identity (\ref{eqformulaA}) by differentiating.
\end{proof}


\section{The Implicit Function Theorem}\label{section.nash.moser} In this paper we will use the Nash-Moser Inverse Function Theorem as stated by Hamilton \cite{Ham}. The following theorem is a  modification of the Implicit Function Theorem with quadratic error (\cite{Ham}, Part III, Theorem 3.3.1).  We refer the reader to that article for the definitions of the terms that appear in the statement. The proof of the next theorem also uses  \cite{Ham2}.

\begin{thm} \label{thmift}
Let $F$ and $H$ be tame Fr\'{e}chet spaces and let $\Lambda$ be a smooth tame map defined on an
open set containing the origin $U\subset F$,
$$
\Lambda:U \subset F \rightarrow H,
$$ 
with $\Lambda(0)=0$.
Suppose  there is a smooth tame map $V(f)h$ linear in $h$
$$
V:U\times H\rightarrow F,
$$
and a smooth tame map $Q(f)\{h,k\}$ bilinear in $h$ and $k$,
$$
Q:U\times H \times H \rightarrow H,
$$
such that for all $f\in U$ and all $h\in H$ we have
$$
D\Lambda(f)V(f)h = h+Q(f)\{\Lambda(f),h\}.
$$
(Notice that $V(f)$ is a right-inverse for $D\Lambda(f)$ for all $f\in \Lambda^{-1}(0)$).
Then there exists a neighborhood $W\subset U$ with $0\in W$ and a smooth tame map
$$
\Gamma:Ker D\Lambda(0)\cap W \rightarrow \Lambda^{-1}(0)
$$
such that 
\begin{equation*}
	\Gamma(0)=0 \quad \text{and} \quad D\Gamma(0)v=v \quad \text{for every} \quad v\in Ker(D\Lambda(0)).
\end{equation*}
\end{thm}

\begin{proof}
	Similarly to the proof of the aforementioned Theorem 3.3.1 in \cite{Ham}, consider the modified map
	\begin{equation*}
		G : (U\subset F) \times H \rightarrow F\times H
	\end{equation*}
	given by
	\begin{equation*}
		G(f,h)= (f - V(f)\Lambda(f),h-\Lambda(f))
	\end{equation*}
	(\cite{Ham},  Part III, Lemma 3.3.2). By construction, the map $G$ is a smooth tame map such that $G(f,h)=(f,h)$ if and only if $\Lambda(f)=0$. In other words, 
	\begin{equation*}
		\Lambda^{-1}(0)\cap U = \pi_F(Fix(G)),
	\end{equation*}
	where $\pi_F : F\times H \rightarrow F$ is the standard projection and $Fix(G)$ denotes the set of fixed points of $G$. 

	We compute
\begin{align*}
DG(f,h)\cdot (g,k) & = (g-V(f)[D\Lambda(f)\cdot g] - [DV(f)\cdot g]\Lambda(f),\\
& \quad \quad \quad k-D\Lambda(f)\cdot g).
\end{align*}

Hence
\begin{align*}
& DG(f,h)\cdot (G(f,h)-(f,h))=DG(f,h)\cdot (-V(f)\Lambda(f),-\Lambda(f)) &\\
& \quad  =(-V(f)\Lambda(f)+V(f)[D\Lambda(f)\cdot V(f)\Lambda(f)] &\\
& \quad \quad \, + [DV(f)\cdot V(f)\Lambda(f)]\Lambda(f), -\Lambda(f)+D\Lambda(f)\cdot V(f)\Lambda(f))& \\
& \quad =(-V(f)\Lambda(f)+V(f)[\Lambda(f)+Q(f)\{\Lambda(f),\Lambda(f)\}] &\\
& \quad \quad \,  + [DV(f)\cdot V(f)\Lambda(f)]\Lambda(f), -\Lambda(f)+\Lambda(f)+Q(f)\{\Lambda(f),\Lambda(f)\}) & \\
& \quad =(V(f)Q(f)\{\Lambda(f),\Lambda(f)\}+ [DV(f)\cdot V(f)\Lambda(f)]\Lambda(f), &\\
&\quad \quad \quad \quad  Q(f)\{\Lambda(f),\Lambda(f)\}).
\end{align*}
		
\indent Therefore, for every $x=(f,g)\in (U\subset F) \times H$,  
	\begin{equation*}
		DG(x)\cdot (G(x) - x) + \widetilde{Q}(x)\cdot \{G(x)-x,G(x)-x\} = 0 
	\end{equation*}
	where 
	\begin{equation*}
		 \widetilde{Q} : (U\subset F) \times H \times (F\times H) \times (F\times H) \rightarrow F\times H
	\end{equation*}
	is the map given by
	\begin{multline*}
		\widetilde{Q}(f,h)\cdot\{(\widetilde{a},\widetilde{c}),(a,c)\} {=} \\
		-\left([DV(f)\cdot\{\widetilde{a}\}]\cdot \{c\} + V(f)Q(f)\cdot\{\widetilde{c},c\}, Q(f)\cdot\{ \widetilde{c},c \} \right).
	\end{multline*}
	The map $\widetilde{Q}$ is a smooth tame map that is bilinear in $(\widetilde{a},\widetilde{c})$ and $(a,c)$.
	
	\indent Therefore the map $G$ is a near-projection in the sense of Hamilton (\cite{Ham2}, Section 2).  As a consequence (\cite{Ham2}), there exists a smooth tame map
	\begin{equation*}
		P : \widetilde{W} \subset (F\times H) \rightarrow  F\times H
	\end{equation*}		
	defined on an open neighborhood of the origin ${\widetilde{W}}\subset U\times H$ that is a projection:
	\begin{equation*}
		P\circ P=P,
	\end{equation*}
	and that moreover has the same fixed point set as $G$ in $\widetilde{W}$: $Fix(P)=Fix(G) \cap \widetilde{W}$. \\
	\indent The value of the map $P$ at a point $x$ is defined explicitly by an inductive algorithm (\cite{Ham2}), from which we can extract some  information.  For example, $P(0)=0$ follows as a consequence of the fact that $G(0)=0$ (\cite{Ham2}, page 26).  Furthermore, whenever $DG(0)\cdot u = u$ (so that the pair $(0,u)$ is a fixed point of the tangent map $TG$) we have $DP(0)\cdot u=u$ as well (\cite{Ham2}, Section 2.5, page 38). \\
	\indent Let $W=\{f:(f,0)\in \widetilde{W}\} \subset F$ and  
	\begin{equation*}
		\Gamma : Ker(D\Lambda(0))\cap W \rightarrow F
	\end{equation*}
	be defined by
	\begin{equation*}
		\Gamma(f) = \pi_F (P(f,0)).
	\end{equation*}		
 (Notice that $Ker(D\Lambda(0))$ is a tame direct summand of $F$, and therefore a tame Fr\'echet space itself by \cite{Ham}, Part II, Definition 1.3.1 and Corollary 1.3.3. Take $L:Ker(D\Lambda(0))\rightarrow F$ given by $L(f)=f$, $M:F\rightarrow Ker(D\Lambda(0))$ given by $M(f)=f-V(0)D\Lambda(0)f$ and check that $M\circ L=Id$). \\
	\indent The map $\Gamma$ is a smooth tame map (as a composition of smooth tame maps,  \cite{Ham}, Part II, Theorem 2.1.6) and its image lies in $\Lambda^{-1}(0)$, because $P(f,0)\in Fix(P)=Fix(G)\cap \widetilde{W}$. Since by definition of the near-projection $G$ we have $G(0,0)=(0,0)$ and $DG(0,0)\cdot (v,0) = (v,0)$ for every $v\in Ker(D\Lambda(0))$, by construction of $P$ we have $P(0,0)=(0,0)$ and $DP(0,0)\cdot (v,0)=(v,0)$. Therefore $\Gamma(0)=0$ and
	\begin{equation*}
		D\Gamma(0)\cdot v = \pi_F(v,0)=v \quad \text{for all} \quad v\in Ker(D\Lambda(0)),
	\end{equation*}
	as we wanted to prove.
\end{proof}

\indent The following corollary shows how Theorem \ref{thmift} can be useful to prove Theorem A.

\begin{cor} \label{corift}
	For every $v\in ker (D\Lambda(0))$ there exists a smooth one-parameter family $u_t$, $t\in (-\delta,\delta)$, such that $u_0=0$, $\dot{u}_0=v$ and
	\begin{equation*}
		u_t \in \Lambda^{-1}(0) \quad \text{for all} \quad t\in (-\delta,\delta).
	\end{equation*}
\end{cor}
\begin{proof}
	Take $u_t = \Gamma(tv)$, choosing $\delta>0$ such  that $tv\in W$ for all $t\in (-\delta,\delta)$.
\end{proof}

\subsection{The  suitable map}\label{problem.formulation}

 Let
\begin{equation*}
	F = C^{\infty}(S^n)\times C^{\infty}_{0,odd}(T_{1}S^n) \quad \text{and} \quad H =  C^{\infty}_{0}({\mathbb{RP}}^n)\times C^{\infty}_{0,odd}(T_{1}S^n),
\end{equation*}
where $C^{\infty}_{0}({\mathbb{RP}}^n)$ is the set of zero average smooth functions on ${\mathbb{RP}}^n$. \\

\indent Assume $(\rho,\Phi) \in U$, where $U$ is a neighborhood of the origin so that a right-inverse $\mathcal{R}(\rho,\Phi)$  of the Funk transform can be constructed. We will show later that $U$ can be taken to be a $C^{3n+4}$ neighborhood of the origin.\\

\indent Define
\begin{equation*}\label{definition.Lambda}
	\Lambda=(\Lambda_1,\Lambda_2) :  (U\subset F) \rightarrow H
\end{equation*}
by setting
\begin{equation}\label{definition.Lambda1}
	\Lambda_1(\rho,\Phi) = \mathcal{A}(\rho,\Phi) - \dashint_{{\mathbb{RP}}^n} \mathcal{A}(\rho,\Phi)(\sigma)dA_{can}(\sigma)
\end{equation}
and
\begin{equation}\label{definition.Lambda2}
	\Lambda_{2}(\rho,\Phi) = \mathcal{H}(\rho,\Phi) - j C(\mathcal{H}(\rho,\Phi)).
\end{equation}
\indent As a consequence of Proposition \ref{propcriterion},
\begin{equation*}
	\Lambda(\rho,\Phi)=0 \quad \Leftrightarrow \quad \mathcal{H}(\rho,\Phi)=0.
\end{equation*}
 Moreover, $\Lambda(0,0)=(0,0)$ since the  equators are minimal hypersurfaces in $(S^n,can)$. 

\subsection{The approximate right-inverse and the quadratic error}\label{sectionhypotheses} 

To apply Theorem \ref{thmift} to the map $\Lambda$ we will find  a right-inverse for $D\Lambda$ modulo a quadratic error. \\
\indent From \eqref{definition.Lambda1} and Proposition \ref{propD_1A},
\begin{align}\label{d1lambda1}
	& D_1\Lambda_1(\rho,\Phi)\cdot f & \nonumber\\
	&\quad = (n-1)\left( \mathcal{F}(\rho,\Phi)(f)-\dashint_{{\mathbb{RP}}^n}\mathcal{F}(\rho,\Phi)(f)(\sigma)dV_{can}(\sigma) \right), &\nonumber \\
	& D_2\Lambda_1(\rho,\Phi)\cdot \phi & \nonumber \\
	& \quad = D_{2}\mathcal{A}(\rho,\Phi)\cdot \phi-\dashint_{\RP^n}(D_{2}\mathcal{A}(\rho,\Phi)\cdot \phi)(\sigma)dV_{can}(\sigma), &
\end{align}
where, by the definition of the Euler-Lagrange operator $\mathcal{H}(\rho,\Phi)$,
\begin{equation*}
	(D_{2}\mathcal{A}(\rho,\Phi)\cdot \phi)([v]) = \int_{\Sigma_v} \mathcal{H}(\rho,\Phi)(x,v)\phi(x,v) dA_{can}(x).
\end{equation*}
From the second part of Proposition \ref{propdecomposition}, for all $\phi\in  C^{\infty}_{0,odd}(T_{1}S^n)$ and $v\in S^n$ we have
\begin{equation*}
	(D_{2}\mathcal{A}(\rho,\Phi)\cdot \phi)([v]) = \int_{\Sigma_v} \Lambda_2(\rho,\Phi)(x,v)\phi(x,v) dA_{can}(x).
\end{equation*}
Likewise, from \eqref{definition.Lambda2} and \eqref{eqP} we have 
\begin{equation}\label{d1lambda2}
	D_1\Lambda_2(\rho,\Phi)\cdot f
	= D_1\mathcal{H}(\rho,\Phi)\cdot f - j C(D_1\mathcal{H}(\rho,\Phi)\cdot f)
	\end{equation}
and
\begin{equation}\label{d2lambda2}
	D_2\Lambda_2(\rho,\Phi)\cdot \phi = \mathcal P(\rho,\Phi)(\phi).
\end{equation}

  Let $\mathcal S(\rho,\Phi)$ be the solution map  as in \eqref{eqS}. We define
 \begin{equation*}
	V=(V_1,V_2) :  (U\subset F)\times H \rightarrow F
\end{equation*}
by
\begin{equation*}
	V_1(\rho,\Phi)\cdot (b,\psi)=\frac{1}{n-1}\mathcal{R}(\rho,\Phi)(b)
\end{equation*}
and, setting $f=V_1(\rho,\Phi)\cdot (b,\psi)$, 
\begin{equation*}
	V_2(\rho,\Phi)\cdot (b,\psi) = \mathcal{S}(\rho,\Phi)(\psi - (D_1\mathcal{H}(\rho,\Phi)\cdot f - j C(D_1\mathcal{H}(\rho,\Phi)\cdot f)).
\end{equation*}
By construction, $V(\rho,\Phi)$ is a linear map for all $(\rho,\Phi)\in U$.

\indent We now show that $V$ is a right-inverse for $D\Lambda$ modulo a quadratic error.
 Consider the  map
 \begin{equation*}\label{definition.Q}
	Q=(Q_1,Q_2) :  (U\subset F)\times H\times H \rightarrow H,
\end{equation*}
bilinear in $(\widetilde{b},\widetilde{\psi}),(b,\psi)$, defined by
\begin{multline*}
	Q_1(\rho,\Phi)\cdot \{(\widetilde{b},\widetilde{\psi}),(b,\psi)\}([v]) \\ 
	  = \int_{\Sigma_{v}}\widetilde{\psi}(x,v)(V_2(\rho,\Phi)\cdot (b,\psi))(x,v)  dA_{can}(x) \quad \quad \quad \quad \quad \quad \quad \quad \\
	   - \dashint_{{\mathbb{RP}}^n}\left(\int_{\Sigma_{v}}\widetilde{\psi}(x,v)(V_2(\rho,\Phi)\cdot (b,\psi))(x,v)  dA_{can}(x)\right) dV_{can}(\sigma)
\end{multline*}
and
\begin{equation*}
	Q_2(\rho,\Phi)\cdot \{(\widetilde{b},\widetilde{\psi}),(b,\psi)\}=0.
\end{equation*} 

We have, with $V_i=V_i(\rho,\Phi)(b,\psi)$, 
\begin{align*}
& D\Lambda_1(\rho,\Phi)\cdot V(\rho,\Phi)(b,\psi) =D_1\Lambda_1(\rho,\Phi)\cdot V_1  +D_2\Lambda_1(\rho,\Phi)\cdot V_2 & \\
& \quad \quad =(n-1)\left( \mathcal{F}(\rho,\Phi)(V_1)-\dashint_{{\mathbb{RP}}^n}\mathcal{F}(\rho,\Phi)(V_1)(\sigma)dV_{can}(\sigma) \right) & \\
& \quad \quad \quad +D_{2}\mathcal{A}(\rho,\Phi)\cdot V_2-\dashint_{\RP^n}(D_{2}\mathcal{A}(\rho,\Phi)\cdot V_2)(\sigma)dV_{can}(\sigma) & \\
& \quad \quad =b &\\
& \quad \quad \quad +\int_{\Sigma_v} \Lambda_2(\rho,\Phi)(x,v)V_2(x,v) dA_{can}(x) &\\
& \quad \quad \quad -\dashint_{{\mathbb{RP}}^n}\left(\int_{\Sigma_v} \Lambda_2(\rho,\Phi)(x,v)V_2(x,v) dA_{can}(x)\right)dV_{can}(\sigma) & \\
& \quad \quad =b+Q_1(\rho,\Phi)(\Lambda(\rho,\Phi),(b,\psi)). &
\end{align*}

And
\begin{align*}
&D\Lambda_2(\rho,\Phi)\cdot V(\rho,\Phi)(b,\psi) =D_1\Lambda_2(\rho,\Phi)\cdot V_1  +D_2\Lambda_2(\rho,\Phi)\cdot V_2 & \\
&\quad \quad = D_1\mathcal{H}(\rho,\Phi)\cdot V_1 - j C(D_1\mathcal{H}(\rho,\Phi)\cdot V_1)+\mathcal{P}(\rho,\Phi)(V_2) &\\
&\quad \quad =D_1\mathcal{H}(\rho,\Phi)\cdot V_1 - j C(D_1\mathcal{H}(\rho,\Phi)\cdot V_1) & \\
& \quad\quad\quad +\psi - (D_1\mathcal{H}(\rho,\Phi)\cdot V_1 - j C(D_1\mathcal{H}(\rho,\Phi)\cdot V_1)) & \\
& \quad\quad =\psi &\\
& \quad\quad =\psi+Q_2(\rho,\Phi)(\Lambda(\rho,\Phi),(b,\psi)). &
\end{align*}

Therefore we deduce that
\begin{equation*}
	D\Lambda(\rho,\Phi)\cdot V(\rho,\Phi)(b,\psi) = (b,\psi) \\
	+ Q(\rho,\Phi)\cdot\{\Lambda(\rho,\Phi),(b,\psi)\}
\end{equation*}
holds for all functions $(\rho,\Phi)\in U$ and $(b,\psi)\in H$, which is what we wanted.

The tameness and smoothness properties of all maps and spaces involved will be proven later.


\section{On some integral operators}\label{integral.operator}

Let $T_1\RP^n$ denote the unit tangent bundle of $(\RP^n,can)$. Set $$\Omega= T_1\RP^n\times[-\pi/4,\pi/4]/\sim \quad\mbox{where }((\sigma,\theta),t)\sim ((\sigma,-\theta),-t) $$ 
 and let $$\Omega_0=\{[((\sigma,\theta),0)]\in \Omega:(\sigma,\theta)\in T_1\RP^n\}\subset \Omega.$$ 
 If $\Delta$ is the diagonal of $\RP^n\times\RP^n$, we glue $\Omega$ to $\RP^n\times\RP^n\setminus \Delta$ by identifying $[((\sigma,\theta),t)]$ with $(\sigma,exp_{\sigma}(t\theta))$. The resulting smooth compact manifold  is denoted by  $B(\RP^n\times\RP^n)$.

Since $\RP^n\times\RP^n\setminus \Delta$  is an  open and dense set in $B(\RP^n\times\RP^n)$,  smooth functions on $B(\RP^n\times\RP^n)$ are uniquely determined by their restrictions to $\RP^n\times\RP^n\setminus \Delta$.

The distance in $S^n$ or $\RP^n$ computed with respect to the canonical metric is denoted by $d$. Associated to any $k\in C^{\infty}(B(\RP^n\times\RP^n))$ there exists a uniquely defined kernel
\begin{equation*}
K\in C^{\infty}(\RP^n\times\RP^n\setminus \Delta), \quad K(\sigma,\tau)= \frac{k(\sigma,\tau)}{\eta(d(\sigma,\tau))},
\end{equation*}
where $\eta:\mathbb{R}\rightarrow \mathbb{R}$ is a nondecreasing smooth function with  $\eta(t)=t$ for $t\leq \frac{\pi}{5}$ and $\eta(t)=\frac{\pi}{4}$ for $t\geq \frac{2\pi}{5}$.
We will consider the operator that associates to every $f\in C^{\infty}(\RP^n)$ the function
$$L(k)(f)(\sigma)=\int_{\RP^n}K(\sigma,\tau)f(\tau)dV_{can}(\tau),\quad \sigma\in \RP^n.$$
Since for each $\sigma\in \RP^n$ the function $\tau\mapsto K(\sigma,\tau)$ is in $L^1(\RP^n)$,  the function $L(k)(f)$ is well-defined and bounded.

Given a compact manifold $M$ and $s\in \mathbb{R}$, $H^s(M)$ denotes the Sobolev space (see Section 1.3 of \cite{Gilkey}) with norm denoted by $||f||_s$. 

\begin{thm}\label{psdo.main.thm} For every $k\in C^{\infty}(B(\RP^n\times\RP^n))$, $L(k)$ is a pseudo-differential operator of order $1-n$.
 Thus
$$L(k):C^{\infty}(\RP^n)\rightarrow C^{\infty}(\RP^n)$$
and there are  bounded extensions to  Sobolev spaces
$$L(k):H^s(\RP^n)\rightarrow H^{s+n-1}(\RP^n), \quad  s\in \mathbb{R}.$$
Moreover,   there is some positive constant $c_n$ so that 
\begin{equation}\label{bound.pso}
||L(k)(f)||_{n-1}\leq c_n||k||_{C^{3n}}||f||_0\quad\mbox{ for all } f\in H^0(\RP^n).
\end{equation}
Similarly, for $q\in \mathbb{N}$ there is some positive constant $c_{n,q}$ such that
\begin{equation}\label{bound.pso2}
||L(k)(f)||_{n-1+q}\leq c_{n,q}||k||_{C^{3n+q}}||f||_q\quad\mbox{ for all } f\in H^q(\RP^n).
\end{equation}

Finally, suppose that for some $k_0\in C^{\infty}(B(\RP^n\times\RP^n))$ $L(k_0)$ is elliptic  and
 $L(k_0):H^0(\RP^n) \rightarrow H^{n-1}(\RP^n)$ is invertible. Then there exist constants  $c>0$ and $\eta_0>0$ so that  for all  $$k\in C^{\infty}(B(\RP^n\times\RP^n))\quad\mbox{with}\quad ||k-k_0||_{3n}<\eta_0,$$  $L(k)$ is  elliptic, $L(k):H^0(\RP^n) \rightarrow H^{n-1}(\RP^n)$ is invertible, and 
\begin{equation}\label{bound.pso.inverse}
||L(k)^{-1}(g)||_0   \leq c  ||g||_{n-1}\quad\mbox{ for all } g\in H^{n-1}(\RP^n).
\end{equation}
   Also $L(k):H^q(\RP^n) \rightarrow H^{n-1+q}(\RP^n)$ is invertible for $q\geq 0$, and there exist positive constants
   $\eta_{0,q}, c'_{n,q}, c''_{n,q}$,  such that if 
$||k-k_0||_{3n+q}<\eta_{0,q}$, then 
\begin{equation}\label{bound.pso.inverse.q}
||L(k)^{-1}(g)||_q    \leq c'_{n,q} ||g||_{n-1+q}\quad\mbox{ for all } g\in H^{n-1+q}(\RP^n).
\end{equation}
 If $||k'-k_0||_{3n+q}<\eta_{0,q}$, then
\begin{equation}\label{bound.pso.inverse.continuity}
||(L(k)^{-1}-L(k')^{-1})(g)||_q \leq c''_{n,q} ||k-k'||_{C^{3n+q}} ||g||_{n-1+q}
\end{equation}
for all  $g\in H^{n-1+q}(\RP^n)$.
\end{thm} 

\begin{proof}	\indent Let $\chi$ be a smooth cut-off function that vanishes outside $[0,\pi/4)$ and is equal to one in $[0,\pi/5)$.\\
	\indent Write the kernel $K$ as the sum of
	\begin{equation*}
		K_1(\sigma,\tau)=\chi(d(\sigma,\tau))K(\sigma,\tau) \quad \text{and} \quad K_2(\sigma,\tau)=(1-\chi(d(\sigma,\tau)))K(\sigma,\tau).
	\end{equation*}	
	Since $K_2$ vanishes near $\Delta$ and is therefore smooth on all of ${\mathbb{RP}}^n\times {\mathbb{RP}}^n$, the integral operator associated to $K_2$ is a smoothing operator. By differentiating we get that the operator  
	\begin{equation*}
		L_2(f)(\sigma) = \int_{{\mathbb{RP}}^n}K_2(\sigma,\tau)f(\tau)dV_{can}(\tau)
	\end{equation*}	
	satisfies
	\begin{equation*}
||L_2(f)||_{n-1}\leq c_n'||k||_{C^{n-1}}||f||_0 
\end{equation*}
for $f\in H^0(\RP^n)$.
	Thus, in order to show that $L(k)$ is a pseudo-differential operator satisfying \eqref{bound.pso}, it suffices to consider the operator $L_1$  defined by 
	\begin{equation*}
		L_1(f)(\sigma) = \int_{{\mathbb{RP}}^n}K_1(\sigma,\tau)f(\tau)dV_{can}(\tau).
	\end{equation*}
	
	Let $\{\eta_i\}_i$ be a partition of unity subordinated to a finite covering of $\mathbb{RP}^n$ by geodesic balls of some radius 
	so that each pair of balls is either disjoint or contained in a ball of radius $\pi/25$. We have $L_1=\sum_{i,j}\eta_iL_1\eta_j$.
	For those $i,j$ corresponding to disjoint balls we have that
	$$
	(\eta_iL_1\eta_j)(f)(\sigma)=\int_{\mathbb{RP}^n}\eta_i(\sigma) K_1(\sigma,\tau)\eta_j(\tau)f(\tau)dV_{can}(\tau)
	$$
	is a smoothing operator with kernel  $(\sigma,\tau) \mapsto \eta_i(\sigma) K_1(\sigma,\tau)\eta_j(\tau)$. As in the case of $L_2$ one has
	\begin{equation*}
||(\eta_iL_1\eta_j)(f)||_{n-1}\leq c_n''||k||_{C^{n-1}}||f||_0 
\end{equation*}
for $f\in H^0(\RP^n)$ for some positive $c_n''$.

For those $i,j$ corresponding to intersecting balls we have that there is a geodesic ball of radius $\pi/25$ containing 
the supports of $\eta_i,\eta_j$. Hence we consider the following situation. 
	
	\indent Let $\phi: \mathbb{R}^n \rightarrow {\mathbb{RP}}^n$ be the exponential map based on a chosen but otherwise arbitrary point $\tilde{\sigma}$ and {set $U\subset \R^n$ and $V\subset \R^n$ to be, respectively, the open balls of radius $\pi/14$ and $\pi/25$ centered at the origin. The map $\phi$ restricted to $U$ is a diffeomorphism onto its image.} 
	Given $\eta,\eta'\in C^\infty_c(\phi(V))$, a calculation in polar coordinates in this coordinate system  {shows that for every $h\in C_c^{\infty}(V)$ we have}:
	\begin{equation*}
		(\eta'L_1\eta)(h\circ \phi^{-1})(\sigma) = \eta'(\sigma)\int_{\mathbb{R}^n}K_1(\sigma,\phi(y))\frac{\sin(|y|)^{n-1}}{|y|^{n-1}}\eta(\phi(y))h(y)dy.
	\end{equation*}
	
	\noindent \textbf{Claim}. The function 
	$
		u:U\times (\mathbb{R}^n\setminus\{0\}) \rightarrow  \mathbb{R}
	$
	given by
	 \begin{equation*}\label{kernel.definition}
	 u(x,w)=|w|\eta(\phi(x-w))K_1(\phi(x),\phi(x-w))\left(\frac{\sin|x-w|}{|x-w|}\right)^{n-1}
	 \end{equation*}
	if $|w|<2\pi/5$ and  otherwise vanishing   is smooth. Moreover, the map
	\begin{equation*}
		(x,\theta,t)\in U\times S^{n-1}\times (0,+\infty) \mapsto u(x,t\theta) \in \mathbb{R}
	\end{equation*}
	has a smooth extension to $U\times S^{n-1}\times [0,+\infty)$. \\
	
	Suppose $|w|<2\pi/5$. Then $|x-w|<2\pi/5+\pi/14<\pi/2$. Since $\phi$ is injective in the open ball of radius $\pi/2$ centered at the origin, $\phi(x)=\phi(x-w)$ if and only if $w=0$. Hence
	$$
	(x,w) \mapsto K_1(\phi(x),\phi(x-w))
	$$
	is smooth on $U\times (B_{2\pi/5}(0)\setminus \{0\})$. To prove smoothness on $U \times \mathbb{R}^n\setminus \{0\}$, since $\sin(t)/t=F(t^2)$ for some smooth function $F$, it is enough to show that the expression defining $u(x,w)$ for $|w|<2\pi/5$ vanishes in a neighborhood of $|w|=2\pi/5$.  If $|w|>2\pi/5-\delta$, then $|x-w| \geq |w|-|x|\geq 2\pi/5-\delta -\pi/14>\pi/25$
	if $\delta$ is sufficiently small.  Hence $\eta(\phi(x-w))=0$ so $u(x,w)=0$.

Let $(x,w)\in U \times B_{2\pi/5}(0)$. Then $d(\phi(x),\phi(x-w))<|w|<\pi/2,$ and we can consider the vector $w'(x,w)\in T_{\phi(x)}{\mathbb{RP}}^n$ uniquely defined by the identities
	\begin{equation*}
		\phi(x-w)=\exp_{\phi(x)}(w'(x,w))\mbox{ and }|w'(x,w)|=d(\phi(x),\phi(x-w)).
	\end{equation*}
	The function 
	$$
	(x,\theta,t) \in U\times S^{n-1} \times (-2\pi/5,2\pi/5) \mapsto w'(x,\theta,t)=w'(x,t\theta)
	$$
	is smooth and satisfies $w'(x,\theta,0)=0$.  Hence there exists a smooth map $(x,\theta,t)\mapsto \tilde{w}(x,\theta,t)$
	such that $w'(x,\theta,t)=t\tilde{w}(x,\theta,t)$. Since $d\phi_x$ is an isomorphism, we have that $\tilde{w}(x,\theta,0)\neq 0$ for every $(x,\theta)\in U\times S^{n-1}$. For $t>0$, $|w'(x,\theta,t)|=t|\tilde{w}(x,\theta,t)|$. Hence the map
	$$
	(x,\theta,t) \in U\times S^{n-1} \times (0,2\pi/5) \mapsto \tilde{t}(x,\theta,t)=|w'(x,\theta,t)|
	$$
	extends smoothly to $U\times S^{n-1} \times [0,2\pi/5)$. Since
	\begin{align*}
	u(x,t\theta) & = tK(\phi(x),\phi(x-t\theta))\eta(\phi(x-t\theta))\left(F(|x-t\theta|^2)\right)^{n-1}\\
	& = t\frac{k(\phi(x),\phi(x-t\theta))}{d(\phi(x),\phi(x-t\theta))}\eta(\phi(x-t\theta))\left(F(|x-t\theta|^2)\right)^{n-1}\\
	& = \frac{k(\phi(x),\frac{w'(x,t\theta)}{|w'(x,t\theta)|}, |w'(x,t\theta)|)}{|\tilde{w}(x,\theta,t)|}\eta(\phi(x-t\theta))\left(F(|x-t\theta|^2)\right)^{n-1}\\
	& = \frac{k(\phi(x),\frac{\tilde{w}(x,\theta,t)}{|\tilde{w}(x,\theta,t)|}, \tilde{t}(x,\theta,t))}{|\tilde{w}(x,\theta,t)|}\eta(\phi(x-t\theta))\left(F(|x-t\theta|^2)\right)^{n-1}
	\end{align*}
	for $0<t<\pi/5$,  we can extend smoothly to $t\in [0,\pi/5)$. This finishes the proof of the claim (see also proof of Lemma 2.1 of \cite{Kiy}). 
	
	It follows from the claim that we can apply Lemma 2.3 of \cite{Kiy} to the function $u$ above and conclude that the function
	\begin{equation}\label{symbol.local}
		a(x,\xi) = \chi(4|x|)\int_{\mathbb{R}^n} \frac{u(x,w)}{|w|}e^{-i\langle \xi, w \rangle}dw
	\end{equation}
	satisfies the inequalities  
\begin{equation}\label{symbol.difference1}
|\partial^{\alpha}_{\xi}\partial_x^\beta a(x,\xi)|\leq C||k||_{C^{n+|\alpha|+|\beta|}} (1+|\xi|)^{1-n-|\alpha|},
\end{equation}
for every multi-index $\alpha, \beta \in \N_0^n$ and  every $(x,\xi)\in U\times \R^n\setminus\{0\}$, where $C$ depends only on $|\alpha|,|\beta|$, and $n$ (and the chosen $\chi$). {The specific bound  $C||k||_{C^{n+|\alpha|+|\beta|}}$ in \eqref{symbol.difference1} follows from inspecting the proof.}

Hence $a$ is a symbol of order $1-n$ (see \cite[Section 1.2]{Gilkey} for definition) with $x$ support in $U$.

The Fourier transform (see \cite{Gilkey}) of a function $f$  in the Schwartz space is  defined as
	 $${\bf F}({f})(\xi)=c\cdot \int_{\R^n}f(x)e^{-i\langle \xi, x \rangle}dx,$$
	 for some positive dimensional constant $c$ (which might depend on the line).
	 
	  For $\sigma=\phi(x)$, $x\in V$, we have
	 \begin{multline*}
		(\eta'L_1\eta)(h\circ \phi^{-1})(\phi(x))   \\
		= \eta'(\phi(x))\int_{\mathbb{R}^n}K_1(\phi(x),\phi(y))\frac{\sin(|y|)^{n-1}}{|y|^{n-1}}\eta(\phi(y))h(y)dy\\
		=\eta'(\phi(x))\int_{\mathbb{R}^n}\frac{u(x,x-y)}{|x-y|}h(y)dy. \hspace{4.9cm}
	\end{multline*}
	
	 By using  Fourier's Inversion Formula and Fubini we have  that for all $h\in C_{c}^{\infty}(V)$,
	  \begin{multline*}
		(\eta'L_1\eta)(h\circ \phi^{-1})(\phi(x))\\
		=c\cdot \eta'(\phi(x))\int_{\mathbb{R}^n}\frac{u(x,x-y)}{|x-y|}\left(\int_{\mathbb{R}^n} e^{i\langle \xi, y \rangle}{\bf F}(h)(\xi)d\xi \right)dy \\
		=c\cdot \int_{\mathbb{R}^n} \eta'(\phi(x)) a(x,\xi) e^{i\langle \xi, x \rangle} {\bf F}(h)(\xi)d\xi. \hspace{3.7cm}
		\end{multline*}

Hence 	$\eta'L_1\eta$ is a pseudodifferential operator of order $1-n$ localized in $\phi(V)$, and hence a pseudodifferential operator of order $1-n$ in $\mathbb{RP}^n$.  Moreover, an inspection of the proof of Lemma 1.2.1 of \cite{Gilkey} shows that, if for some constant $C$ we have $$|\partial_x^\beta a(x,\xi)|\leq C(1+|\xi|)^{1-n}$$ for all $(x,\xi)$ and all multi-index $\beta$ with $|\beta|\leq 2n$, then there is a {dimensional constant $d_n$ so that  
$$||(\eta'L_1\eta)(f)||_{n-1}\leq d_nC||f||_0\quad\mbox{for all }f\in H^0(\RP^n).$$}
Thus we deduce \eqref{bound.pso} from the estimate above combined with \eqref{symbol.difference1}. Similarly for \eqref{bound.pso2}.

This shows that $L(k)$ is a pseudo-differential operator of order $1-n$ on $\mathbb{RP}^n$ satisfying \eqref{bound.pso}.
		Hence $L(k)$ maps smooth functions into smooth functions.
  The operator $L(k)$ admits bounded extensions from the Sobolev spaces $H^{s}(\RP^n)$ to $H^{s+n-1}(\RP^n)$ for all $s\in \R$, as proven in  Lemma 1.3.4  of \cite{Gilkey}. 
  
  Let us now suppose that $L(k_0)$ is elliptic and  $L(k_0):H^0(\RP^n) \rightarrow H^{n-1}(\RP^n)$ is invertible.  Then $L_1(k_0)$ is elliptic because $L_2(k_0)$ is smoothing.  Let $\eta_i(\phi(x))a_0^{(i)}(x,\xi)$ denote the symbol of $\eta_i L_1(k_0)\eta_i$ which is computed as in \eqref{symbol.local}.  The estimate \eqref{symbol.difference1} applied to $L(k-k_0)$ implies that for all $(x,\xi)\in U\times \R^n\setminus\{0\}$
$$
|a^{(i)}(x,\xi)-a_0^{(i)}(x,\xi)|\leq C||k-k_0||_{C^{n}} (1+|\xi|)^{1-n}.
$$
Hence, assuming that $||k-k_0||_{C^{n}}<\delta$ for some small $\delta$, the ellipticity of $L_1(k_0)$ implies the ellipticity of $L_1(k)$ and hence the ellipticity of $L(k)$.

The set of invertible operators is open in the operator norm, hence $L(k):H^0(\RP^n) \rightarrow H^{n-1}(\RP^n)$ is invertible if $\delta$ is sufficiently small. 
Since $L(k_0)$ is invertible, there is a positive constant $c$ such that
$$||L(k_0)(f)||_{n-1}\geq 2c||f||_0\quad\mbox{for all }f\in H^0(\RP^n).$$
By \eqref{bound.pso}, if $||k-k_0||_{C^{3n}}<c/c_n$ then
\begin{multline*}
2c||f||_0\leq ||L(k_0)(f)||_{n-1}\leq ||L(k-k_0)(f)||_{n-1}+||L(k)(f)||_{n-1}\\
\leq c||f||_0+||L(k)(f)||_{n-1}.
\end{multline*}
 This shows \eqref{bound.pso.inverse}.
 
 Since $L(k):H^0(\mathbb{RP}^n)\rightarrow H^{n-1}(\mathbb{RP}^n)$ is invertible, its index as a Fredholm operator vanishes.
 But the index of $L(k):H^q(\mathbb{RP}^n)\rightarrow H^{n-1+q}(\mathbb{RP}^n)$ does not depend on $q$ (Lemma 1.4.5 of \cite{Gilkey}). The fact that $L(k)$ is injective on $H^q(\mathbb{RP}^n)$, $q\geq 0$, implies  the invertibility of $L(k):H^q(\mathbb{RP}^n)\rightarrow H^{n-1+q}(\mathbb{RP}^n)$. The inequality \eqref{bound.pso.inverse.q} follows as \eqref{bound.pso.inverse}. The inequality
 \eqref{bound.pso.inverse.continuity} follows by using the identity  $$L(k)^{-1}-L(k')^{-1}=L(k)^{-1}\circ L(k'-k)\circ L(k')^{-1},$$  finishing the proof of the theorem.

\end{proof}


\section{The dual of the Funk transform}\label{section.funk.2}

\subsection{The dual family}\label{the.dual.family}  The generalized Gauss map $\mathcal{G}(\Phi)$ defined in Section \ref{section.gauss.map} is a diffeomorphism and so we can consider
\begin{equation*}
	\mathcal{G}^{-1}(\Phi) : (q,w) \in T_{1}S^n \mapsto (\Upsilon_{q}(\Phi)(w),\Xi_{q}(\Phi)(w)) \in T_{1}S^n.
\end{equation*}
Since $N_{-v}(\Phi)(p)=-N_{v}(\Phi)(p)$ for every $(p,v)\in T_{1}S^n$,
\begin{equation} \label{eqsymmetriesinverseG}
	\Upsilon_{q}(\Phi)(-w)=\Upsilon_{q}(\Phi)(w)\quad \text{and} \quad \Xi_{q}(\Phi)(-w)=-\Xi_{q}(\Phi)(w)
\end{equation}
for every $(q,w)\in T_{1}S^n$. \\
\indent Clearly, $[\Xi_{q}(0)(w)]=[w]\in {\mathbb{RP}}^n$ for every $w\in \Sigma_q$, or equivalently $[w]\in \Sigma^{*}_{q}$. More generally:

\begin{prop} \label{propgelfand3}
	For $q\in S^n$, the map
	\begin{equation*}
		[\Xi_{q}(\Phi)] : [w] \in \Sigma^*_{q} \mapsto [\Xi_{q}(\Phi)(w)] \in {\mathbb{RP}}^n
	\end{equation*}
	is a well-defined smooth embedding whose image is the dual hypersurface $\Sigma^{*}_q(\Phi)$.
\end{prop}

\begin{proof}
	The map is well-defined as a consequence of \eqref{eqsymmetriesinverseG}.  Let $[w] \in \Sigma^*_{q}$.
	Define $(p,v)=\mathcal{G}^{-1}(\Phi)(q,w)=(\Upsilon_{q}(\Phi)(w),\Xi_{q}(\Phi)(w))$. Then
	$$
	(q,w)=\mathcal{G}(\Phi)(p,v) = (\Sigma(\Phi)(p,v),N(\Phi)(p,v)).
	$$
	Therefore $q\in \Sigma_v(\Phi)$, and hence $[v]\in \Sigma^*_q(\Phi)$. This shows $[\Xi_{q}(\Phi)(w)] \in \Sigma^*_q(\Phi).$
	Similarly one can prove that any element of $\Sigma^*_q(\Phi)$ is of the form $[\Xi_{q}(\Phi)(w)]$ for some $[w] \in \Sigma^*_{q}$.
	
	Since $\mathcal{G}(\Phi)$ is $C^1$ close to the identity map, we have that 
	$\mathcal{G}^{-1}(\Phi)$ is also $C^1$ close to the identity. Hence $[\Xi_{q}(\Phi)]$ is a $C^1$ perturbation of the inclusion map and therefore it is a smooth embedding.

\end{proof}

Similarly the map 
	\begin{equation*}
		\Xi_{q}(\Phi) : w \in \Sigma_q \mapsto \Xi_{q}(\Phi)(w) \in S^n
	\end{equation*}
	is an embedding of the $(n-1)$-dimensional sphere $\Sigma_q$ into $S^n$. Let 
		\begin{equation*}
		N^{*}_q(\Phi) : w\in \Sigma_q  \mapsto N^{*}_q(\Phi)(w) \in T_{\Xi_{q}(\Phi)(w)}S^n
	\end{equation*}
	be the unique unit normal vector field along $\Xi_{q}(\Phi)$ with $\langle N^{*}_q(\Phi)(w),q\rangle > 0$ for all $(q,w)\in T_1S^n$.
	The map $N^*(\Phi)(q,w)=N^{*}_{q}(\Phi)(w)$ is a smooth function of $(q,w)\in T_1S^n$ satisfying  $N^*(\Phi)(q,-w)=N^*(\Phi)(q,w)$.  And $N^{*}_q(0)(w)=q$ for all $w\in \Sigma_q$.
	
\subsection{Intersections of hypersurfaces}\label{intersections}

Since the map $\mathcal G(\Phi)$ is a diffeomorphism,   any two hypersurfaces $\Sigma_\sigma(\Phi)$ and $\Sigma_{\tau}(\Phi)$, $\sigma\neq \tau$,  intersect transversely.  For the purposes of proving tame estimates we need to have explicit parametrizations of the intersections  $\Sigma_\sigma(\Phi)\cap \Sigma_{\tau}(\Phi)$.

For each $v\in S^n$ we first construct $I_v(\Phi)\in C^{\infty}(S^n)$ so that $I_v(\Phi)$ coincides with the {sine of the} signed distance function to $\Sigma_v(\Phi)$ in its neighborhood and such that $I_v(\Phi)^{-1}(0)=\Sigma_v(\Phi)$. The construction needs to depend smoothly on $\Phi$, and some care is required.

Fix $\zeta\in C_{c}^{\infty}(\R)$ an even cutoff function that is one near zero and zero outside the interval $[-\pi/5,\pi/5]$. Consider the smooth map $$D(\Phi):S^n\times S^n\rightarrow S^n$$ so that if $v\in S^n$, $x\in \Sigma_v$, and $|t|\leq 2\pi/5$ then
\begin{equation*}
D(\Phi)(\exp_x(tv),v)=\cos(t)\Sigma_v(\zeta(t)\Phi)(x)+\sin(t)N_v(\zeta(t)\Phi)(x).
\end{equation*}
Notice that $D(\Phi)(\cdot,v)$ is the identity map outside a neighborhood of $\Sigma_v$ (if $|t|\geq \pi/5$) and that 
$D(0)(p,v)=p$ for all $(p,v)\in S^n\times S^n$. We have $D(\Phi)(x,v)=\Sigma_v(\Phi)(x)$ if $x\in \Sigma_v$ (case $t=0$).  Also
\begin{align*}
D(\Phi)(\exp_x(tv),-v) & =D(\Phi)(\exp_x((-t)(-v)),-v)\\
& =D(\Phi)(\exp_x(tv),v)
\end{align*}
for $|t|\leq 2\pi/5$. Hence $D(\Phi)(p,-v)=D(\Phi)(p,v)$ for all $(p,v) \in S^n \times S^n$.

The map $(p,v)\mapsto (D(\Phi)(p,v),v)$ is a diffeomorphism of $S^n \times S^n$ since it is $C^1$ close to the identity map. Hence for every $v\in S^n$, the map ${p}\mapsto D(\Phi)({p},v)$ is a diffeomorphism of $S^n$. {Denote its} inverse by ${p}\mapsto D^{-1}(\Phi)({p},v)$.  We have
$D^{-1}(\Phi)(p,-v)=D^{-1}(\Phi)(p,v)$. 

If $\Phi$ is $C^l$ close to $\Phi'$, then $D(\Phi)$ is $C^{l-1}$ close to $D(\Phi')$.
Consequently, if $\Phi$ is $C^l$ close to $\Phi'$, $l\geq 2$, then $D^{-1}(\Phi)$ is $C^{l-1}$ close to $D^{-1}(\Phi')$.

The function $I_v(\Phi)\in C^{\infty}(S^n)$ is then defined as 
\begin{equation}\label{Iv.map}
I_v(\Phi)({p})=\langle D^{-1}(\Phi)({p},v),v \rangle.
\end{equation}
The map satisfies $I_v(\Phi)=-I_{-v}(\Phi)$ for all $v\in S^n$.  When $\Phi=0$, we simply have $I_v(0)(p)=\langle p,v\rangle$.
If $\Phi$ is $C^l$ close to $\Phi'$, then $I(\Phi)$ is $C^{l-1}$ close to $I(\Phi')$.

Then
$
I_v(\Phi)(p)=0 \Leftrightarrow  D^{-1}(\Phi)({p},v)\in \Sigma_v
\Leftrightarrow p \in \Sigma_v(\Phi).
$
Suppose $y=\cos(t)\Sigma_v(\Phi)(x)+\sin(t)N_v(\Phi)(x)$ with sufficiently small $t$. Then
$y=D(\Phi)(\exp_x(tv),v)$. Therefore
$$
I_v(\Phi)(y)=\langle \exp_x(tv),v\rangle =\sin(t).
$$
 We have proved that $I_v(\Phi)$ has the desired properties.

\indent The $(n-1)$-dimensional spheres $\Sigma_{\sigma}$  and $\Sigma_{{\tau}}$ intersect transversely for all $\sigma\neq \tau \in \RP^n$ and, for all $(v,\theta)\in T_1S^n$, $0< |t|<\pi/2$,
$$\Sigma_{[v]}\cap \Sigma_{[\exp_{v}(t\theta)]}=\{x\in S^n: \langle x,v \rangle =\langle x,\theta \rangle=0\}.$$
Hence the family
\begin{equation}\label{parametrized.circles}
S_q= \Sigma_{\sigma}\cap \Sigma_{\tau}\quad\text{whenever} \quad q=(\sigma,\tau)\in \RP^n\times \RP^n\setminus \Delta,
\end{equation}
extends smoothly to a family $\{S_q\}_{q\in B(\RP^n\times \RP^n)}$ of $(n-2)$-dimensional spheres of $S^n$.

\begin{lem} \label{lemtransversal}
There exists a smooth map
$$J(\Phi):B(\RP^n\times \RP^n)\times S^n\rightarrow S^n$$
that is $C^1$ close to $(q,x)\mapsto x$, with $J(0)(q,x)=x$ and   such that
\begin{itemize}
 \item[$i)$] $x\mapsto J(\Phi)(q)(x)$ is a diffeomorhism for all $q\in B(\RP^n\times \RP^n)$; 
\item[$ii)$] for all $q=(\sigma,\tau)\in \RP^n\times \RP^n\setminus \Delta$ we have
\begin{equation}\label{inversa.mapa.Z}
J(\Phi)(q)(S_q)= \Sigma_{\sigma}(\Phi)\cap \Sigma_{\tau}(\Phi).
\end{equation}

\end{itemize}

 If $\Phi$ is $C^l$ close to $\Phi'$, $l\geq 3$, $J(\Phi)$ is $C^{l-2}$ close to $J(\Phi')$.

\end{lem}

\begin{proof}
	
{ {For $u\in S^n$ with $u\neq  \pm v$, set
$$\theta(v,u)=\frac{u-\langle u,v\rangle v}{\sqrt{1-\langle u,v\rangle^2}}\in S^n.$$
The vector is uniquely determined by the property that $ \langle \theta(v,u), v\rangle =0$ and $\exp_v(d(u,v)\theta(u,v))=u$. 
Consider the map
$$ Z(\Phi):(\RP^n\times\RP^n\setminus \Delta) \times S^n \rightarrow \R^{n+1}$$
where
\begin{multline}\label{mapa.Z}
Z(\Phi)([v],[u])(x)=x-\langle v, x\rangle v-\langle \theta(v,u), x\rangle\theta(v,u) \\
+I_v(\Phi)(x)v+\frac{(I_{u}(\Phi)(x)-\langle u,v\rangle I_v(\Phi)(x))}{\sqrt{1-\langle u,v\rangle^2}}\theta(v,u).
\end{multline}
The map is well-defined because 
\begin{equation*}
	I_{-v}(\Phi)=-I_v(\Phi),\quad \theta(v,-u)=-\theta(v,u), \quad\text{and}\quad\theta(-v,u)=\theta(v,u).
\end{equation*}
\indent We now argue that $Z(\Phi)$ can be extended to a smooth map defined on  $B(\RP^n\times\RP^n)\times S^n$.} Choose a sufficiently small  $\delta>0$. The smooth map on $T_1S^n\times[-\delta,\delta]\times S^n$ given by
$$(v,\theta,t,x)\mapsto I_{\exp_v(t\theta)}(\Phi)(x)-\cos(t)I_v(\Phi)(x)$$
vanishes on $T_1S^n\times\{0\}\times S^n$. Therefore there exists a smooth map{
$$Q(\Phi):T_1S^n\times[-\delta,\delta]\times S^n\rightarrow \R $$}
such that
\begin{equation*}\label{relative.function}{
I_{\exp_v(t\theta)}(\Phi)(x)=\cos(t)I_v(\Phi)(x)+\sin(t)Q(\Phi)(v,\theta,t,x).}
\end{equation*}
{This identity and the fact that $$\theta(v,\exp_v(t\theta))=\frac{t}{|t|}\theta\quad\text{for all} \quad (v,\theta)\in T_1S^n,\quad  0<|t|<\delta,$$ imply 
\begin{multline}\label{mapa.Z2}
Z(\Phi)([v],[\exp_v(t\theta)])(x)=
x-\langle v, x\rangle v-\langle \theta, x\rangle\theta\\
+I_v(\Phi)(x)v+Q(\Phi)(v,\theta,t,x)\theta.
\end{multline}
From this expression we see that $Z(\Phi)$ can be extended smoothly to $B(\RP^n\times\RP^n)\times S^n$. 
 If $\Phi$ is $C^l$ close to $\Phi'$, then $Q(\Phi)$ is $C^{l-2}$ close to $Q(\Phi')$. Hence $Z(\Phi)$  is $C^{l-2}$ close to $Z(\Phi')$.

When $\Phi=0$  we have $Z(0)(q)(x)=x$ for all $(q,x)\in B(\RP^n\times\RP^n)\times S^n$. Hence we obtain that  $Z(\Phi)(q)(x)\neq 0$ for all $(q,x)\in B(\RP^n\times\RP^n)\times S^n$. As a result, the map
\begin{equation*}
\bar Z(\Phi):B(\RP^n\times \RP^n)\times S^n\rightarrow S^n, \quad (q,x)\mapsto {Z(\Phi)(q)(x)}/{|Z(\Phi)(q)(x)|}
\end{equation*}
is well-defined. Since $\bar Z(\Phi)$ is $C^1$ close to $\bar Z(0)$, the map
 $x\mapsto \bar Z(\Phi)(q)(x)$ is a diffeomorphism of $S^n$ for every $q\in B(\RP^n\times\RP^n)$. Thus we obtain a smooth map
 $$J(\Phi):B(\RP^n\times \RP^n)\times S^n\rightarrow S^n, \quad (q,y)\mapsto\bar Z(\Phi)(q)^{-1}(y).$$
  If $\Phi$ is $C^l$ close to $\Phi'$, $l\geq 3$, $J(\Phi)$ is $C^{l-2}$ close to $J(\Phi')$.

 We are left to prove \eqref{inversa.mapa.Z}. If $q=([v],[u])\in \RP^n\times \RP^n\setminus \Delta$, we have that
 $$y\in J(\Phi)(q)(\Sigma_{[v]}\cap \Sigma_{[u]})$$
 if and only if  $\bar Z(\Phi)(q)(y)\in \Sigma_{[v]}\cap \Sigma_{[u]}$, which is equivalent to 
 $$\langle Z(\Phi)(q)(y),v\rangle=0\quad\text{and}\quad\langle Z(\Phi)(q)(y),\theta(v,u)\rangle =0.$$
 From \eqref{mapa.Z} we see that the first identity occurs if and only if $I_v(\Phi)(y)=0$ and the second identity occurs if and only if  $I_u(\Phi)(y)=\langle u,v\rangle I_v(\Phi)(y)$. Hence this is equivalent to $y\in \Sigma_{[v]}(\Phi)\cap\Sigma_{[u]}(\Phi)$.
}
}

\end{proof}

\subsection{The dual of the Funk transform} \label{section.funkdual} { Recall the incidence set defined in Section \ref{section.incidence},
\begin{equation*}
	[F(\Phi)] = \{(p,\sigma)\in S^n\times {\mathbb{RP}}^n : p\in \Sigma_\sigma(\Phi)\},
\end{equation*}
and the projections $\pi_1$ and $\pi_2$ of points in $[F(\Phi)]$ onto the first and second coordinates respectively. We endow $[F(\Phi)]$ with the induced metric $can$ from the product $(S^n\times {\mathbb{RP}}^n,can\times can)$. Thus $\pi_1$ maps $\pi^{-1}_2(\sigma)$ isometrically into $\Sigma_{\sigma}(\Phi)$, and $\pi_{2}$ maps $\pi_{1}^{-1}(p)$ isometrically into $\Sigma^*_{p}(\Phi)$.  Similarly for the orientation-inducing set $F(\Phi)$.

For $(p,v)\in F(\Phi)$, we denote by ${\bf N}^*(\Phi)(p,v)$ the unit normal to $\pi_1^{-1}(p)\subset S^n$ at $v\in S^n$ with 
$\langle {\bf N}^*(\Phi)(p,v),p\rangle>0$. Writing $(p,v)= (q,\Xi_q(\Phi)(w))$, we have that
${\bf N}^*(\Phi)(p,v)=N^*_q(\Phi)(w)$ for the $N^*$ defined earlier. 
It satisfies $\textbf{N}^*(\Phi)(p,-v)=\textbf{N}^*(\Phi)(p,v)$ and constitutes a smooth map ${\bf N^*}(\Phi):F(\Phi)\rightarrow S^n$.

The dual of the Funk transform is the operator $\mathcal F^*(\rho,\Phi)$ such that
$$\int_{\RP^n} \mathcal F(\rho,\Phi)(f)(\sigma)g(\sigma)dV_{can}(\sigma)= \int_{S^n} f(x)\mathcal F^*(\rho,\Phi)(g)(x)dV_{can}(x)$$
 for all $f\in C^\infty(S^n)$, $g\in C^{\infty}(\RP^n)$

\begin{prop} \label{propdualfunk}
	For every $g\in C^{\infty}({\mathbb{RP}}^n)$ and every $p\in S^n$,
	\begin{equation} \label{eqdualfunk}
		\mathcal{F}^{*}(\rho,\Phi)(g)(p) = e^{(n-1)\rho(p)}\int_{\Sigma^{*}_p(\Phi)} g(\tau)U(\Phi)(p,\tau)dA_{can}(\tau),
	\end{equation}
	where $U(\Phi) \in C^{\infty}([F(\Phi)])$ is 
	\begin{equation*}
		U(\Phi)(p,[v]) = \frac{\sqrt{\cos(\Phi(x,v))^{2} + |\nabla^{\Sigma_v}\Phi_v|^2(x)}}{| \eta(\Phi)(x,v,{\bf N}^*(\Phi)(p,v))|\cos(\Phi(x,v))}
	\end{equation*}
	for  $p=\Sigma_{v}(\Phi)(x)$.
\end{prop}

\begin{proof}

\indent According to the co-area formula for Riemannian submersions  (\cite{Chavel}, Chapter III), 
\begin{align*}
	\int_{[F(\Phi)]} F(x,\tau) dV_{can} & = \int_{{\mathbb{RP}}^n}\left( \int_{\Sigma_\tau(\Phi)} \frac{F(x,\tau)}{|Jac(\pi_2)|(x,\tau)} dA_{can}(x)\right) dV_{can}(\tau)  \\
	 & = \int_{S^n}\left(\int_{\Sigma^{*}_x(\Phi)} \frac{F(x,\tau)}{|Jac(\pi_1)|(x,\tau)} dA_{can}({\tau}) \right)dV_{can}({x})
\end{align*}
for every $F\in C^{\infty}([F(\Phi)])$. Here, $|Jac(\pi_i)| = \sqrt{det[D\pi_i\circ (D\pi_{i})^{*}]}$. \\

\indent Thus, for every $f\in C^{\infty}(S^n)$ and every $g\in C^{\infty}({\mathbb{RP}}^n)$,
\begin{multline*}
	\int_{{\mathbb{RP}}^n} \left(\int_{\Sigma_\tau(\Phi)} f(x)e^{(n-1)\rho(x)}dA_{can}(x)\right) g(\tau)dV_{can}(\tau) \\
	= \int_{[F(\Phi)]} e^{(n-1)\rho(x)}f(x)g(\tau) |Jac(\pi_2)|(x,\tau) dV_{can}(x,\tau) \\
	=\int_{S^n} f(x)\left(e^{(n-1)\rho(x)}\int_{\Sigma^{*}_x(\Phi)} g(\tau)\frac{|Jac(\pi_2)|(x,\tau)}{|Jac(\pi_1)|(x,\tau)} dA_{can}(\tau)\right)dV_{can}(x).
\end{multline*}

\indent By the definition of $\mathcal{F}^*(\rho,\Phi)(g)$, this formula proves \eqref{eqdualfunk} with
\begin{equation*}
	U(\Phi)(p,\sigma) = \frac{|Jac(\pi_2)|(p,\sigma)}{|Jac(\pi_1)|(p,\sigma)} \quad \text{for all} \quad (p,\sigma) \in [F(\Phi)].
\end{equation*}
\indent  The tangent space of the $(2n-1)$-dimensional manifold $[F(\Phi)]$ at the point $(p,\sigma)$ admits an orthogonal decomposition
\begin{equation*}
	T_{(p,\sigma)}\pi_2^{-1}(\sigma)\oplus T_{(p,\sigma)}\pi_1^{-1}({p})\oplus span\{V(p,\sigma)\}
\end{equation*}
where both components of $V(p,\sigma)=(V_1(p,\sigma),V_2(p,\sigma))$ are non-zero,
\begin{equation*}
	T_{(p,\sigma)}\pi_2^{-1}(\sigma) = T_{p}\Sigma_\sigma(\Phi)\times \{0\} \quad \text{and} \quad 
	T_{(p,\sigma)}\pi_1^{-1}({p}) = \{0\}\times T_\sigma\Sigma^*_p(\Phi).
\end{equation*}

\indent By definition of $|Jac(\pi_i)|$, we compute
\begin{equation*}
	|Jac(\pi_i)|(p,\sigma) = \frac{|V_i(p,\sigma)|}{\sqrt{|V_1(p,\sigma)|^2+|V_2(p,\sigma)|^2}}
\end{equation*}
so that 
\begin{equation*}
	\frac{|Jac(\pi_2)|(p,\sigma)}{|Jac(\pi_1)|(p,\sigma)} = \frac{|V_{2}(p,\sigma)|}{|V_{1}(p,\sigma)|}.
\end{equation*}

\indent It remains to compute $V_1$ and $V_2$.  Let $\sigma=[v]$ and $p=y=\Sigma_v(\Phi)(x)$. Using the notation of Section \ref{section.var.constraint}, we choose $u\in T_vS^n$, $|u|=1$, and define $v(t),y(t)$. We can choose $u={\bf N}^*(\Phi)(p,v)$. Notice that the projection $[u]$  of $u$ on $T_1\mathbb{RP}^n$  is orthogonal to $T_{[v]}\Sigma_p^*(\Phi)$. Since $y(t)\in \Sigma_{v(t)}(\Phi)$, we have that $(y(t),[v(t)]) \in [F(\Phi)]$ and $(y(0),[v(0)])=(p,\sigma)$. We have $v'(0)=u$ and
\begin{align*}
y'(0) & =\cos\Phi(x,v) x'(0)-\sin\Phi(x,v) (\eta(x) +D_2\Psi(0,x)\cdot x'(0))x\\
&\quad +\cos\Phi(x,v) (\eta(x) +D_2\Psi(0,x)\cdot x'(0)) v\\
&=\sin\Phi(x,v)u-\sin\Phi(x,v)(D\Phi_{(x,v)} (-\langle x,u\rangle v,u))x\\
&\quad +\cos\Phi(x,v) (-\langle x,u\rangle + D\Phi_{(x,v)} (-\langle x,u\rangle v,u)) v.\\
\end{align*}

We have
\begin{align*}
\langle y'(0), N_v(\Phi)(x)\rangle &=Z\langle y'(0), \cos\Phi (-\sin\Phi \,x+\cos\Phi \,v)-\nabla^{\Sigma_v}\Phi_v(x)\rangle \\
&=Z\big(-\cos\Phi \langle u,x\rangle-\sin \Phi \langle u, \nabla^{\Sigma_v}\Phi_v(x)\rangle +\cos\Phi  D\Phi  \big)\\
&=Z \, \cos\Phi \, \eta(\Phi)(x,v,u)
\end{align*}
by  (\ref{eqeta2}), where $\Phi=\Phi(x,v)$,  $Z=(\cos(\Phi)^2 + |\nabla^{\Sigma_v}\Phi_v|^2(x))^{-\frac12}$ and $D\Phi=D\Phi_{(x,v)} (-\langle x,u\rangle v,u)$.

Let $\{e_i\}$ be an orthonormal basis of $T_y\Sigma_v(\Phi)=T_{(p,\sigma)}\pi_2^{-1}(\sigma)$.
 We choose $k_i$ such that
$(y'(0)-\sum k_ie_i,[u])$ is orthogonal to $T_y\Sigma_v(\Phi) \oplus T_{[v]}\Sigma_p^*(\Phi)$, or equivalently such that
$y'(0)-\sum k_ie_i$ is orthogonal to $T_y\Sigma_v(\Phi)$. Then $V_1=y'(0)-\sum k_ie_i$ and $V_2=[u]$. 

We have
\begin{equation*}
|V_1| =| \langle V_1, N_v(\Phi)(x)\rangle |=|\langle y'(0),N_v(\Phi)(x)\rangle|
\end{equation*}
and $|V_2|=1$.

Hence
\begin{equation*}
U(\Phi)(p,[v]) = \frac{1}{|V_1|}=\frac{\sqrt{\cos(\Phi(x,v))^2 + |\nabla^{\Sigma_v}\Phi_v|^2(x)}}{| \eta(\Phi)(x,v,{\bf N}^*(\Phi)(p,v))|\cos\Phi(x,v)}.
\end{equation*}

\end{proof}

}

\subsection{The operator $\mathcal F\circ\mathcal F^*$ as an integral operator}
For $(y,v) \in F(\Phi)$, writing $(y,v)=(\Sigma_v(\Phi)(x),v)$, we define ${\bf N}(\Phi)(y,v)=N_v(\Phi)(x)$. This defines a smooth map
${\bf N}(\Phi):F(\Phi)\rightarrow S^n$.  Hence  the generalized Gauss map can be seen as a map
 \begin{equation*}
	{\bf G}(\Phi): (y,v) \in F(\Phi) \mapsto (y,{\bf N}(\Phi)(y,v)) \in T_{1}S^n.
\end{equation*}
Notice that $\mathcal{G}(\Phi)={\bf G}(\Phi)\circ T(\Phi)$, where $T(\Phi)$ is the map in \eqref{eqparamincidset}. Thus the map ${\bf G}(\Phi)$ must be a diffeomorphism as well.

\indent Consider the smooth function $K(\rho,\Phi)$ defined on $\RP^n\times\RP^n\setminus \Delta$ by
\begin{equation} \label{eqkernel}
		K(\rho,\Phi)([v],[u])=
		 \int_{\scriptstyle{\Sigma_{v}(\Phi)\cap \Sigma_{u}(\Phi)}} {\textstyle\frac{e^{2(n-1)\rho(y)}}{\sqrt{1 - \langle {\bf N}(\Phi)(y,v), {\bf N}(\Phi)(y,u)\rangle^2}}} dA_{can}(y).
\end{equation}
The function is well-defined and smooth on $\RP^n\times\RP^n\setminus \Delta$ because ${\bf G}(\Phi)$ is a diffeomorphism.

\begin{prop} \label{proppropertiesK}
	The  function $$k(\rho,\Phi)(\sigma,\tau)=\eta(d(\sigma,\tau))K(\rho,\Phi)(\sigma,\tau)$$ extends to a smooth function $k(\rho,\Phi)\in C^{\infty}(B(\RP^n\times\RP^n))$, where $\eta$ is as in Section \ref{integral.operator}. If  $(\rho,\Phi)$ is $C^l$ close to $(\rho',\Phi')$,  for $l\geq 3$, 
	 then  $k(\rho,\Phi)$ is $C^{l-3}$ close to $k(\rho',\Phi')$.
\end{prop}

\begin{proof}

We  will refer to the families of spheres $\{S_q\}_{q\in B(\RP^n\times \RP^n)}$ defined in \eqref{parametrized.circles} and to the map $J(\Phi)$ defined in Section \ref{lemtransversal}.  We  take as representatives of $([v],[u])\in \RP^n\times \RP^n\setminus \Delta$ a pair of points $v,u$  in $S^n$ such that $d(u,v)\leq\pi/2$.

 Using \eqref{inversa.mapa.Z} we have that, for all $(\sigma,\tau)\in \RP^n\times \RP^n\setminus \Delta$, 
\begin{align} \label{eqauxdistK}
	k(\rho,\Phi)(\sigma,\tau)& =\eta(d(\sigma,\tau))K(\rho,\Phi)(\sigma,\tau) \nonumber \\
	& = \int_{S_{(\sigma,\tau)}} V(\Phi)(\sigma,\tau)(x)e^{2(n-1)\rho(\sigma,\tau)(x)}dA_{can}(x),
\end{align}
where $\rho(\sigma,\tau)=\rho\circ J(\Phi)(\sigma,\tau)$ and $V(\Phi)(
\sigma,\tau)(x)$ is given by
\begin{equation}\label{V.map}
	{\frac{\eta(d(u,v))}{|{\bf N}(\Phi)(y,u)-{\bf N}(\Phi)(y,v)|}\frac{2|Jac\,J(\Phi)(
\sigma,\tau)_{|S_{(\sigma,\tau)}}(x)| }{\sqrt{4-|{\bf N}(\Phi)(y,v)-{\bf N}(\Phi)(y,u)|^2}}}
\end{equation}
with $\sigma=[v]$, $\tau=[u]$, and $y= J(\Phi)(\sigma,\tau)(x)$. The domain of $V(\Phi)$ is the open set in 
$$\mathcal B=\{(x,q)\in S^n\times B(\RP^n\times \RP^n):x\in S_q\}$$
consisting of those elements $(x,q)$ where $q\in \RP^n\times\RP^n\setminus\Delta$.

We wish to extend $V(\Phi)$ to a smooth function on $\mathcal B$, because if that is the case we have from \eqref{eqauxdistK} that indeed $k(\rho,\Phi)$ extends to a smooth function on $B(\RP^n\times\RP^n)$.

We have that
$$
I_v(\Phi)(y)=\sin d_{\Sigma_v(\Phi)}(y)
$$
for $y$ in some $\delta$-neighborhood of $\Sigma_v(\Phi)$, where $d_{\Sigma_v(\Phi)}$ is the signed distance to the hypersurface $\Sigma_v(\Phi)$. Then
$$
\nabla_{S^n}I_v(\Phi)(y)=\big(\cos d_{\Sigma_v(\Phi)}(y)\big)\nabla_{S^n}d_{\Sigma_v(\Phi)}(y).
$$
Hence, if $y\in \Sigma_v(\Phi)$ (equivalently, if $(y,v)\in F(\Phi)$) then
$$
{\bf N}(\Phi)(y,v)=\nabla_{S^n}I_v(\Phi)(y).
$$

Therefore, choosing $(\sigma,\tau)=([v],[\exp_v(t\theta)])=q$,
\begin{multline}\label{Nphi}
	\frac{2|Jac\,J(\Phi)(
\sigma,\tau)_{|S_{(\sigma,\tau)}}(x)| }{\sqrt{4-|{\bf N}(\Phi)(y,v)-{\bf N}(\Phi)(y,u)|^2}}\\
=	\frac{2|Jac\,J(\Phi)(q)_{|S_{q}}(x)| }{\sqrt{4-|\nabla_{S^n}I_v(\Phi)(y)-\nabla_{S^n}I_{\exp_v(t\theta)}(\Phi)(y)|^2}}
\end{multline}
with $y= J(\Phi)(q)(x)$ and $t>0$ extends smoothly to  $\mathcal B$.

Now define $$h(\Phi)(v,\theta,t,y)=\nabla_{S^n}I_v(\Phi)(y)-\nabla_{S^n}I_{\exp_v(t\theta)}(\Phi)(y)$$ for $(v,\theta,t,y)\in T_1S^n\times \mathbb{R} \times S^n$. Notice that $$h(\Phi)(v,\theta,0,y)=0,$$
hence there exists  a smooth function $\lambda(\Phi)(v,\theta,t,y)$ such that
$$
h(\Phi)(v,\theta,t,y)=t \lambda(\Phi)(v,\theta,t,y).
$$

Therefore, setting $(\sigma,\tau)=([v],[\exp_v(t\theta)])=q$, for sufficiently small $t>0$ 
\begin{eqnarray}\label{lambdaphi}
	\frac{\eta(d(u,v))}{|{\bf N}(\Phi)(y,u)-{\bf N}(\Phi)(y,v)|}=\frac{t}{|h(\Phi)(v,\theta,t,y)|}=\frac{1}{|\lambda(\Phi)(v,\theta,t,y)|}
\end{eqnarray}
with $y= J(\Phi)(q)(x)$ extends smoothly to  $\mathcal B$ provided
$$
\lambda(\Phi)(v,\theta,0,y)\neq 0
$$
for $x\in S_q$, $q=([(v,\theta)],0)$.

We will compute $\lambda(0)(v,\theta,t,y)$. Since we have  $I_v(0)(y)=\langle y,v\rangle$, $\nabla_{S^n}I_v(0)(y)= v-\langle v,y\rangle y$ and $\exp_v(t\theta)=(\cos t) v + (\sin t) \theta$,  we obtain
\begin{align*}
h(0)(v,\theta,t,y)& =v-\langle v,y\rangle y\\
&\quad -\big((\cos t) v + (\sin t) \theta-\langle (\cos t) v + (\sin t) \theta,y\rangle y\big)\\
&= (1-\cos t)v-(\sin t)\theta  -\langle  (1-\cos t)v-(\sin t)\theta ,y\rangle y.
\end{align*}
Therefore
$
\lambda(0)(v,\theta,0,y)=-\theta +\langle \theta,y\rangle y.
$
If $y=J(0)(q)(x)=x\in S_q$ with $q=([(v,\theta)],0)$, then $\langle y,v\rangle=\langle y, \theta\rangle =0$ so
$$
\lambda(0)(v,\theta,0,y)=-\theta.
$$
Then we can assume
$
|\lambda(\Phi)(v,\theta,0,y)|>\frac12
$
for $y= J(\Phi)(q)(x)$ and $x\in S_q$. This proves the existence of the smooth extension. The statement about closeness
in $C^l$ norms follows from the proof.

\end{proof}

{Together with Proposition \ref{proppropertiesK}, the next proposition identifies the operators $\mathcal{F}(\rho,\Phi)\circ \mathcal{F}^{*}(\rho,\Phi)$ as integral operators associated to kernels $k(\rho,\Phi)$ like those considered in Section \ref{integral.operator}.}

\begin{prop}\label{map.is.K}
For every $f\in C^{\infty}({\mathbb{RP}}^n)$, we have	
\begin{equation*}
		(\mathcal{F}(\rho,\Phi)\circ \mathcal{F}^{*}(\rho,\Phi))(f)(\sigma) =\int_{{\mathbb{RP}}^n}K(\rho,\Phi)(\sigma,\tau)f(\tau)dV_{can}(\tau).
	\end{equation*}
	\end{prop}

\begin{proof}

By \eqref{eqdeffunk} and \eqref{eqdualfunk},
\begin{multline*}
	\mathcal{F}(\rho,\Phi)[\mathcal{F}^{*}(\rho,\Phi)(f)](\sigma) = \int_{\Sigma_\sigma(\Phi)}e^{(n-1)\rho(y)}\mathcal{F}^{*}(\rho,\Phi)(f)(y) dA_{can}(y) \\
	= \int_{\Sigma_\sigma(\Phi)}\left(\int_{\Sigma^{*}_y(\Phi)}e^{2(n-1)\rho(y)}f(\tau)U(\Phi)(y,\tau) dA_{can}(\tau)\right) dA_{can}(y)
\end{multline*}
where  $U(\Phi)\in C^{\infty}([F(\Phi)])$ is as in Proposition \ref{propdualfunk}. \\

\indent Fix a point $\sigma\in {\mathbb{RP}}^n$. Consider the set  $[B_\sigma(\Phi)] \subset [F(\Phi)]$ defined by
\begin{equation} \label{eqB1}
		[B_\sigma(\Phi)]= \{(y,\tau)\in S^n\times \mathbb{RP}^n:\, \tau\in \Sigma^*_{y}(\Phi), y\in \Sigma_\sigma(\Phi)\}.
\end{equation}
\noindent Equivalently, 
\begin{equation*}
		[B_\sigma(\Phi)] = \{(y,\tau)\in S^{n}\times \mathbb{RP}^n:\, y\in \Sigma_{\tau}(\Phi)\cap \Sigma_{\sigma}(\Phi)\}.
\end{equation*}

\indent Let $I_v(\Phi)$ {denote the smooth function on $S^n$ that was defined in \eqref{Iv.map}} with $\sigma=[v]$.  Then
\begin{equation*}
	[B_\sigma(\Phi)] =\{(y,\tau)\in [F(\Phi)]:\, I_v(\Phi)(y)=0\}.
\end{equation*}
For $(y,\tau)\in [B_\sigma(\Phi)] $,  write $y=\Sigma_v(\Phi)(x)$.  Since $\pi_1:[F(\Phi)]\rightarrow S^n$ is a submersion,
there exists $X\in T_{(y,\tau)}[F(\Phi)]$ with $(d\pi_1)_{(y,\tau)}(X)=N_v(\Phi)(x)$. Then $d(I_v(\Phi)\circ \pi_1)_{(y,\tau)}(X)=1.$
Hence $0$ is a regular value of the map $(y,\tau)\in [F(\Phi)] \mapsto I_v(\Phi)(y) \in \mathbb{R}$. Thus $[B_\sigma(\Phi)]$ is a smooth embedded hypersurface of $[F(\Phi)]$, hence  of dimension $2n-2$.

From \eqref{eqB1}, we see that the projection  $Q_1=(\pi_1)_{[B_\sigma(\Phi)]}:[B_\sigma(\Phi)] \rightarrow \Sigma_\sigma(\Phi)$ of  $[B_\sigma(\Phi)]$ onto the first factor  is a  surjective submersion  since $Q_1(y,\sigma)=y$ for $y\in \Sigma_\sigma(\Phi)$.  %
By the co-area formula
\begin{multline*} 
	\int_{\Sigma_v(\Phi)}\int_{\Sigma^*_{y}(\Phi)}e^{2(n-1)\rho(y)}f(\tau)U(\Phi)(y,\tau)dA_{can}(\tau)dV_{can}(y) =  
	\\ = \int_{[B_\sigma(\Phi)]} e^{2(n-1)\rho(y)}f(\tau)U(\Phi)(y,\tau)|Jac(Q_1)|(\sigma,y,\tau)dV_{can}(y,\tau),
\end{multline*}
where we write the Jacobian of $Q_1$ with an extra entry so to remind us that this map depends on $\sigma$ (which is fixed in the argument). \\

Let $(y,\tau)\in [B_\sigma(\Phi)]$, $y=\Sigma_v(\Phi)(x)$.  Then
\begin{align*}
T_{(y,\tau)}[B_\sigma(\Phi)] & = \{X\in T_{(y,\tau)}[F(\Phi)]:dI_v(\Phi)d\pi_1X=0\}\\
&=\{X\in T_{(y,\tau)}[F(\Phi)]:d\pi_1X\in T_y\Sigma_v(\Phi)\}.
\end{align*}
As in Section \ref{section.funkdual},
it admits an orthogonal decomposition
\begin{equation*}
	\big((T_{y}\Sigma_\tau(\Phi) \cap T_{y}\Sigma_\sigma(\Phi)) \times \{0\}\big) \oplus \big( \{0\}\times T_\tau\Sigma^*_y(\Phi)\big)\oplus span\{W(y,\tau)\}
\end{equation*}
for $\tau\neq \sigma$, with $W(y,\tau)=(W_1(y,\tau),W_2(y,\tau))$.

Let $\{e_i\}_i$ be an orthonormal basis of $T_{y}\Sigma_\tau(\Phi)$ with  $\langle e_i,{\bf N}(\Phi)(y,v)\rangle =0$ for $i \geq 2$.
Let $\{f_j\}_j$ be an orthonormal basis of $T_\tau\Sigma^*_y(\Phi)$. We write
$$
W(y,\tau)=V(y,\tau)  -\sum_ia_ie_i -\sum b_jf_j.
$$

Recall that 
$$
V_1(y,\tau)= Z \cos\Phi \eta(\Phi)(\tilde x,w,{\bf N}^*(\Phi)(y,w)) N(\Phi)(\tilde{x},w)
$$
and
$$
V_2(y,\tau)=[{\bf N}^*(\Phi)(y,w)],
$$
where $\tau=[w]$, $y=\Sigma_w(\Phi)(\tilde x)$, $\Phi=\Phi(\tilde x,w),$ and $Z=(\cos(\Phi)^2+|\nabla^{\Sigma_w}\Phi_w|^2(\tilde x))^{-\frac12}$.

Then $W(y,\tau)=V(y,\tau)-a_1e_1$ by orthogonality, and
$$
e_1 = \frac{{\bf N}(\Phi)(y,v)-\langle {\bf N}(\Phi)(y,v),{\bf N}(\Phi)(y,w)\rangle {\bf N}(\Phi)(y,w)}
{|{\bf N}(\Phi)(y,v)-\langle {\bf N}(\Phi)(y,v),{\bf N}(\Phi)(y,w)\rangle {\bf N}(\Phi)(y,w)|}.
$$
Hence $W_2(y,\tau)=V_2(y,\tau)$.

We need $\langle W_1(y,\tau), {\bf N}(\Phi)(y,v)\rangle =0$, hence
\begin{equation*}
\langle V_1(y,\tau)-a_1e_1, {\bf N}(\Phi)(y,v)\rangle =0.
\end{equation*}
Therefore
\begin{equation*}
a_1=\lambda\frac{Z \cos\Phi \eta(\Phi)(\tilde x,w,{\bf N}^*(\Phi)(y,w)) \langle {\bf N}(\Phi)(y,w), {\bf N}(\Phi)(y,v)\rangle} {1-\langle {\bf N}(\Phi)(y,v),{\bf N}(\Phi)(y,w)\rangle^2},
\end{equation*}
where $\lambda =|{\bf N}(\Phi)(y,v)-\langle {\bf N}(\Phi)(y,v),{\bf N}(\Phi)(y,w)\rangle {\bf N}(\Phi)(y,w)|$. Hence
\begin{equation*}
a_1=\frac{Z \cos\Phi \eta(\Phi)(\tilde x,w,{\bf N}^*(\Phi)(y,w)) \langle {\bf N}(\Phi)(y,w), {\bf N}(\Phi)(y,v)\rangle}
 {\sqrt{1-\langle {\bf N}(\Phi)(y,v),{\bf N}(\Phi)(y,w)\rangle^2}}.
\end{equation*}

To compute the Jacobian of $Q_1$, 
we choose the orthonormal basis
\begin{equation*}
\{e_i\}_{i\geq 2},   K=\frac{{\bf N}(\Phi)(y,w)-\langle {\bf N}(\Phi)(y,w),{\bf N}(\Phi)(y,v)\rangle {\bf N}(\Phi)(y,v)}{\sqrt{1-\langle {\bf N}(\Phi)(y,w),{\bf N}(\Phi)(y,v)\rangle^2 }}
\end{equation*}
of $T_y\Sigma_v(\Phi)$. Then
\begin{equation*}
|Jac(Q_1)|(\sigma,y,\tau)=\frac{|\langle W_1,K\rangle |}{|W|}.
\end{equation*}

We want to compute also  the Jacobian of $Q_2=(\pi_2)_{[B_\sigma(\Phi)]}:[B_\sigma(\Phi)] \rightarrow \mathbb{RP}^n$.
We decompose $T_\tau\mathbb{RP}^n=T_\tau\Sigma^*_y(\Phi) \oplus span\{[{\bf N}^*(\Phi)(y,w)]\}$. Hence
\begin{equation*}
Jac(Q_2)(\sigma,y,\tau)=\left| \langle \frac{W_2}{|W|}, [{\bf N}^*(\Phi)(y,w)] \rangle \right| =\frac{1}{|W|}.
\end{equation*}

Hence
\begin{align*}
\frac{|Jac(Q_1)|(\sigma,y,\tau)}{|Jac(Q_2)|(\sigma,y,\tau)} & =|\langle W_1,K\rangle|\\
&=|\langle V_1-a_1e_1,K\rangle|\\
&=\frac{|Z \cos\Phi \eta(\Phi)(\tilde x,w,{\bf N}^*(\Phi)(y,w))|}{\sqrt{1-\langle {\bf N}(\Phi)(y,w),{\bf N}(\Phi)(y,v)\rangle^2 }}.
\end{align*}

Since the integrand is a smooth function, and $Q_2^{-1}(\sigma)=\Sigma_\sigma(\Phi)\times \{\sigma\}$ has dimension $n-1$, we have that
\begin{multline*}
 \int_{[B_\sigma(\Phi)]} e^{2(n-1)\rho(y)}f(\tau)U(\Phi)(y,\tau)|Jac(Q_1)|(\sigma,y,\tau)dV_{can}(y,\tau)  \\
   =\lim_{\delta \rightarrow 0} \int_{[B_\sigma(\Phi)]\setminus Q_2^{-1}(B_\delta(\sigma))} e^{2(n-1)\rho(y)}f(\tau)U(\Phi)(y,\tau) \\
   \cdot |Jac(Q_1)|(\sigma,y,\tau)dV_{can}(y,\tau). 
\end{multline*}

Now by the co-area formula
\begin{align*}
&\int_{[B_\sigma(\Phi)]\setminus Q_2^{-1}(B_\delta(\sigma))} e^{2(n-1)\rho(y)}f(\tau)U(\Phi)(y,\tau) & \\
&\hspace{4cm} \cdot |Jac(Q_1)|(\sigma,y,\tau)dV_{can}(y,\tau) & \\
&\quad =\int_{\mathbb{RP}^n\setminus B_\delta(\sigma)} \big(\int_{\Sigma_\sigma(\Phi)\cap \Sigma_\tau(\Phi)} 
 \frac{e^{2(n-1)\rho(y)}f(\tau)}{\sqrt{1-\langle N(\Phi)(y,w),N(\Phi)(y,v)\rangle^2 }}\\
 &\hspace{8cm} dA_{can}(y)\big)dV_{can}(\tau) & \\
 &\quad =\int_{\mathbb{RP}^n\setminus B_\delta(\sigma)}f(\tau)K(\rho,\Phi)(\sigma,\tau) dV_{can}(\tau). &
\end{align*}

Since $|K(\rho,\Phi)(\sigma,\tau)|\leq C d(\sigma,\tau)^{-1}$,  
$\tau \mapsto f(\tau)K(\rho,\Phi)(\sigma,\tau) $ is integrable and
\begin{multline*}
\lim_{\delta \rightarrow 0}\int_{\mathbb{RP}^n\setminus B_\delta(\sigma)}f(\tau)K(\rho,\Phi)(\sigma,\tau) dV_{can}(\tau) \\
=\int_{\mathbb{RP}^n}f(\tau)K(\rho,\Phi)(\sigma,\tau) dV_{can}(\tau),
\end{multline*}
which proves the proposition.

\end{proof}

\subsection{The operator $\mathcal{F}\circ \mathcal{F}^{*}$ as a pseudo-differential operator}\label{adjoint.pdo.section} 
Consider the kernel $k(\rho,\Phi)$ given by \eqref{eqkernel} and  Proposition \ref{proppropertiesK}. From Proposition \ref{map.is.K} we have that 
$$\mathcal{F}(\rho,\Phi)\circ \mathcal{F}^*(\rho,\Phi)(f)=L(k(\rho,\Phi))(f) \quad\mbox{for all } f\in C^{\infty}(\RP^n),$$
where $L(k)$ is as in Section \ref{integral.operator}. and so we obtain from Theorem \ref{psdo.main.thm}  that the operator
	\begin{equation*}
		\mathcal{F}(\rho,\Phi)\circ \mathcal{F}^*(\rho,\Phi) : C^{\infty}({\mathbb{RP}}^n) \rightarrow C^{\infty}({\mathbb{RP}}^n)
	\end{equation*}
	 is a pseudo-differential operator of order $1-n$.

The standard Funk transform $\mathcal{F}=\mathcal{F}(0,0)$ is such that $$L(k(0,0))=\mathcal{F}\circ \mathcal{F}^*$$ is an invertible elliptic pseudo-differential operator of order $1-n$, see Corollary  \ref{Apcorollary}. Let $\eta_0$ be the constant determined by Theorem \ref{psdo.main.thm} with $k_0=k(0,0)$. 

Assume that $(\rho,\Phi)$ has sufficiently small $C^{3n+3}$ norm. Hence we can suppose 
by Proposition \ref{proppropertiesK} that 
$$||k(\rho,\Phi)-k(0,0)||_{C^{3n}}<\eta_0/2.$$
Therefore $\mathcal{F}(\rho,\Phi)\circ \mathcal{F}^*(\rho,\Phi):H^0(\mathbb{RP}^n) \rightarrow H^{n-1}(\mathbb{RP}^n)$ is an invertible  elliptic  pseudo-differential operator of order $1-n$. Thus with these assumptions  we have proved:
\begin{thm}\label{pdo.elliptic.mainthm}
The  operator
$$\mathcal{F}(\rho,\Phi)\circ \mathcal{F}^*(\rho,\Phi):H^0(\mathbb{RP}^n) \rightarrow H^{n-1}(\mathbb{RP}^n)$$ is an invertible elliptic pseudo-differential operator of order $1-n$.
The operator
\begin{equation*} \label{eqR}
	\mathcal{R}(\rho,\Phi) = \mathcal{F}^{*}(\rho,\Phi)\circ (\mathcal{F}(\rho,\Phi)\circ \mathcal{F}^*(\rho,\Phi))^{-1} : C^{\infty}({\mathbb{RP}}^n) \rightarrow C^{\infty}(S^n)
\end{equation*}
is well-defined, and it is a right-inverse for the Funk transform $\mathcal{F}(\rho,\Phi)$:
\begin{equation*}
	\mathcal{F}(\rho,\Phi)\circ \mathcal{R}(\rho,\Phi) = Id. 
\end{equation*}
\end{thm}


\section{Checking tameness}\label{sectiontame}   This section is based on \cite{Ham} and we refer the reader to that paper for the appropriate definitions.

\subsection{Tame estimates for $L(k)$} We start deriving the estimates for the integral operators $L(k)$ defined in Section \ref{integral.operator}. 

\indent Let $V(\RP^n)$ be the space of Killing vector fields  of $(\RP^n,can)$. Given $X\in V(\RP^n)$ and $k\in C^{\infty}(B(\RP^n\times\RP^n))$ there exists $X^B(k)\in  C^{\infty}(B(\RP^n\times\RP^n))$ such that for all $(\sigma,\tau)\in\RP^n\times\RP^n\setminus \Delta $
\begin{equation}\label{commute.formula}
X^B(k)(\sigma,\tau)=\langle X(\tau),\nabla_{\tau} k(\sigma,\tau)\rangle+\langle X(\sigma),\nabla_{\sigma} k(\sigma,\tau)\rangle.
\end{equation}
Indeed, if $X$ is generated by isometries $\{R_s\}_{s}\in PSO(n+1)$ with $R_0=Id$ we define $\phi_s:B(\RP^n\times\RP^n)\rightarrow B(\RP^n\times\RP^n)
$ so that 
\begin{equation*}
	\phi_s(\sigma,\tau)=(R_s(\sigma),R_s(\tau)) \quad \text{if} \quad (\sigma,\tau)\in \RP^n\times\RP^n\setminus\Delta
\end{equation*} and
\begin{equation*}
	\phi_s[(\sigma,\theta),t]=[(R_s(\sigma),{dR_s}_{|\sigma}(\theta)),t)] \quad \text{if} \quad [(\sigma,\theta),t]\in\Omega.
\end{equation*}
The maps $\{\phi_s\}_{s}$ are well-defined smooth diffeomorphisms and hence generate a smooth vector field $X^B$ of $B(\RP^n\times\RP^n)$ so that the derivative $X^B(k)\in  C^{\infty}(B(\RP^n\times\RP^n))$ satisfies \eqref{commute.formula}.

By differentiating, we have that
\begin{align*}
X(L(k)(f))(\sigma)& =\frac{d}{ds}_{|s=0}L(k)(f)(R_s\sigma)\\
& =\frac{d}{ds}_{|s=0}\int_{\mathbb{RP}^n}K(\rho,\Phi)(R_s\sigma,\tau)f(\tau)dV_{can}(\tau)\\
& =\frac{d}{ds}_{|s=0}\int_{\mathbb{RP}^n}\frac{k(\rho,\Phi)(R_s\sigma,R_s\tau)}{\eta(d(\sigma,\tau))}f(R_s\tau)dV_{can}(\tau)\\
& =\int_{\mathbb{RP}^n}\frac{X^B(k)(\sigma,\tau)f(\tau)+k(\sigma,\tau)X(f)(\tau)}{\eta(d(\sigma,\tau))}dV_{can}(\tau).
\end{align*}
Thus
\begin{equation}\label{commute.formula2}
X(L(k)(f))=L(X^B(k))(f)+L(k)(X(f)).
\end{equation}

We are going to use that
\begin{equation}\label{derivative.formula}
||Df||_{q}\leq d_{n,q}\sum_{l=1}^ N||X_l(f)||_q
\end{equation}
for all integers $q\geq 0$, where $\{X_1,\ldots,X_N\}$ is a chosen basis of $V(\RP^n)$ and $d_{n,q}>0$.
	 
The proofs of   Section II.2.2 \cite{Ham} give that for all $f\in  C^{\infty}(\RP^n)$, $k\in C^{\infty}(B(\RP^n\times\RP^n))$, and $p,q,j,l\in \N_0$ with $p\geq j, q\geq l$, we have
\begin{equation}\label{interpolation}
||k||_{C^p}||f||_q\leq c_{n,p,q,j,l}(||k||_{C^{p-j}}||f||_{q+j}+||k||_{C^{p+l}}||f||_{q-l})
\end{equation}
where $c_{n,p,q,j,l}$ is a positive constant depending only on $n,p,q,j,l\in \N_0$.

\begin{prop}\label{bounds.pdo.higher}
For all  $k\in C^{\infty}(B(\RP^n\times\RP^n))$, $f\in  C^{\infty}(\RP^n)$, and $q\in\N_0$ we have
$$||L(k)(f)||_{n-1+q}\leq c_{n,q}(||k||_{C^{3n}}||f||_q+||k||_{C^{3n+q}}||f||_0),$$
where $c_{n,q}$ is a positive constant depending only on $n$ and $q$.
\end{prop}
\begin{proof}

We argue by induction on $q\in\N_0$. The case $q=0$ corresponds to inequality \eqref{bound.pso}.

For each $X\in V(\RP^n)$, using \eqref{commute.formula2} and the induction hypothesis (applied to $X^B(k)$ and $X(f)$) we obtain
\begin{align*}
||X(L(k)(f))||_{n-1+q} & \leq  ||L(X^B(k))(f))||_{n-1+q} +  ||L(k)(X(f))||_{n-1+q}\\
& \leq c_{n,q} ( ||X^B(k)||_{C^{3n}}||f||_q+||X^B(k)||_{C^{3n+q}}||f||_0 \\
& \quad \quad \quad +||k||_{C^{3n}}||X(f)||_{q}+||k||_{C^{3n+q}}||X(f)||_0)\\
& \leq c ( ||k||_{C^{3n+1}}||f||_q+||k||_{C^{3n+q+1}}||f||_0 \\
& \quad \quad \quad +||k||_{C^{3n}}||f||_{q+1}+||k||_{C^{3n+q}}||f||_1).
\end{align*}
By the interpolation inequalities \eqref{interpolation}
$$||k||_{C^{3n+1}}||f||_q+ ||k||_{C^{3n+q}}||f||_1\leq c ( ||k||_{C^{3n}}||f||_{q+1}+||k||_{C^{3n+q+1}}||f||_0).$$
Hence
$$||X(L(k)(f))||_{n-1+q}\leq c(  ||k||_{C^{3n}}||f||_{q+1}+||k||_{C^{3n+q+1}}||f||_0).$$
The desired result follows now from \eqref{derivative.formula}.

\end{proof}

\begin{prop}\label{inverse.tame}
Suppose $k_0\in C^{\infty}(B(\RP^n\times\RP^n))$ is such that $L(k_0)$ is elliptic and $L(k_0):H^0(\mathbb{RP}^n)\rightarrow H^{n-1}(\mathbb{RP}^n)$ is invertible.

There exists $\eta_1>0$ such that for all functions $f\in C^{\infty}(\RP^n)$ and $k\in C^{\infty}(B(\RP^n\times\RP^n))$ with $||k-k_0||_{C^{3n+1}}<\eta_1$ we have
$$||f||_q\leq c_{q}(k_0)(||L(k)(f)||_{n-1+q}+||k||_{C^{3n+q}}||L(k)(f)||_{n-1}),$$
where $c_q(k_0)$ is a positive constant depending only on $q$ and $k_0$.
\end{prop}
	 
\begin{proof}
We argue  by induction.  The case $q=0$ was proven in \eqref{bound.pso.inverse}. 

For each $X\in V(\RP^n)$, the induction hypothesis (applied to  $X(f)$) and   \eqref{commute.formula2} give
\begin{align*}
||X(f)||_q & \leq c_q(k_0)(||L(k)(X(f))||_{n-1+q} +||k||_{C^{3n+q}}||L(k)(X(f))||_{n-1})\\
&\leq c(||L(k)(f)||_{n+q}+||k||_{C^{3n+q}}||L(k)(f)||_{n})\\
&\quad \quad +c_q(k_0)(||L(X^B(k))(f)||_{n-1+q} \\
& \quad \quad \quad  \quad \quad \quad  +||k||_{C^{3n+q}}||L(X^B(k))(f)||_{n-1}).
\end{align*}

By assumption, $||k||_{C^{3n+1}}$ is bounded by a constant depending only on  $k_0$. Thus, from Proposition \ref{bounds.pdo.higher}, the induction hypothesis, and \eqref{bound.pso.inverse}, we can estimate
\begin{align*}
||L(X^B(k))(f)||_{n-1+q} & \leq c_{n,q}(||X^B(k)||_{C^{3n}}||f||_q+||X^B(k)||_{C^{3n+q}}||f||_0)\\
&\leq c(||k||_{C^{3n+1}}||f||_q+||k||_{C^{3n+q+1}}||f||_0)\\
&\leq c ( ||f||_{q}+||k||_{C^{3n+q+1}}||f||_0)\\
&\leq c( ||L(k)(f)||_{n-1+q}+||k||_{C^{3n+q+1}}||L(k)(f)||_{n-1})
\end{align*}
and, likewise,
$||L(X^B(k))(f)||_{n-1}
\leq c ||L(k)(f)||_{n-1}$.

Hence
\begin{multline*}
||X(f)||_q\leq c ( ||L(k)(f)||_{n+q}+||k||_{C^{3n+q}}||L(k)(f)||_{n}\\
+||k||_{C^{3n+q+1}}||L(k)(f)||_{n-1}).
\end{multline*}
By the interpolation inequality and the uniform bound on $||k||_{C^{3n}}$,
 $$||k||_{C^{3n+q}}||L(k)(f)||_{n}\leq c ( ||L(k)(f)||_{n+q}+||k||_{C^{3n+q+1}}||L(k)(f)||_{n-1}).$$ Therefore,
$$||X(f)||_q \leq c (||L(k)(f)||_{n+q}+||k||_{C^{3n+q+1}}||L(k)(f)||_{n-1}).$$
The desired result follows at once from \eqref{derivative.formula}.

	 \end{proof}
	 
	 \subsection{Tame estimates for $\Lambda$ and its right-inverse}

The spaces of smooth functions and smooth one-forms, or more generally smooth sections of a vector  bundle,  on  compact manifolds 
are tame Fr\'echet spaces with their standard gradings, by \cite{Ham}, Part II, Theorem 1.3.6 and Corollary 1.3.9.  We will denote by $C^\infty(X,Y)$ the space of smooth maps from $X$ to $Y$.

\begin{lem}\label{Citame}
	The spaces $C^{\infty}_{*,odd}(T_{1}S^n)\subset C^{\infty}(T_{1}S^n)$ and $\Omega^1_{even}(S^n)\subset \Omega^1(S^n)$ are tame Fr\'echet spaces with the gradings induced by the inclusions. The maps
	\begin{equation*}
		C : C^{\infty}_{*,odd}(T_{1}S^n) \rightarrow \Omega^{1}_{even}(S^n)
	\end{equation*}
	and 
	\begin{equation*}
		j : \Omega^{1}_{even}(S^n) \rightarrow C^{\infty}_{*,odd}(T_{1}S^n)
	\end{equation*} 
	are  tame linear maps.
	     
\end{lem}

\begin{proof}
The space $\Omega^{1}_{even}(S^n)$ is a tame direct summand of $\Omega^{1}(S^n)$ as the map
$$
\omega \in \Omega^{1}(S^n) \mapsto \frac{\omega+A^*\omega}{2} \in \Omega^{1}_{even}(S^n)
$$
(with $A$ the antipodal map) is a tame linear map that is a left-inverse of the inclusion $\Omega^{1}_{even}(S^n) \rightarrow \Omega^{1}(S^n) $. Similarly, $C^{\infty}_{*,odd}(T_{1}S^n)$ is a tame direct summand of $C^{\infty}(T_{1}S^n)$.
 Therefore these spaces are also tame  by \cite{Ham}, Part II, Lemma 1.3.3. 
 
 It follows from the definition that both $C$ and $j$ satisfy a tame estimate of degree and base equal to zero, so they
 are tame linear maps.
\end{proof}

\begin{lem} \label{lemFHtame}
	The spaces 
	\begin{equation*}
	F = C^{\infty}(S^n)\times C^{\infty}_{0,odd}(T_{1}S^n) \quad \text{and} \quad H =  C^{\infty}_{0}({\mathbb{RP}}^n)\times C^{\infty}_{0,odd}(T_{1}S^n)
\end{equation*}
 are tame Fr\'echet spaces with their standard gradings. Let $U\subset F$ be a $C^{3n+4}$-neighborhood of the origin. The maps
	\begin{equation*}
		\mathcal{A}: (\rho,\Phi) \in U\subset F \mapsto \mathcal{A}(\rho,\Phi) \in C^{\infty}({\mathbb{RP}}^n) \quad {and}
	\end{equation*}
	\begin{equation*}
		\mathcal{H}: {(\rho,\Phi) \in U}  \subset F \mapsto {\mathcal{H}(\rho,\Phi)\in  C^{\infty}_{*,odd}(T_1S^n)}
	\end{equation*}
     are smooth tame maps. Thus the map 
     \begin{equation*}
     	\Lambda:(U\subset F)\rightarrow H
     \end{equation*}  
     defined in Section \ref{problem.formulation} is a smooth tame map.
\end{lem}

\begin{proof}
Since $F$ and $H$ are Cartesian products, by \cite{Ham}, Part II, Lemma 1.3.4, it is enough to argue that their factors are tame Fr\'echet spaces. 

	The maps $L:C^{\infty}_0({\mathbb{RP}}^n) \rightarrow C^{\infty}({\mathbb{RP}}^n)$ and $M:C^{\infty}({\mathbb{RP}}^n)\rightarrow C^{\infty}_0({\mathbb{RP}}^n)$  given by $L(f)=f$ and $M(f)=f-\dashint_{{\mathbb{RP}}^n}f(\sigma)dV_{can}(\sigma)$
	are tame linear maps such that $ML$ is the identity. Hence $C^{\infty}_0({\mathbb{RP}}^n)$ is a tame direct summand of 
	$C^{\infty}({\mathbb{RP}}^n)$ and hence  is tame.
	
	Since
	$C: C^{\infty}_{*,odd}(T_1S^n) \rightarrow \Omega^1_{even}(S^n)$
	and
	$j:\Omega^1_{even}(S^n)\rightarrow C^{\infty}_{*,odd}(T_1S^n) $
	are tame linear maps,  the projection
	\begin{equation*}
	\Psi \in C^{\infty}_{*,odd}(T_1S^n) \mapsto \Psi-j C(\Psi) \in C^{\infty}_{0,odd}(T_{1}S^n)
\end{equation*}
is a tame linear map that is a left-inverse for the inclusion  $C^{\infty}_{0,odd}(T_{1}S^n) \rightarrow C^{\infty}_{*,odd}(T_1S^n)$.
Therefore $C^{\infty}_{0,odd}(T_{1}S^n)$ is a tame direct summand of $C^{\infty}_{*,odd}(T_{1}S^n)$, and hence  is tame.

	  Notice that (using the identification $C^{\infty}({\mathbb{RP}}^n)\simeq C^{\infty}_{even}(S^n)$)
\begin{equation*} 
	\mathcal{A}(\rho,\Phi)(v) = \int_{\Sigma_v} \tilde{A}(\rho,\Phi)(x,v)dA_{can}(x) \quad \text{for all} \quad v\in S^n
\end{equation*} 
where
\begin{equation*}
	\tilde{A}:U\subset C^{\infty}(S^n) \times C^{\infty}_{*,odd}(S^n) \rightarrow C^{\infty}(T_{1}S^n)
\end{equation*}
is given in \eqref{eqformulaA}. Hence $\mathcal{A}$ is the composition of $\tilde{A}$ and a tame  linear map. Explicitly,
 \begin{multline*}
 \tilde{A}(\rho,\Phi)(x,v) = e^{(n-1)\rho(\cos(\Phi(x,v))x+\sin(\Phi(x,v))v)}\cos^{n-2}(\Phi(x,v))\\
  \cdot \sqrt{\cos^{2}(\Phi(x,v))+|\nabla^{\Sigma_v}\Phi_v(x)|^2}.
 \end{multline*}
 The map  that sends $\Phi$ to
 $$
 (x,v) \mapsto \cos^{n-2}(\Phi(x,v)) \sqrt{\cos^{2}(\Phi(x,v))+|\nabla^{\Sigma_v}\Phi_v(x)|^2}
 $$
 is a nonlinear differential operator of degree 1, hence smooth tame by \cite{Ham}, Part II, 2.2.7. The map that sends $\Phi$
  to
  $$
 (x,v) \mapsto \cos(\Phi(x,v))x+\sin(\Phi(x,v))v
 $$
 in  $C^{\infty}(T_{1}S^n,S^n) \subset C^{\infty}(T_{1}S^n,\mathbb{R}^{n+1})$
 is also smooth tame (by \cite{Ham}, Part II, Corollary 2.3.2,  $C^{\infty}(T_{1}S^n,S^n)$ is a tame manifold).
 The composition map
 $$
 \mathcal{C}: C^\infty(S^n) \times C^\infty(T_{1}S^n,S^n) \rightarrow C^\infty(T_{1}S^n),
 $$
 $\mathcal{C}(h,\Gamma)=h\circ \Gamma$, is smooth tame by \cite{Ham}, Part II, 2.3.3. Therefore the map that sends 
 $(\rho,\Phi)$ to
 $$
 (x,v)\mapsto e^{(n-1)\rho(\cos(\Phi(x,v))x+\sin(\Phi(x,v))v)}
 $$
 is smooth tame. This proves that  $\tilde{A}$ is smooth tame which implies $\mathcal{A}$ is smooth tame.

The fact $\mathcal{H}$ is a smooth tame map follows similarly as for $\tilde{A}$ using (\ref{eqH}). 
Hence Lemma \ref{Citame} implies $\Lambda$ is smooth tame.
	
\end{proof}

\begin{lem} \label{lemFandF*tame}
{The map
	\begin{equation*}
		\mathcal{F}^{*} : ((\rho,\Phi),g) \in (U\subset F)\times C^{\infty}({\mathbb{RP}}^n) \mapsto \mathcal{F}^*(\rho,\Phi)(g) \in C^{\infty}(S^n)
	\end{equation*}
	is a smooth tame}  {map.} Similarly,
	\begin{equation*}
		\mathcal{F} : ((\rho,\Phi),f) \in (U\subset F)\times C^{\infty}(S^n) \mapsto \mathcal{F}(\rho,\Phi)(f) \in C^{\infty}({\mathbb{RP}}^n)
	\end{equation*}
	is smooth tame.
\end{lem}

\begin{proof}

	For every $g\in C^{\infty}({\mathbb{RP}}^n)$ and every $q\in S^n$,
	\begin{equation*}
		\mathcal{F}^{*}(\rho,\Phi)(g)(q) = e^{(n-1)\rho(q)}\int_{\Sigma^{*}_q(\Phi)} g(\tau)U(\Phi)(q,\tau)dA_{can}(\tau),
	\end{equation*}
	where $U(\Phi) \in C^{\infty}([F(\Phi)])$ is given by
	\begin{equation*}
		U(\Phi)(q,[v]) = \frac{\sqrt{\cos(\Phi(x,v))^{2} + |\nabla^{\Sigma_v}\Phi_v|^2(x)}}{| \eta(\Phi)(x,v,{\bf N}^*(\Phi)(q,v))|\cos(\Phi(x,v))},
	\end{equation*}
	for  $q=\Sigma_{v}(\Phi)(x)$.

The map $\tilde{g}$ that sends $\Phi$ to 
$$
(x,v)\mapsto(\Sigma(\Phi)(x,v),N(\Phi)(x,v))
$$
in $C^\infty(T_1S^n,\mathbb{R}^{2n+2})$
is a nonlinear differential operator of  degree $1$. Hence the generalized Gauss map defined in Section \ref{section.gauss.map}, regarded as a map
\begin{equation*}
	\mathcal{G}: (\rho,\Phi) \in U\subset F\mapsto \mathcal{G}(\Phi)\in \text{Diff}(T_1S^n),
\end{equation*}
is smooth tame.  The inverse map 
\begin{equation*}
	\mathcal{G}^{{-1}}: {(\rho,\Phi) \in U\subset F\mapsto \mathcal{G}^{-1}(\Phi)\in \text{Diff}(T_1S^n)},
\end{equation*}
is also smooth tame  (\cite{Ham}, Part II, Theorem 2.3.5). 

Since
\begin{equation*}
	\mathcal{G}^{-1}(\Phi) : (q,w) \in T_{1}S^n \mapsto (\Upsilon_{q}(\Phi)(w),\Xi_{q}(\Phi)(w)) \in T_{1}S^n,
\end{equation*}
the map  that sends $\Phi$ to
$$
(q,w)\in  T_1S^n \mapsto \Xi_{q}(\Phi)(w) \in S^n
$$
in $C^\infty(T_1S^n,S^n)$
is smooth  tame.  

The map $\mathcal{W}\subset C^\infty(T_1S^n,\mathbb{R}^{n+1})\rightarrow C^\infty(T_1S^n)$ that sends $\Xi$ to
$$
(q,w)\in  T_1S^n \mapsto Jac ({\Xi_q}_{|\Sigma_q})(w),
$$
where $\mathcal{W}$ is a sufficiently small neighborhood of the map $(q,w)\rightarrow w$, is a nonlinear differential operator of
degree 1. Hence it is smooth  tame. 

The map $\mathcal{W} \cap C^\infty(T_1S^n,S^n) \rightarrow C^\infty(T_1S^n,\mathbb{R}^{n+1})$ that sends $\Xi$ to
$$
(q,w) \mapsto  n({\Xi})(q,w) \in T_{\Xi(q,w)}S^n,
$$
where $n(\Xi)(q,w)$ is the  unit normal to $\Xi(q,\Sigma_q)$ at $\Xi(q,w)$ satisfying $\langle n(\Xi)(q,w), q\rangle >0$, is the composition of the inclusion map 
$$\mathcal{W} \cap C^\infty(T_1S^n,S^n)\rightarrow \mathcal{W} \cap C^\infty(T_1S^n,\mathbb{R}^{n+1})
$$ 
and a nonlinear differential operator of degree 1
$$
\mathcal{W} \cap C^\infty(T_1S^n,\mathbb{R}^{n+1}) \rightarrow C^\infty(T_1S^n,\mathbb{R}^{n+1}).
$$
 Hence it is smooth  tame. 

We have that $T_1S^n$ is a compact hypersurface of $S^n\times S^n$. For $\delta>0$ sufficiently small, consider
$V_\delta(T_1S^n)$ the tubular neighborhood of $T_1S^n$ of radius $\delta$ and $\tilde{P}:V_\delta(T_1S^n)\rightarrow T_1S^n$ the nearest point projection.  If $\Phi\in C^\infty_{*,odd}(T_1S^n)$ is in some sufficiently small neighborhood of the origin then $F(\Phi)\subset V_\delta(T_1S^n)$ and $\tilde{P}_{|F(\Phi)}:F(\Phi)\rightarrow T_1S^n$ is a diffeomorphism.  The map that sends 
$\Phi$  to 
$$
\Big((x,v) \mapsto \tilde{P}((\Sigma_v(\Phi)(x),v))\Big) \in \text{Diff}(T_1S^n)
$$
is smooth tame. The same is true for the map
$$
\Big((q,w) \mapsto \tilde{P}((q,\Xi_q(\Phi)(w)))\Big) \in \text{Diff}(T_1S^n).
$$
This implies that  the map that sends such a $\Phi$ to
$$
X(\Phi): \Big((q,w) \mapsto (x=x(\Phi)(q,w),v=v(\Phi)(q,w))\Big)\in \text{Diff}(T_1S^n)
$$
with
$$
(\Sigma_v(\Phi)(x),v))=(q,\Xi_q(\Phi)(w))
$$
is smooth tame. Similarly for $\Phi\mapsto X(\Phi)^{-1}$.

We have that for $v=\Xi_q(\Phi)(w)$
$$
{\bf N}^*(\Phi)(q,v)=n(\Xi(\Phi))(q,w).
$$ 
Hence the map that sends $\Phi$ to
$$
(x,v) \in T_1S^n \mapsto  N^*(\Phi)(\Sigma_v(\Phi)(x),v)=n(\Xi(\Phi))\circ X(\Phi)^{-1} (x,v) \in \mathbb{R}^{n+1}
$$
is smooth  tame.

Let $V$ be the pullback of the vector bundle $TS^n$ under the map $(x,v)\in T_1S^n\mapsto v\in  S^n$. Given $(x,v)\in T_1S^n$, $u\in T_vS^n$, recall \eqref{eqeta2}
\begin{align*} 
\eta(\Phi)(x,v,u) = & -\langle x,u\rangle +D\Phi_{(x,v)}\cdot  (-\langle x,u\rangle v,u)\\
& -\tan\Phi(x,v)\langle \nabla^{\Sigma_v}\Phi_v(x),u\rangle. 
\end{align*} 
Then the map
that sends $\Phi$ to $\eta(\Phi)\in C^\infty(V)$ is smooth tame.
Therefore the map that sends $\Phi$ to 
$$
(x,v) \mapsto  \eta(\Phi)(x,v,N^*(\Phi)(\Sigma_v(\Phi)(x),v))
$$
in $C^\infty(T_1S^n)$ is smooth  tame.

This proves that the map that sends $\Phi$ to
$$
U'(\Phi): (x,v) \mapsto U(\Phi)(\Sigma_v(\Phi)(x),[v])
$$
in $C^\infty(T_1S^n)$ is smooth tame. Hence the map that sends $\Phi$ to $U'(\Phi)\circ X(\Phi)\in C^\infty(T_1S^n)$ given by
$$
U'(\Phi)\circ X(\Phi)(q,w)=U(\Phi)(q,[\Xi_q(\Phi)(w)])
$$
is smooth and tame.

In conclusion, the map that sends $(\Phi,g)$ to the   function in  $C^\infty(T_1S^n)$
$$
\Gamma(q,w)=g([\Xi_q(\Phi)(w)])U(\Phi)(q,[\Xi_q(\Phi)(w)])Jac ({\Xi_q(\Phi)}_{|\Sigma_q})(w),
$$
satisfying $\Gamma(q,w)=\Gamma(q,-w)$, is smooth  tame.

Since
$$
		\mathcal{F}^{*}(\rho,\Phi)(g)(q) = e^{(n-1)\rho(q)}\int_{\Sigma^{*}_q} \tilde{\Gamma}(q,[w])dA_{can}([w]),
	$$
where $\tilde{\Gamma}(q,[w])=\Gamma(q,w)$, we get that $\mathcal{F}^{*}$ is smooth tame.

The fact that $\mathcal{F}$ is smooth tame follows from the same methods and formula (\ref{eqdeffunk}).

\end{proof}

Since $(\rho,\Phi)$ has sufficiently small $C^{3n+4}$ norm, we can apply Proposition \ref{inverse.tame} 
to $k(\rho,\Phi)$ with $k_0=k(0,0)$.

{\begin{lem} \label{lemFF*inversetame}
	The maps
	\begin{multline*}
		\mathcal{L}: ((\rho,\Phi),g) \in (U\subset F)\times C^{\infty}({\mathbb{RP}}^n) \\
		 \mapsto
		 (\mathcal{F}(\rho,\Phi)\circ\mathcal{F}^*(\rho,\Phi))(g) \in C^{\infty}(\mathbb{RP}^n)
	 \end{multline*}
	 and
	 	\begin{multline*}
		\mathcal{V}: ((\rho,\Phi),g) \in (U\subset F)\times C^{\infty}({\mathbb{RP}}^n)\\
		  \mapsto (\mathcal{F}(\rho,\Phi)\circ\mathcal{F}^*(\rho,\Phi))^{-1}(g) \in C^{\infty}(\mathbb{RP}^n)
	 \end{multline*}
	 are smooth tame.
\end{lem}

\begin{proof}
The map $((\rho,\Phi),g)  \mapsto \mathcal{F}(\rho,\Phi)(\mathcal{F}^*(\rho,\Phi)(g))$ is the composition of the map
$((\rho,\Phi),g)  \mapsto ((\rho,\Phi),\mathcal{F}^*(\rho,\Phi)(g))$ and the map $\mathcal{F}$ of Lemma \ref{lemFandF*tame}.
Hence $\mathcal{L}$ is smooth tame.

Recall that $(\mathcal{F}(\rho,\Phi)\circ\mathcal{F}^*(\rho,\Phi))^{-1}(g)=L(k(\rho,\Phi))^{-1}(g)$ for $g\in  C^{\infty}({\mathbb{RP}}^n) $.
Choose $(\rho',\Phi')$, $g'$  and  $q\in \mathbb{N}$. Then we write, using Theorem \ref{psdo.main.thm},
\begin{align*}
&  ||L(k(\rho,\Phi))^{-1}(g)-L(k(\rho',\Phi'))^{-1}(g')||_q & \\
 & \quad \quad \leq  ||L(k(\rho,\Phi))^{-1}(g)-L(k(\rho,\Phi))^{-1}(g')||_q &\\
 & \quad \quad \quad \quad + ||L(k(\rho,\Phi))^{-1}(g')-L(k(\rho',\Phi'))^{-1}(g')||_q & \\
 & \quad \quad \leq c_q (||g-g'||_{n-1+q}+||k(\rho,\Phi)-k(\rho',\Phi')||_{C^{3n+q}}||g'||_{n-1+q}) &
\end{align*}
if 
$||k-k'||_{3n+q}$ is sufficiently small. Hence, if $(\rho,\Phi)$ is $C^{3n+3+q}$ close to $(\rho',\Phi')$ and $g$ is $H^{q+n-1}$-close to
$g'$ then $\mathcal{V}((\rho,\Phi),g)$ is $H^q$ close to $\mathcal{V}((\rho',\Phi'),g')$. By the Sobolev embedding theorems, since $q$ is arbitrary, we have that the map $\mathcal{V}$ is continuous.

We will argue  that the map 
\begin{equation*}
k: (\rho,\Phi) \in U\subset F\mapsto k(\rho,\Phi)\in C^{\infty}(B(\RP^n\times\RP^n))
\end{equation*}
given by Proposition \ref{proppropertiesK} is smooth tame.  Since $D(\Phi)(y,v)$ (as in Section \ref{intersections}) is a nonlinear  function of
$(y,v), \Phi(x,v), \nabla^{\Sigma_v}\Phi_v(x)$ where $y=\exp_x(tv)$, $x\in \Sigma_v$, for $|t|\leq 2\pi/5$, and $D(\Phi)(y,v)=y$ for $|t|\geq \pi/5$,
we have that the map
$$
D: \Phi \in C^\infty_{*,odd}(T_1S^n) \rightarrow D(\Phi) \in C^\infty(S^n\times S^n,S^n)
$$
is smooth tame. Since $(y,v)\rightarrow (D(\Phi)(y,v),v)$ is a diffeomorphism of $S^n\times S^n$ with inverse
$(p,v)\rightarrow (D^{-1}(\Phi)(p,v),v)$, the map $\Phi \rightarrow D^{-1}(\Phi)$ is smooth tame and hence the map
$$
I: \Phi \in C^\infty_{*,odd}(T_1S^n) \rightarrow I(\Phi) \in C^\infty(S^n \times S^n)
$$
is smooth tame.

Recall the proof of Lemma \ref{lemtransversal}. If
$$
\lambda(\Phi)(v,\theta,t,x)=I_{\exp_v(t\theta)}(\Phi)(x)-\cos(t)I_v(\Phi)(x),
$$
then 
$$
\lambda(\Phi)(v,\theta,t,x)=t \frac{\sin(t)}{t}Q(\Phi)(v,\theta,t,x).
$$
Hence
$$
Q(\Phi)(v,\theta,t,x) = \frac{t}{\sin(t)}\int_0^1 \frac{\partial}{\partial t}\lambda(\Phi)(v,\theta,t h,x)dh.
$$
This implies that the map
$$
Z: \Phi \in C^\infty_{*,odd}(T_1S^n) \rightarrow Z(\Phi) \in C^\infty(B(\mathbb{RP}^n\times \mathbb{RP}^n) \times S^n,\mathbb{R}^{n+1})
$$
is smooth tame, by using formulas \eqref{mapa.Z} and \eqref{mapa.Z2}. This implies $\overline{Z}$ is smooth tame. Since we have that $(q,x)\mapsto  (q, \overline{Z}(\Phi)(q)(x))$ is a diffeomorphism of $B(\mathbb{RP}^n\times \mathbb{RP}^n) \times S^n$
with inverse $(q,y)\mapsto (q, \overline{Z}(\Phi)^{-1}(q)(y))$, the map $\overline{Z}^{-1}$ and
$$
J: \Phi \in C^\infty_{*,odd}(T_1S^n) \rightarrow J(\Phi) \in C^\infty(B(\mathbb{RP}^n\times \mathbb{RP}^n) \times S^n,S^n)
$$
are smooth tame.

Using formulas \eqref{eqauxdistK}, \eqref{V.map}, \eqref{Nphi}, \eqref{lambdaphi} and an argument similar to what we used for the map $Z$, we conclude that the map $k$ is smooth tame.

By Proposition \ref{inverse.tame} with $k_0=k(0,0)$,
 we have
$$||f||_q\leq c_q(||L(k(\rho,\Phi))(f)||_{n-1+q}+||k(\rho,\Phi)||_{C^{3n+q}}||L(k(\rho,\Phi))(f)||_{n-1})$$
for every $f\in C^\infty(\mathbb{RP}^n)$. Hence setting $f=\mathcal{V}((\rho,\Phi),g)$,
$$||\mathcal{V}((\rho,\Phi),g)||_q\leq c_q(||g||_{n-1+q}+||k(\rho,\Phi)||_{C^{3n+q}}||g||_{n-1})$$
for every $g\in C^\infty(\mathbb{RP}^n)$. If $||g||_{n-1}\leq C$, we get
$$||\mathcal{V}((\rho,\Phi),g)||_q\leq c_q(||g||_{n-1+q}+ C\, ||k(\rho,\Phi)||_{C^{3n+q}})$$
for every $q\geq 0$. Hence the tameness of $\mathcal{V}$ follows from the tameness of $k$. By \cite{Ham}, Part II, Theorem 3.1.1,  the map
$\mathcal{V}$ is smooth tame.

\end{proof}

\begin{lem} \label{lemRtame}
	The right-inverse of the Funk transform
	\begin{multline*}
		\mathcal{R} : ((\rho,\Phi),g) \in (U \subset F) \times C^{\infty}({\mathbb{RP}}^n) \\
		\mapsto  \mathcal{F}^{*}(\rho,\Phi)((\mathcal{F}(\rho,\Phi)\circ \mathcal{F}^*(\rho,\Phi))^{-1}(g))\in C^{\infty}(S^n)
	\end{multline*}
	is smooth tame.
\end{lem}

\begin{proof}
	Follows  from {Lemma \ref{lemFandF*tame} and Lemma \ref{lemFF*inversetame}.}
	\end{proof}
	
	\begin{lem} \label{lemStame}
	The solution map
	\begin{equation*}
		\mathcal{S} : ((\rho,\Phi),\psi) \in (U \subset F) \times C^{\infty}_{0,odd}(T_{1}S^n) \\
		\mapsto  \mathcal{S}(\rho,\Phi)(\psi) \in C^{\infty}_{0,odd}(T_{1}S^n), 
	\end{equation*}
	of \eqref{eqS} is smooth tame.
\end{lem}

\begin{proof}

We can identify with a diffeomorphism a neighborhood of $\Sigma_v \times \{v\} \subset T_1S^n$ with $S^{n-1}\times D^n$, where $D^n$ is an $n$-dimensional disk. In that way a function $\phi\in C^\infty(T_1S^n)$ can be locally represented by a function $\phi'$  in $C^\infty(S^{n-1}\times D^n)$.
We can use a diffeomorphism of the form $(x,w) \in S^{n-1}\times D^n  \mapsto (R_w^{-1}(x),w) \in T_1S^n$ where $R:D^n \rightarrow SO(n+1)$ is a smooth map satisfying $R_w(w)=v$ (and $R_v=Id$), thinking of $D$ as a neighborhood of $v\in S^n$ and identifying $S^{n-1}$ with $\Sigma_v$.
Hence the function $\phi'$ is linear in $x$ if and only if the function $\phi_w$ is linear for every $w\in D$. And
$$
\int_{S^{n-1}}\phi'(x,w)\eta'(x,w)dA_{can}(x) =\int_{\Sigma_w}\phi(y,w)\eta(y,w)dA_{can}(y).
$$

We will use \cite{Ham}, Part II, Section 3.3 on elliptic equations. It follows from \eqref{eqH} that the map $\mathscr{L}$ sending $(\rho,\Phi)$ to the linear operator 
$$
\phi \mapsto \frac{d}{dt}_{|t=0}\mathcal{H}(\rho,\Phi+t\phi)(x,v)
$$
seen, using the diffeomorphism, as a map taking values in   
$$C^\infty(S^{n-1}\times D^n, \pi_1^*(D^2(S^{n-1})))$$
is smooth tame. Here $\pi_1(x,w)=x$ and $D^2(S^{n-1})$  is the bundle of coefficients of  linear differential operators of degree 2 on $S^{n-1}$. 

We take $W$ to be the space of linear functions on $\mathbb{R}^n$, 
$j:C^\infty(S^{n-1})\rightarrow W$ to be the orthogonal projection and $i:W \rightarrow C^\infty(S^{n-1})$ to be the inclusion.
For some $C^1$ neighborhood $\tilde{U}$ of $\Delta_{S^{n-1}}+(n-1)$ in $C^\infty(S^{n-1},D^2(S^{n-1}))$, the map
$$
\tilde{L}:\big(\tilde{U} \subset C^\infty(S^{n-1},D^2(S^{n-1}))\big) \times C^\infty(S^{n-1}) \times W \rightarrow C^\infty(S^{n-1})\times W
$$
given by $\tilde{L}(f,h,p)=(k=L(f)h+ip, q=jh)$ is an isomorphism. Here $L(f)$ denotes the linear operator associated to $f$. By \cite{Ham}, Part II, Theorem 3.3.3, the solution
$$
\tilde{S}:\big(\tilde{U} \subset C^\infty(S^{n-1},D^2(S^{n-1}))\big)\times C^\infty(S^{n-1})\times W \rightarrow C^\infty(S^{n-1}) \times W,
$$
$\tilde{S}(f,k,q)=(h,p)$, is smooth tame. Notice that if $k$ is orthogonal to the linear functions and if $\tilde{S}(f,k,0)=(h,p)$, then $h$ is
orthogonal to the linear functions and 
$$
(L(f)h)^\perp = k.
$$

In our case we are dealing with functions of an extra parameter (in $C^\infty(S^{n-1}\times D^n)$). Let $\tilde{U}_D\subset C^\infty(S^{n-1}\times D^n,\pi_1^*D^2(S^{n-1}))$ be a $C^1$ neighborhood of $\Delta_{S^{n-1}}+(n-1)$ such that if
$f\in \tilde{U}_D$ then $f_w\in \tilde{U}$ for every $w\in D$. We can assume that $\mathscr{L}(\rho,\Phi)\in \tilde{U}_D$ for all $(\rho,\Phi)\in U$.

We denote by $C^\infty_\perp(S^{n-1}\times D^n)$ the tame space of $\psi'\in C^\infty(S^{n-1}\times D^n)$ such that $x\mapsto \psi'(x,w)$
is orthogonal to the linear functions for every $w\in D^n$.  The same inductive scheme of
Theorem II.3.3.1 and Theorem II.3.3.3 in \cite{Ham} based on differentiating the equation can be used to show that the map
defined by
$$
\tilde{S}_D(f,\psi')_w=\pi_2\big(\tilde{S}(f_w,\psi'_w,0)\big)
$$
for every $w\in D^n$, where $(f,\psi')\in \tilde{U}_D \times C^\infty_\perp(S^{n-1}\times D^n)$, takes values in $C^\infty_\perp(S^{n-1}\times D^n)$ and is smooth tame. Hence the map
$$
\mathcal{S}': ((\rho,\Phi),\psi') \in (U \subset F) \times C^\infty_\perp(S^{n-1}\times D^n)\rightarrow C^\infty_\perp(S^{n-1}\times D^n)
$$
defined  by 
$$
\mathcal{S}'\big(((\rho,\Phi),\psi')\big)= \tilde{S}_D\big(\mathscr{L}((\rho,\Phi)),\psi'\big)
$$
is smooth tame which proves $\mathcal{S}$ is smooth tame.

\end{proof}

\begin{cor} \label{corQVtame}
	The maps
	\begin{equation*}
		V : U\subset F \times H \rightarrow H \quad \text{and} \quad Q : (U\subset F)\times H \times H \rightarrow H
	\end{equation*}
	of Section \ref{problem.formulation} are smooth tame.
\end{cor}
\begin{proof}
	By definition, $V=(V_1,V_2)$ where
	\begin{equation*}
		V_1(\rho,\Phi)(b,\psi) = \frac{1}{n-1}\mathcal{R}(\rho,\Phi)(b)
	\end{equation*}
	and
	\begin{equation*}
		V_2(\rho,\Phi)(b,\psi) = \mathcal{S}(\rho,\Phi)(\psi - (D_1\mathcal{H}(\rho,\Phi)\cdot f - j C(D_1\mathcal{H}(\rho,\Phi)\cdot f))	
	\end{equation*}	
	with $f=V_1(\rho,\Phi)\cdot (b,\psi)$.	
	By Lemmas  \ref{Citame}, \ref{lemFHtame}, \ref{lemRtame} and \ref{lemStame}, we conclude that $V_1$, $V_2$ and therefore $V$ are all smooth tame. \\	
	\indent The map $Q$ can be written as the map
	\begin{equation*}
		Q(\rho,\Phi)\cdot\{(\widetilde{b},\widetilde{\psi}),(b,\psi)\} = (Q_1(\rho,\Phi)\cdot \{(\widetilde{b},\widetilde{\psi}), (b,\psi)\}, 0 )
	\end{equation*}
	where
	\begin{align*}
	& Q_1(\rho,\Phi)\cdot \{(\widetilde{b},\widetilde{\psi}), (b,\psi)\}(\sigma) & \\
	& \quad \quad = \int_{\Sigma_\sigma}	\widetilde{\psi}(x,v)(V_2(\rho,\Phi)\cdot (b,\psi))(x,v)dA_{can}(x)	& \\
	& \quad \quad \quad - \dashint_{{\mathbb{RP}}^n}\left( \int_{\Sigma_\sigma}	\widetilde{\psi}(x,v)(V_2(\rho,\Phi)\cdot (b,\psi))(x,v)dA_{can}(x)\right) dA_{can}(\sigma) &
	\end{align*}
	The tameness of $Q$ follows from the tameness of $V_2$. 

\end{proof}


\section{Proofs of the Main Theorems} We prove the theorems stated in the Introduction.

\subsection{Proof of Theorem A}
The Fr\'echet spaces appearing in the definitions of the maps $\Lambda$, $V$ and $Q$ are tame Fr\'echet spaces with their standard gradings by Lemma \ref{lemFHtame}. The map $\Lambda$ defined in {Section \ref{problem.formulation}} is smooth tame by Lemma  \ref{lemFHtame}, and so are the maps $V$ and $Q$ defined in Section \ref{sectionhypotheses} due to Corollary \ref{corQVtame}. Hence, we can apply Theorem \ref{thmift} and conclude the existence of an open subset $W\subset U\subset F$ containing the origin and a smooth tame map
\begin{equation*}
	\Gamma : Ker(D\Lambda(0,0))\cap W \rightarrow \Lambda^{-1}(0)
\end{equation*}
such that $\Gamma(0)=0$ and $D\Gamma(0)v=v$ for all $v\in Ker(D\Lambda(0,0))$.

\begin{prop} \label{propkerLambda}
	The kernel of $D\Lambda(0,0)$ consists of all the pairs $(f,\phi)$ in $C^{\infty}(S^n)\times C^{\infty}_{0,odd}(T_1S^n)$ such that
	\begin{itemize}
		\item[$i)$]	$f$ is the sum of a constant function and an odd function; 
		\item[$ii)$] $\phi$ is such that
		\begin{equation*}
			\Delta_{(\Sigma_v,can)}\phi_v + (n-1)\phi_v = {(n-1)\langle \nabla f,v\rangle} \quad \text{on} \quad \Sigma_v
		\end{equation*}
		for every $v\in S^n$.
	\end{itemize}
	Moreover, every $f$ as in $i)$ uniquely determines $\phi$ as in $ii)$.
\end{prop}

\begin{proof}
{
From \eqref{d1lambda1} we have
\begin{equation*}
	D_1\Lambda_1(0,0)\cdot f = (n-1)\left( \mathcal{F}(0,0)(f) - \dashint_{{\mathbb{RP}}^n}\mathcal{F}(0,0)(f)(\sigma)dA_{can}(\sigma)\right).
\end{equation*}
Also $D_2\Lambda_1(0,0)\cdot \phi=0$ since $\mathcal{H}(0,0)=0$.
Hence, we  see from Lemma A.1 that $D\Lambda_1(0,0)\cdot (f,\phi)=0$ is equivalent to $f$ being the sum of a constant function and an odd function. 

From Proposition \ref{propDH_1}, item $i)$, 
\begin{equation*}
		(D_{1}\mathcal{H}(0,0)\cdot f)(x,v) = (n-1)\langle \nabla f (x),v\rangle
\end{equation*}
for  all $(x,v)$ in $T_1S^n$.  In particular $D_{1}\mathcal{H}(0,0)\cdot f\in C^{\infty}_{0,odd}(T_1S^n)$, because
for each $v\in S^n$ the function $x\mapsto \langle \nabla f(x),v\rangle$ is an even function on $\Sigma_v$. From \eqref{d1lambda2} and \eqref{d2lambda2} we see that $D\Lambda_2(0,0)\cdot(f,\phi)=0$ is equivalent to
\begin{equation*}
   (n-1)\langle \nabla f ,v\rangle + \mathcal P_v(0,0)(\phi_v)=0
	\end{equation*} 
	for all $v\in S^{n}$.
	The result follows from Remark \ref{rmkJacobicanonico} and the fact that $\phi \in C^{\infty}_{0,odd}(T_1S^n)$.
}
\end{proof}

\indent Given an arbitrary $f\in C^{\infty}_{odd}(S^n)$, there is a unique function $\phi\in C^{\infty}_{0,odd}(T_{1}S^n)$ such that 
\begin{equation*}
	(f,\phi) \in ker(D\Lambda(0,0))
\end{equation*}
by Proposition \ref{propkerLambda}. Then, there exists $\delta>0$ such that
\begin{equation*}
	(\rho_t,\Phi_t) = \Gamma(tf,t\phi) \in \Lambda^{-1}(0), \quad t\in (-\delta,\delta),
\end{equation*}
(as in Corollary \ref{corift}) defines a path of conformal metrics $e^{2\rho_t}can$ on the sphere $S^n$ and families $\{\Sigma_{\sigma}(\Phi_t)\}_{\sigma\in {\mathbb{RP}}^n}$ of $e^{2\rho_t}can$-minimal spheres. Notice that $\rho_0=0$ and $\dot{\rho}_0=f$. A map $\lambda:W\rightarrow \mathcal{Z}$ can be defined as
$$
\lambda(f)=e^{2\Gamma_1(f, -\mathcal{S}(0,0) (D_{1}\mathcal{H}(0,0)\cdot f))}can.
$$
 Since $\frac{d}{dt}_{|t=0} area(\Sigma_{\sigma}(\Phi_t),e^{2\rho_t}can)=0$, appropriate deformations in $\mathcal{Z}'$ can be obtained by scaling. This finishes the proof of Theorem A. 

\subsection{Proof of Theorem B}

Let  $$\mathcal{L}=\{\Sigma_{p}: \sigma=[p]\in \RP^3\}$$
 be the space of totally geodesic spheres in $S^3$.

 \begin{prop} \label{propuniqnear} For every neighborhood $\mathcal U$ of $\mathcal{L}$ in the space of smoothly embedded spheres in $S^3$,  there is a  neighborhood $V$ in the smooth topology of  the canonical metric such that if $g\in V$,  every minimal embedded sphere of $(S^3,g)$  lies in $\mathcal U$.
 \end{prop}
 \begin{proof}
 This follows by the compactness result of \cite{ChoSch} applied to the space of minimal spheres and the fact that every minimal embedded sphere in $(S^3,can)$ is totally geodesic \cite{Almgren66}.
  \end{proof}
  
  We can choose   $V$ so that the Perturbation Theorem 3.2 of \cite{Whi} gives a  map
\begin{equation*}
\bar \psi:V\times \RP^3 \rightarrow \mathcal U
\end{equation*}
so that $\sigma\mapsto\bar\psi(g,\sigma)$ is an embedding for all $g\in V$, $\bar \psi(can,[p])=\Sigma_{p}$, and $\bar 
\psi(g,\RP^3)$ contains every minimal embedded sphere with respect to the metric $g$ that lies in $\mathcal U$. Hence by the previous lemma $\bar \psi(g,\RP^3)$ contains every minimal embedded  sphere with respect to $g\in V$.

Suppose $g=e^{2\rho}can\in  V$ is conformal to the standard metric. For $[p]\in \RP^3$,  $\bar \psi(g,[p])$ can be written as a graph over $\Sigma_p$ of a function $
\psi(\rho)_p$  as in \eqref{graph}. This defines a function $\psi(\rho)$ on $T_1S^n$ which according to Perturbation Theorem 3.2 of \cite{Whi}
will be of class $C^k$ (where $k$ can  be chosen a priori) and $C^k$ close to the origin (by adjusting the neighborhood $V$).

  We now claim the existence of $\alpha>0$ and smooth functions
   $$G_i':T_1S^3\times\R\times\R\rightarrow \R\quad i=1,\ldots,4,$$ 
   such that for all  $v\in S^3$, and $ i=1,\ldots,4$,
   \begin{equation}\label{orthogonal.condition}
\int_{\Sigma_{v}}\psi(\rho)_{v} x_idA_{can}=\int_{\Sigma_v}G_i'((x,v),\psi(\rho)_v,|\nabla^{\Sigma_v} \psi(\rho)_v|^2)dA_{can}(x)
\end{equation}
   and, for all $\omega\in T_1S^3,$ $s,t\in \R$, and $i=1,\ldots,4$, the following conditions are met:
  \begin{equation}\label{constant.quadratic}
  |G_i'(\omega,s,t)|+|D_\omega G_i'|(\omega,s,t)\leq \alpha(|s|^2+|t|), 
  \end{equation}
  \begin{equation}\label{constant.quadratic2}
  |D_sG_i'|(\omega,s,t)\leq\alpha(|s|+|t|),  \quad\mbox{and}\quad |D_t G_i'|(\omega,s,t)\leq \alpha.
  \end{equation}
  
  With  $M(4\times 4)=\R^{16}$ being the space of $4\times 4$ matrices  we  choose   $$\Omega:\mathcal U\ \rightarrow M(4\times 4),\quad \Omega(\Sigma)=\left(\int_{\Sigma}x_ix_jdA_{can}\right)_{i,j}.$$
The map $\Omega$ is an embedding of  $\mathcal L$ and let  $L=\Omega(\mathcal L)$. {Let $P$ be  the smooth map which projects a tubular neighborhood of $L$ onto $L$.} The map $\bar\psi$ (as constructed in \cite[Theorem 3.2]{Whi})  has the property that  for all $\sigma \in \RP^3$
{$$P\circ \Omega(\bar\psi(g,\sigma))=P\circ\Omega(\bar\psi(can,\sigma))$$}
and so
\begin{equation}\label{white.condition} \left(\Omega(\bar\psi(g,\sigma))-\Omega(\Sigma_{\sigma})\right)\bot T_{\Omega(\Sigma_{\sigma})} L.
\end{equation}
Notice that
\begin{multline*}
\Omega(\bar\psi(g,[v]))_{ij}-\Omega(\Sigma_{[v]})_{ij}=
\int_{\Sigma_v(\psi(\rho))}x_ix_jdA_{can}- \int_{\Sigma_v}x_ix_jdA_{can} \\
=\int_{\Sigma_v}\psi(\rho)_v(x_i\langle v,e_j\rangle+ x_j\langle v,e_i\rangle)dA_{can}\\
+\int_{\Sigma_v}G_{i,j}((x,v),\psi(\rho)_v,|\nabla^{\Sigma_v} \psi(\rho)_v|^2)dA_{can}(x)
\end{multline*}
where $G_{i,j}:T_1S^3\times\R\times\R\rightarrow \R$ is a smooth function satisfying  similar conditions to \eqref{constant.quadratic} and \eqref{constant.quadratic2}  for some constant $\alpha$. Set
$$G_i((x,v),s,t)=-\sum_{j=1}^4\langle v,e_j\rangle G_{i,j}((x,v),s,t)$$
and
$$G((x,v),s,t)=\sum_{i,j=1}^4\langle v,e_i\rangle \langle v,e_j\rangle G_{i,j}((x,v),s,t).$$
We obtain for $i=1,\ldots,4$,
\begin{align*}
w_i& = \sum_{j=1}^4(\Omega(\bar\psi(g,[v]))_{ij}-\Omega(\Sigma_{[v]})_{ij})\langle v,e_j\rangle
 \\
& =\int_{\Sigma_v}\psi(\rho)_vx_idA_{can}-\int_{\Sigma_v}G_{i}((x,v),\psi(\rho)_v,|\nabla^{\Sigma_v} \psi(\rho)_v|^2)dA_{can}(x).
\end{align*}

For each $Z$ orthogonal to $v$, we have that the vector $(Z_iv_j+Z_jv_i)_{i,j}$ is in $T_{\Omega(\Sigma_{[v]})} L$. Hence condition \eqref{white.condition} implies that
$$
\sum_{i,j} \big(\Omega(\bar\psi(g,[v]))_{ij}-\Omega(\Sigma_{[v]})_{ij}\big)(Z_iv_j+Z_jv_i)_{i,j}=0
$$
for such $Z$. Hence $\sum_i w_iZ_i=0$. This implies that 
vector $w=\sum_{i=1}^4w_ie_i\in \R^4$ is parallel to $v$ and thus $w=\langle w, v\rangle v$. 

On the other hand
$$\langle w, v\rangle=\int_{\Sigma_v}G((x,v),\psi(\rho)_v,|\nabla^{\Sigma_v} \psi(\rho)_v|^2)dA_{can}(x)$$
and so
\begin{multline*}
\int_{\Sigma_v}\psi(\rho)_vx_idA_{can}=\langle v,e_i\rangle\int_{\Sigma_v}G((x,v),\psi(\rho)_v,|\nabla^{\Sigma_v} \psi(\rho)_v|^2)dA_{can}(x)\\
+\int_{\Sigma_v}G_{i}((x,v),\psi(\rho)_v,|\nabla^{\Sigma_v} \psi(\rho)_v|^2)dA_{can}(x).
\end{multline*}
This proves the desired claim.
\medskip

  Given $\phi\in C^1(T_1S^3)$ and $u\in T_vS^3$, we consider $\partial_u\phi\in C^0(\Sigma_v)$ defined as
 $$\partial_u\phi(x)=d\phi_{(x,v)}(-\langle x, u\rangle v,u).$$
We will need the following pointwise estimate \begin{equation}\label{pointwise}
|\partial_u|\nabla^{\Sigma_v} \psi(\rho)_v|^2|\leq \alpha_1(|\nabla^{\Sigma_v} \partial_u\psi(\rho)|^2+|\nabla^{\Sigma_v} \psi(\rho)_v|^2),
\end{equation}
where $\alpha_1$ is some positive constant.

Recall that $\Sigma(\psi(\rho))(x,v)=\cos(\psi(\rho)(x,v))x+\sin(\psi(\rho)(x,v)) v.$
  We define $\beta(\rho)=\rho\circ\Sigma(\psi(\rho)).$
 
From \eqref{eqformulaA} and Lemma \ref{lembasicformulae} (iii)  we see that  there is a smooth function
$$A:T_1S^3\times\R\times\R\times\R\rightarrow \R$$ such that
\begin{multline}\label{expansion.A}
\mathcal A(\rho,\psi(\rho))(v)=4\pi+2\int_{\Sigma_v}\beta(\rho)dA_{can}\\
+\int_{\Sigma_v}A(\beta(\rho),\psi(\rho)_v,|\nabla^{\Sigma_v} \psi(\rho)_v|^2)dA_{can}(x),
\end{multline}
where  $A$ satisfies  for some other positive constant $\alpha$, 
\begin{equation}\label{constant.quadratic3}
  |A(l,s,t)|\leq \alpha(|l|^2+|s|^2+|t|), 
  \end{equation}
  and
  \begin{equation*}
 	|D_lA|(\omega,l,s,t)\leq\alpha(|l|+|s|^2+|t|), \, |D_sA|(\omega,l,s,t)\leq\alpha|s|,\end{equation*}
 \begin{equation} \label{constant.quadratic4}
 	 |D_t A|(\omega,s,t)\leq \alpha.
  \end{equation}
  
We can suppose that the  $C^{3,\alpha}$-norm of $\psi(\rho)$ and $\rho$ are uniformly bounded by some positive constant $C_0$ that is independent of $\rho\in V$.  We use the notation $f \lesssim g$ to mean an inequality of the form $f\leq Cg$ where $C$ is a constant that depends only on $C_0$ and $\alpha$. 
 
 The $k$-derivatives of $\phi\in C^k(T_1S^3)$ are denoted by $D^k\phi$. We define $$\nabla^k\beta(\rho)=\nabla^k \rho\circ \Sigma(\psi(\rho)),\quad k=1,2.$$  With this notation we have for all $(x,v)\in T_1S^3$
$$|D\beta(\rho)|(x,v)\lesssim |\nabla\beta(\rho)|(x,v)$$
and
$$|D^2\beta(\rho)|(x,v)\lesssim |\nabla^2\beta(\rho)|(x,v) +|\nabla\beta(\rho)|(x,v).$$
We use $||\phi||_{H^k(\Sigma_v)}$ to denote $||\phi_v||_{H^k(\Sigma_v)}$. We write  $||\nabla^k\beta(\rho)||_{L^2(\Sigma_v)}$ to denote the $L^2$-norm of $|\nabla^k\rho|\circ \Sigma(\psi(\rho))$ restricted to $\Sigma_v$.

\begin{lem}\label{funk.estimate} 
 We have (by adjusting $V$ if necessary)  that for all metrics $e^{2\rho}can \in V\cap \mathcal{Z}$:
  $$
  ||\psi(\rho)||_{H^2(\Sigma_v)}\lesssim||\nabla\beta(\rho)||_{L^2(\Sigma_v)}
$$
and
$$
 ||\partial_u\psi(\rho)||_{H^2(\Sigma_v)}
\lesssim  ||\nabla\beta(\rho)||_{L^2(\Sigma_v)}+||\nabla^2\beta(\rho)||_{L^2(\Sigma_v)}
$$
for all $v\in S^3$ and unit vectors $u\in T_vS^3$.
 \end{lem}
 
  \begin{proof}
  
  Since the generalized Gauss map $\mathcal{G}(\psi(\rho))$ is a diffeomorphism, for every tangent plane at some point of $S^3$ there is a unique surface among the $\{\Sigma_\sigma(\psi(\rho))\}_{\sigma\in \mathbb{RP}^3}$ tangent to it. Therefore if
$g=e^{2\rho}can \in \mathcal{Z}$ it follows that all $\Sigma_\sigma(\psi(\rho))$ are minimal  in $(S^3,g)$.

Using \eqref{eqH} and the fact that $\mathcal H(\rho,\psi(\rho))=0$ we see that $\mathcal H(0,\psi(\rho))$ can be expressed in terms of $\psi(\rho)$,  $D\psi(\rho)$ and $\nabla \beta(\rho)$. For our purposes the important properties are that, for all $v\in S^3$,
 \begin{equation}\label{l2.norm.h}
 \int_{\Sigma_v}|\mathcal H(0,\psi(\rho))|^2(x,v)dA_{can}(x)\lesssim||\nabla \beta(\rho)||^2_{L^2(\Sigma_v)}
 \end{equation} 
 and, for every unit vector $u\in T_vS^3$,
 \begin{equation}\label{h2.norm.h}
 \int_{\Sigma_v}|\partial_u\mathcal H(0,\psi(\rho))|^2dA_{can}
 \lesssim||\nabla\beta(\rho)||^2_{L^2(\Sigma_v)}+||\nabla^2\beta(\rho)||^2_{L^2(\Sigma_v)}.
 \end{equation}
 
 From the expression for $\mathcal H(0,\psi(\rho))$ given by \eqref{eqH} we see that
 $$E(\rho)=\mathcal J_v(0,0)(\psi(\rho))- \mathcal H(0,\psi(\rho))$$
is a quadratic term:
\begin{equation}\label{error.l2}
|E(\rho)|(x,v) \lesssim \psi(\rho)^2(x,v)+|\nabla^{\Sigma_v}\psi(\rho)|^2(x,v).
\end{equation}
Differentiating both sides of the expression for  $E(\rho)$  we have that for all $u\in T_vS^3$,
$$E_u(\rho)=\mathcal J_v(0,0)(\partial_u\psi(\rho))- \partial_u\mathcal H(0,\psi(\rho))$$
is also a quadratic term:
\begin{equation}\label{uerror.l2}
|E_u(\rho)|(x,v) \lesssim \psi(\rho)^2(x,v)+|\nabla^{\Sigma_v}\psi(\rho)|^2(x,v)+|(\nabla^{\Sigma_v})^2\psi(\rho)|^2(x,v).
\end{equation}

Denote by $\phi_v^{T}$ the projection of $\phi_v$ onto the kernel of $\mathcal J_v(0,0)$. Standard energy estimates, \eqref{l2.norm.h}, and \eqref{error.l2} show that
\begin{align*}
 ||\psi(\rho)||_{H^2(\Sigma_v)}& \lesssim ||\mathcal H(0,\psi(\rho))||_{L^2(\Sigma_v)}+||E(\rho)||_{L^2(\Sigma_v)}+||\psi(\rho)_v^{T}||_{L^2(\Sigma_v)} \\
&\lesssim ||\nabla \beta(\rho)||_{L^2(\Sigma_v)}+||\psi(\rho)||_{W^{1,4}(\Sigma_v)}^2+||\psi(\rho)_v^{T}||_{L^2(\Sigma_v)}.
\end{align*}
Likewise, energy estimates, \eqref{h2.norm.h}, and \eqref{uerror.l2}, show that
\begin{multline}\label{first.order.thmb}
 ||\partial_u\psi(\rho)||_{H^2(\Sigma_v)} \lesssim ||\partial_u\mathcal H(0,\psi(\rho))||_{L^2(\Sigma_v)}+||E_u(\rho)||_{L^2(\Sigma_v)}+||\partial_u\psi(\rho)_v^{T}||_{L^2(\Sigma_v)}\\
\lesssim  ||\nabla\beta(\rho)||_{L^2(\Sigma_v)}+||\nabla^2\beta(\rho)||_{L^2(\Sigma_v)}
 \\+||\psi(\rho)||^2_{W^{2,4}(\Sigma_v)}+||\partial_u\psi(\rho)_v^{T}||_{L^2(\Sigma_v)}.
\end{multline}

From \eqref{orthogonal.condition} and \eqref{constant.quadratic}:
\begin{equation}\label{estimativa.perp}
 ||\psi(\rho)_v^{T}||_{L^2(\Sigma_v)}\lesssim ||\psi(\rho)||_{H^1(\Sigma_v)}^2
 \end{equation}
 and so, using the previous estimate for $||\psi(\rho)||_{H^2(\Sigma_v)}$ we obtain
 $$
  ||\psi(\rho)||_{H^2(\Sigma_v)}\lesssim ||\nabla \beta(\rho)||_{L^2(\Sigma_v)}+||\psi(\rho)||_{H^1(\Sigma_v)}^2+ ||\psi(\rho)||_{W^{1,4}(\Sigma_v)}^2.  $$
The continuity of the map $\psi$  implies that if we reduce the neighborhood $V$ further we can absorb the last term of the right hand side and conclude
\begin{equation}\label{estimativa.l2}
  ||\psi(\rho)||_{H^2(\Sigma_v)}\lesssim||\nabla \beta(\rho)||_{L^2(\Sigma_v)}.
 \end{equation}
 
 We now  estimate $||\partial_u\psi(\rho)_v^{T}||_{L^2(\Sigma_v)}$. Differentiating the left-hand side of \eqref{orthogonal.condition} with respect to $u\in T_vS^3$ we obtain
$$
\int_{\Sigma_v}\partial_u\psi(\rho)_vx_idA_{can}-\langle v,e_i\rangle\int_{\Sigma_v}\psi(\rho)_v \langle x,u\rangle dA_{can}.
$$
 Differentiating the right-hand side of \eqref{orthogonal.condition} with respect to the unit vector $u\in T_vS^3$ we obtain (with simplified notation)
 \begin{multline*}
\int_{\Sigma_v}D_{\omega}G_i'.(-\langle x,u \rangle v,u)dA_{can}(x)\\
+ \int_{\Sigma_v}D_sG_i'\partial_u\psi(\rho)_vdA_{can}+\int_{\Sigma_v}D_tG_i'\partial_u|\nabla^{\Sigma_v} \psi(\rho)_v|^2dA_{can}.
 \end{multline*}
 Therefore, using the  pointwise estimate  \eqref{pointwise}, we
deduce from  \eqref{constant.quadratic}, and \eqref{constant.quadratic2}, that
 $$||\partial_u\psi(\rho)_v^{T}||_{L^2(\Sigma_v)}\lesssim ||\psi(\rho)_v^{T}||_{L^2(\Sigma_v)}+||\psi(\rho)||_{H^1(\Sigma_v)}^2+||\partial_u\psi(\rho)||^2_{H^1(\Sigma_v)}.
$$
Using \eqref{estimativa.perp} this simplifies to
$$||\partial_u\psi(\rho)_v^{T}||_{L^2(\Sigma_v)}\lesssim ||\psi(\rho)||_{H^1(\Sigma_v)}^2+||\partial_u\psi(\rho)||^2_{H^1(\Sigma_v)}. $$
 Inserting this inequality in \eqref{first.order.thmb} and using \eqref{estimativa.l2} we have
\begin{multline*}
 ||\partial_u\psi(\rho)||_{H^2(\Sigma_v)}
\lesssim  ||\nabla\beta(\rho)||_{L^2(\Sigma_v)}+||\nabla^2\beta(\rho)||_{L^2(\Sigma_v)}\\
 +||\nabla\beta(\rho)||^2_{L^2(\Sigma_v)}+||\partial_u\psi(\rho)||^2_{H^1(\Sigma_v)}.
\end{multline*}
Reducing the neighborhood $V\subset C^7(S^3)$ if necessary, we can absorb  the last  terms on the right-hand side and conclude
$$
 ||\partial_u\psi(\rho)||_{H^2(\Sigma_v)}
\lesssim  ||\nabla\beta(\rho)||_{L^2(\Sigma_v)}+||\nabla^2\beta(\rho)||_{L^2(\Sigma_v)}.
$$

\end{proof}

Set $c(\rho)=2\pi-\mathcal A(\rho,\psi(\rho))/2.$
 \begin{lem}\label{funk.estimate2} 
  We have (by adjusting $V$ if necessary)  that  for all metrics $e^{2\rho}can\in V\cap \mathcal{Z}$,
$$
 |\mathcal F(\rho)+c(\rho)| 
  \lesssim ||\beta(\rho)||_{L^2(\Sigma_v)}^2+||\nabla\beta(\rho)||_{L^2(\Sigma_v)}^2+||\nabla \rho||^2_{L^2(\Sigma_v)}
$$
and
\begin{multline*}
|\nabla \mathcal F(\rho)(v)|\lesssim ||\nabla \rho||_{L^2(\Sigma_v)}^2+||\nabla^2 \rho||_{L^2(\Sigma_v)}^2\\
+||\beta(\rho)||_{L^2(\Sigma_v)}^2+||\nabla\beta(\rho)||_{L^2(\Sigma_v)}^2+||\nabla^2\beta(\rho)||_{L^2(\Sigma_v)}^2.
\end{multline*}
 \end{lem}

 \begin{proof}
 From \eqref{expansion.A} and \eqref{constant.quadratic3} we have that for all $v\in S^3$
$$
\left|2\int_{\Sigma_v}\beta(\rho)dA_{can}+4\pi-\mathcal A(\rho,\psi(\rho))\right|
\lesssim  ||\beta(\rho)||_{L^2(\Sigma_v)}^2+||\psi(\rho)||_{H^1(\Sigma_v)}^2.
$$
We have that for all $(x,v)\in T_1S^3$
$$
 |\beta(\rho)(x,v)-\rho(x)|
 \lesssim |\nabla \rho(x)|^2+|\psi(\rho)(x,v)|^2 .
$$
 Thus
$$
 |\mathcal F(\rho)+c(\rho)|
  \lesssim ||\beta(\rho)||_{L^2(\Sigma_v)}^2+||\psi(\rho)||_{H^1(\Sigma_v)}^2+||\nabla \rho||^2_{L^2(\Sigma_v)},
$$
which when combined with Lemma \ref{funk.estimate}
implies
$$
 |\mathcal F(\rho)+c(\rho)| 
  \lesssim ||\beta(\rho)||_{L^2(\Sigma_v)}^2+||\nabla\beta(\rho)||_{L^2(\Sigma_v)}^2+||\nabla \rho||^2_{L^2(\Sigma_v)}.
$$
This proves the first estimate. 

To prove the second estimate we use the fact that $c(\rho)$ is  constant  to differentiate \eqref{expansion.A} in the direction $u\in T_vS^3$ (with $u$ being a unit vector) to obtain, in light of  \eqref{constant.quadratic3}, \eqref{constant.quadratic4}, and the pointwise estimate \eqref{pointwise},
\begin{multline}\label{derivada.funk1}
\left|\int_{\Sigma_v}\partial_u\beta(\rho)dA_{can}\right|\lesssim ||\beta(\rho)||_{L^2(\Sigma_v)}^2+||\partial_u\beta(\rho)||_{L^2(\Sigma_v)}^2\\
+||\partial_u\psi(\rho)||_{H^1(\Sigma_v)}^2+||\psi(\rho)||_{H^1(\Sigma_v)}^2.
\end{multline}

Set $\Sigma(t)(x,v):= \cos(t\psi(\rho))x+\sin(t\psi(\rho))v$ and 
$$f(t):=\partial_u(\rho\circ\Sigma(t))\in C^{2,\alpha}(T_1S^3).$$
Notice that $f(1)=\partial_u\beta(\rho)$, $f(0)(x,v)=-\langle x,u\rangle\langle\nabla \rho(x),v\rangle$, and
$$\langle\nabla \mathcal F(\rho)(v),u\rangle=\int_{\Sigma_v} f(0) dA_{can}(x).$$

 An explicit computation shows
\begin{align*}
f'(0) & =\partial_u(\psi(\rho)\langle \nabla \rho(x),v\rangle )\\
& =\langle\nabla \rho,v\rangle\partial_u\psi(\rho)-\langle x,u\rangle\nabla^2\rho(v,v)\psi(\rho)+\langle\nabla \rho,u\rangle\psi(\rho)
\end{align*}
and
$$|f''(t)(x,v)|\lesssim |\partial_u\psi(\rho)(x,v)|^2+|\psi(\rho)(x,v)|^2$$
for all $0\leq t\leq 1$, $(x,v)\in T_1S^3$. Hence
\begin{multline*}
\left|\langle\nabla \mathcal F(\rho)(v),u\rangle-\int_{\Sigma_v}\partial_u\beta(\rho)dA_{can}\right|\leq \int_{\Sigma_v}|f(0)-f(1)|dA_{can}\\\lesssim ||\partial_u\psi(\rho)||_{L^2(\Sigma_v)}^2+||\psi(\rho)||_{L^2(\Sigma_v)}^2+||\nabla \rho||_{L^2(\Sigma_v)}^2+||\nabla^2 \rho||_{L^2(\Sigma_v)}^2.
\end{multline*}
Combining with Lemma \ref{funk.estimate} and \eqref{derivada.funk1}  we deduce
\begin{multline*}
|\langle\nabla \mathcal F(\rho)(v),u\rangle|\lesssim ||\nabla \rho||_{L^2(\Sigma_v)}^2+||\nabla^2 \rho||_{L^2(\Sigma_v)}^2 \\
+||\beta(\rho)||_{L^2(\Sigma_v)}^2+||\nabla\beta(\rho)||_{L^2(\Sigma_v)}^2+||\nabla^2\beta(\rho)||_{L^2(\Sigma_v)}^2.
\end{multline*}
The arbitrariness of the unit vector $u\in T_vS^3$ implies the result.

\end{proof}

We need the following lemma:

\begin{lem}\label{co-area.formula.uniqueness} If  $\Phi \in C^{2}_{*,odd}(T_{1}S^n)$ has sufficiently small $C^2$ norm, then the following holds for some positive dimensional constant $c_n$.
Given $f\in C^0(S^n)$,  set $f(\Phi)=f\circ\Sigma(\Phi)\in C^0(T_1S^n)$. Then
$$\int_{S^n}||f(\Phi)||^2_{L^2(\Sigma_v)}dV_{can}(v)\leq c_n||f||^2_{L^2(S^n)}$$
and
$$
 \int_{S^n}||f(\Phi)||^4_{L^2(\Sigma_v)}dV_{can}(v)\leq c_n ||f||^4_{L^4(S^n)}.$$
\end{lem}

\begin{proof}
{Let $\pi_1$ and $\pi_2$ denote the projections of points in $F(\Phi)$ onto the first and second coordinates. 
If $\Phi=0$, $F(0)=T_1S^n=\{(p,v)\in S^n \times S^n: \langle p, v\rangle=0\}$. 
We have that $|Jac(\pi_1)|=\sqrt{det[D\pi_1\circ (D\pi_{1})^{*}]}=\frac{1}{\sqrt{2}}$ and similarly $|Jac(\pi_2)|=\sqrt{det[D\pi_2\circ (D\pi_{2})^{*}]}=\frac{1}{\sqrt{2}}$.

For general $\Phi$,  we can assume that  both Jacobians are in $(1/2,1)$, and similarly that $|Jac(\Sigma_v(\Phi))|\geq 1/2$ for all $v\in S^n$.
Also  we have that $\Sigma_p^*(\Phi)$ is $C^1$ close to  $\Sigma_p^*(0)$ for all $p\in S^n$. Thus we can assume that $area(\Sigma_p^*(\Phi),can)\leq \omega_{n-1}$ for all $p\in S^n$, where $\omega_{n-1}=area(S^{n-1},can)$. (The area of $\Sigma_{p}^*(0)$ is $\omega_{n-1}/2$). \\
\indent 

We have that
$$
\int_{S^n}||f(\Phi)||^2_{L^2(\Sigma_v)}dV_{can}(v)
=\int_{S^n}\int_{\Sigma_v}|f\circ \Sigma(\Phi)|^2(x,v)dA_{can}(x)dV_{can}(v),
$$
hence
\begin{align*}
\int_{S^n}||f(\Phi)||^2_{L^2(\Sigma_v)}dV_{can}(v) & \leq 2\int_{S^n}\int_{\Sigma_v(\Phi)}|f|^2(p)dA_{can}(p)dV_{can}(v)\\
  & =2\int_{S^n} \int_{(\pi_2)_{|F(\Phi)}^{-1}(v)} |f|^2\circ \pi_1dA_{can}dV_{can}(v).
\end{align*}

Using the co-area formula (\cite{Chavel}, Chapter III)  for $\pi_2:F(\Phi)\rightarrow S^n$ we deduce
\begin{multline*}
\int_{S^n}||f(\Phi)||^2_{L^2(\Sigma_v)}dV_{can}(v)\leq 2\int_{F(\Phi)}|f|^2(p)|Jac(\pi_2)|(p,v)d V_{can}(p,v)\\
 = 2\int_{F(\Phi)}|f|^2(p)\frac{|Jac(\pi_2)|(p,v)}{|Jac(\pi_1)|(p,v)}|Jac(\pi_1)|(p,v)d V_{can}(p,v).
\end{multline*}

Applying the co-area formula again for $\pi_1:F(\Phi)\rightarrow S^n$
\begin{align*}
&\int_{S^n}||f(\Phi)||^2_{L^2(\Sigma_v)}dV_{can}(v) & \\
& \quad \quad \leq 2 \int_{S^n}|f|^2(p)\int_{(\pi_1)_{|F(\Phi)}^{-1}(p)}\frac{|Jac(\pi_2)|(p,v)}{|Jac(\pi_1)|(p,v)}dA_{can}(v)dV_{can}(p) &\\
& \quad \quad \leq 8 \, \omega_{n-1}\int_{S^n}|f|^2(p)dV_{can}(p).
\end{align*}

The second inequality follows from the first inequality because  we have $||f(\Phi)||^4_{L^2(\Sigma_v)}\leq c_n ||f(\Phi)^2||^2_{L^2(\Sigma_v)}$ and  $||f^2||^2_{L^2(S^n)}=||f||^4_{L^4(S^n)}$.}
\end{proof}

We have  from Lemma \ref{co-area.formula.uniqueness} that
\begin{align*}
\int_{S^3}|| \beta(\rho)||^4_{L^2(\Sigma_v)}dV_{can} & \lesssim ||\rho||^4_{L^4(S^3)}\\
\int_{S^3}|| \nabla \beta(\rho)||^4_{L^2(\Sigma_v)}dV_{can} &\lesssim ||\nabla \rho||^4_{L^4(S^3)}\\
\int_{S^3}|| \nabla^2\beta(\rho)||^4_{L^2(\Sigma_v)}dV_{can} &\lesssim ||\nabla^2 \rho||^4_{L^4(S^3)}.
\end{align*}
Similar inequalities hold for $||\nabla \rho||^4_{L^2(\Sigma_v)}$  and  $||\nabla^2 \rho||^4_{L^2(\Sigma_v)}$.
Using these  inequalities in Lemma \ref{funk.estimate2}  we have
\begin{multline}\label{norm.estimate.funk}
\int_{S^3}|\mathcal F(\rho)+c(\rho)|^2+|\nabla \mathcal F(\rho)|^2 dV_{can}\\
\lesssim ||\rho||^4_{L^4(S^3)}+||\nabla \rho||^4_{L^4(S^3)}+||\nabla^2 \rho||^4_{L^4(S^3)}.
\end{multline}

\begin{lem} 
For a metric $e^{2\rho}can \in V \cap \mathcal{Z}$  with $\int_{S^3}\rho dV_{can}=0$, we have
$$|c(\rho)|\lesssim ||\rho||^2_{L^2(S^3)}+||\nabla \rho||^2_{L^2(S^3)}.$$ 
\end{lem}

\begin{proof}
Integrating $\mathcal F(\rho)$ over $S^3$ and using the co-area formula we have the identity
$$\int_{S^3}\mathcal F(\rho)(v)dV_{can}=4\pi\int_{S^3}\rho(x)dV_{can}(x)=0.$$
Thus from Lemma \ref{funk.estimate2} we deduce
\begin{align*}
2\pi^2 |c(\rho)| & =\left|\int_{S^3}\mathcal F(\rho)(v)dV_{can}+2\pi^2 c(\rho)\right|=\left|\int_{S^3}(\mathcal F(\rho)+c(\rho))dV_{can}\right| \\
& \lesssim \int_{S^3}|| \beta(\rho)||^2_{L^2(\Sigma_v)}+||\nabla \beta(\rho)||^2_{L^2(\Sigma_v)}+||\nabla \rho||^2_{L^2(\Sigma_v)}dV_{can}\\
& \lesssim ||\rho||^2_{L^2(S^3)}+||\nabla \rho||^2_{L^2(S^3)},
\end{align*}
where in the last inequality we used Lemma \ref{co-area.formula.uniqueness}.
\end{proof}

To prove Theorem B, we can suppose that $\rho$ is even and  $\int_{S^3}\rho dV_{can}=0$.

The previous lemma and \eqref{norm.estimate.funk} show that
\begin{multline*}
||\mathcal F(\rho)||^2_{H^1(S^3)}=\int_{S^3}|\mathcal F(\rho)|^2+|\nabla \mathcal F(\rho)|^2 dV_{can}\\
\lesssim ||\rho||^4_{L^4(S^3)}+||\nabla \rho||^4_{L^4(S^3)}+||\nabla^2 \rho||^4_{L^4(S^3)}=||\rho||_{W^{2,4}(S^3)}^4.
\end{multline*}
Because we are assuming that $\rho$ is even, we know from Lemma A.1 (ii) that
$||\rho||_{L^2(S^3)}\lesssim ||\mathcal F(\rho)||_{H^1(S^3)}.$
 Hence   we obtain 
$ ||\rho||_{L^2(S^3)}\lesssim ||\rho||_{W^{2,4}(S^3)}^2.$
 From Gagliardo--Nirenberg interpolation inequality we have that
$$||\rho||_{W^{2,4}(S^3)}^2\lesssim |\rho|_{C^4(S^3)}||\rho||_{L^2(S^3)}.$$
Hence 
$$
||\rho||_{L^2(S^3)}\lesssim |\rho|_{C^4(S^3)}||\rho||_{L^2(S^3)}.
$$
Thus by adjusting $V$  we can absorb the right-hand side on the left-hand side. This proves that $\rho =0$, and hence $e^{2\rho}can=can$ proving the Theorem. 
}

\subsection{Proof of Theorem C}

\indent Let $W\subset C_{odd}^{\infty}(S^n)$ denote the space of linear functions. Consider the set $W^{\perp}\cap C^{\infty}_{odd}(S^n)$ of functions that are $L^2$-orthogonal to $W$.  Let  $X$ be the set of all $r\in W^{\perp}\cap C^{\infty}_{odd}(S^n)$
that are Morse with critical points $\{x_1, \dots,x_k\}\subset S^n$, depending on $r$,  satisfying
$$
r(x_i)\neq r(x_j)
$$
for $i\neq j$, such that $\nabla^2r(x_i)$ has distinct eigenvalues (with respect to the round metric) for each $i$, and such that 
$\nabla^3r(x_i)\cdot (w,w,w)\neq 0$ for each $i$ and  $w$ eigenvector  of $\nabla^2r(x_i)$. Since $S^n$ can be embedded in some Euclidean space by spherical harmonics of degree three, a classic result says that any smooth function on $S^n$  can be perturbed by one of these to become Morse.  This implies the set of Morse functions is open and dense in  $W^{\perp}\cap C^{\infty}_{odd}(S^n)$, from which one can see $X$ is also open and dense.

Let $r\in X$ and $f\in C^\infty(S^n)$ be the unique solution in $W^\perp \subset C^{\infty}(S^n)$ of
$$
\Delta f + n\, f=-\frac{1}{2(n-1)}r.
$$
Uniqueness implies that $f\in C_{odd}^{\infty}(S^n)$.

We apply Theorem A to $f$ to get $\rho_t$. Let $g(t)=e^{2\rho_t}can$.  The scalar curvature $R_{g(t)}$ of $g(t)$ can be computed as
$$
R_{g(t)}=e^{-2\rho_t}\left(n(n-1)-2(n-1)\Delta\rho_t-(n-1)(n-2)|\nabla \rho_t|^2\right).
$$
Hence
$$
\frac{d}{dt}_{|t=0}R_{g(t)}=-2(n-1)(\Delta f+n\,f)=r.
$$

This implies
$$
\frac{R_{g(t)}-n(n-1)}{t}=r+O(t).
$$
Since the set of Morse functions is open in $C^2(S^n)$, for sufficiently small $t$ we have that $R_{g(t)}$ is Morse and has as many critical points as $r$. If $I_t\in Isom(S^n,g(t))$, we have that $R_{g(t)}(I_t(x))=R_{g(t)}(x)$ for every $x\in S^n$. Then $I_t$ leaves invariant the
set of critical points of $R_{g(t)}$, which in turn converges to the set of critical points of $r$ as $t\rightarrow 0$.

If $\{x_1,\dots,x_k\}$ is the  set of critical points of $r$, let $\{x_1(t),\dots,x_k(t)\}$ be the set of critical points of $R_{g(t)}$ with $x_i(t) \rightarrow x_i$ as $t\rightarrow 0$. We have 
\begin{align*}
R_{g(t)}(x_i(t)) & =n(n-1)+t r(x_i(t)) + O(t^2)\\
& =n(n-1)+t r(x_i)+o(t).
\end{align*}
The conditions satisfied by $r$ then imply that $I_t(x_i(t))=x_i(t)$ for each $i$. Hence $I_t$ is determined by any of the linear maps 
$
DI_t(x_i(t)):T_{x_i(t)}S^n\rightarrow T_{x_i(t)}S^n.
$

Now
$
\frac{d}{dt}_{|t=0}\nabla^2_{g(t)}R_{g(t)}=\nabla^2_{can}r,
$
since $R_{can}$ is constant. Therefore
$$
\frac{\nabla^2_{g(t)}R_{g(t)}(x_i(t))}{t}= \nabla^2_{can}r(x_i(t))+O(t),
$$
which implies that for each $i$ the bilinear form $\nabla^2_{g(t)}R_{g(t)}(x_i(t))$ has distinct eigenvalues with respect to $g(t)$. Since
$I_t$ is an isometry of $(S^n,g(t))$ that fixes the critical points of $R_{g(t)}$, we have that
$$
\nabla^2_{g(t)}R_{g(t)}(x_i(t))(DI_t(x_i(t))\cdot v,DI_t(x_i(t))\cdot v) = \nabla^2_{g(t)}R_{g(t)}(x_i(t))(v,v)
$$
for each $i$ and $v\in T_{x_i(t)}S^n$. Then $DI_t(x_i(t))$ will send an eigenvector of $ \nabla^2_{g(t)}R_{g(t)}(x_i(t))$ into an eigenvector of $ \nabla^2_{g(t)}R_{g(t)}(x_i(t))$ with the same eigenvalue. Since the eigenvalues are distinct, $DI_t(x_i(t))\cdot w=\pm w$ for each eigenvector $w$.

Suppose there is a sequence of positive numbers $t_j\rightarrow 0$ such that $(S^n,g(t_j))$ admits an isometry $I_{t_j}$ that is not the identity map. Then 
$I_{t_j} \rightarrow I \in O(n+1)$ (after maybe passing to a subsequence).  For each $x_i\in S^n$ there exists  $w_i\in T_{x_i}S^n$ eigenvector of $\nabla^2r(x_i)$ such that
$DI(x_i)\cdot w_i=-w_i$. Since  $r(I(x))=r(x)$ for every $x\in S^n$, we get $\nabla^3r(x_i)(w_i,w_i,w_i)=0$.   This is a contradiction
and hence for sufficiently small $t>0$ the metric $(S^n,e^{2\rho_t}can)$ has no nontrivial isometry.  

If $l\in W$, the same arguments apply to $f+l$ instead of $f$ since we also have $\Delta (f+l)+n\, (f+l)=-\frac{1}{2(n-1)}r.$ The space of all such functions $f+l$  is open and dense in $C^\infty_{odd}(S^n)$ as desired. (An alternative proof of the result can be obtained by generalizing \cite{Bes}, 4.71.)

\subsection{Proof of Theorem D}\label{proof.theorem.D}

We will start by establishing  a correspondence between the set of Riemannian metrics on $S^n$ with minimal equators and the set of positive definite Killing symmetric tensors on $(S^n,can)$. The explicit description given here is based on the method of Hangan for the Euclidean space \cite{Han1}. It can also be obtained from \cite{Han1} by use of the gnomonic projection.

\indent Let $(M^n,g)$ be a Riemannian manifold. A symmetric $p$-tensor $k$ on $M$ is called Killing when the symmetrization of its covariant derivative $\nabla^g k$ vanishes. The vector space consisting of Killing symmetric $p$-ten\-sors will be denoted by $\mathcal{K}_{p}(M^n,g)$. For instance, $k\in \mathcal{K}_2(M^n,g)$ if and only if 
\begin{equation*}
	\nabla^{g}k(X,Y,Z) + \nabla^{g}k(Y,Z,X) + \nabla^{g}k(Z,X,Y) = 0 
\end{equation*}
\noindent for all vector fields $X$, $Y$ and $Z$ on $M^n$.\\
\indent Given Killing vector fields $K_1,\ldots, K_p$ on $(M^n,g)$, the symmetric product $K_{1}\odot\ldots\odot K_{p}$ defined by
\begin{equation} \label{eqsymmetricproduct}
	(K_{1}\odot\ldots\odot K_{p})(X_1,\ldots,X_p) = \sum_{\sigma} g(K_{\sigma(1)},X_{1})\cdot\ldots \cdot g(K_{\sigma(p)},X_{p})
\end{equation}
 is an element of $\mathcal{K}_p(M^n,g)$. In the above formula, the summation is over the set of all permutations of elements of the set $\{1,\ldots,p\}$. 
It was proven in \cite{SumTan2}  that all symmetric Killing $p$-tensors in $(S^n,can)$ are linear combinations of symmetric products of $p$ Killing vector fields as in \eqref{eqsymmetricproduct}. The dimension of the space of Killing $p$-tensors of $(S^ n,can)$ has been computed (\cite{Tak}): for every integer $p\geq 1$,
\begin{equation*}
	dim\, \mathcal{K}_{p}(S^n,can) = \frac{1}{n} {n+p\choose p+1}{n+p-1\choose p}.
\end{equation*} 
\indent Given a three-tensor $t$ on $S^n$, we denote by $t^S$ its cyclic symmetrization given by
\begin{equation*}
	t^{S}(X,Y,Z) = t(X,Y,Z) + t(Y,Z,X) + t(Z,X,Y).
\end{equation*}

\begin{prop}  \label{propfundamentalequation}
Let $g$ be a Riemannian metric on an open subset $W$ of the sphere $S^n$. The following  statements are equivalent
\begin{itemize}
\item[$i)$] All $(n-1)$-equators intersecting $W$ are minimal in $(S^n,g)$.
\item[$ii)$] The metric $g$ and the smooth positive function $\psi$ on $W$ uniquely defined by $dV_{g}=\psi dV_{can}$ are such that
	\begin{equation*}
		\left(\nabla^{can}g - \frac{4}{n+1}d\log(\psi)\otimes g\right)^S = 0 \quad \text{on} \quad W.
	\end{equation*}
\end{itemize}	
\end{prop}

\begin{proof}
	A point $p$ in $W$ belongs to the $(n-1)$-equator $\Sigma_{v}$, $v\neq 0$, if and only if $v$ is orthogonal to $p$, that is if and only if $v$ belongs to $T_{p}S^n$. 
	Let $v\in \mathbb{R}^{n+1}\setminus \{0\}$. Define $\hat{V} : S^n \rightarrow \mathbb{R}$ by $\hat{V}(p)=\langle p,v\rangle$. The formula
$
	N^g = \frac{\nabla^{g} \hat{V}}{|\nabla^{g} \hat{V}|_g}
$
defines a unit normal vector field on $\Sigma_v$ in $(S^n,g)$. The second fundamental form of $\Sigma_v$ in $(S^n,g)$ at a point $p$ in $\Sigma_v$ is then given by
\begin{equation*}
	A(X,Y) = g(\nabla^{g}_{X}{N^g},Y) = \frac{1}{|\nabla^{g} \hat{V}|_g} Hess_{g} \hat{V}(X,Y) \,\, \text{for all} \,\, X,Y\in T_{p}\Sigma_v.
\end{equation*}

Notice that $Hess_{can}\hat{V}=-\hat{V}can$, hence $Hess_{can}\hat{V}=0$ on $\Sigma_v$. Therefore
$$
Hess_{g} \hat{V}(X,Y)=XY\hat{V}-(\nabla^g_XY)\hat{V}=(\nabla^{can}_XY-\nabla^g_XY)\hat{V}
$$
for every $p\in \Sigma_v$, $X,Y\in T_pS^n$.
We define the tensor 
\begin{equation*}
	\mathcal{T}_{g}(X,Y,Z) = g(\nabla^{g}_{X}Y - \nabla^{can}_{X}Y,Z).
\end{equation*} 
Since the linear map $v \in T_{p}S^n \mapsto \nabla^{g}\hat{V}(p) \in T_{p}S^{n}$ is an isomorphism,  the condition that all equators passing by $p$ are minimal at $p$ is equivalent to
$$
tr_g^{12}\mathcal{T}_{g}(Z)=\frac{\mathcal{T}_g(Z,Z,Z)}{|Z|_g^2}
$$
for every $Z\in T_pS^n$. This is equivalent to requiring that $\mathcal{S}_g(Z,Z,Z)=0$ for every $Z$, with $\mathcal{S}_g=\mathcal{T}_g-g \otimes tr_g^{12}\mathcal{T}_{g}$. Since $\mathcal{S}_g$ is symmetric in the first two variables, this is equivalent to vanishing of the cyclic symmetrization of $\mathcal{S}_g$.

We have
	$$
	X g(Y,Z) = \nabla^{can}g(Y,Z,X) + g(\nabla^{can}_XY,Z)+ g(\nabla^{can}_XZ,Y)
	$$
	and
	$$
	X g(Y,Z) =g(\nabla^{g}_XY,Z)+ g(\nabla^{g}_XZ,Y).
$$
	By subtracting,
	$$
	\nabla^{can}g(Y,Z,X) =\mathcal{T}_g(X,Y,Z)+\mathcal{T}_g(X,Z,Y)=\mathcal{T}_g(X,Y,Z)+\mathcal{T}_g(Z,X,Y).
	$$
	Hence $(\mathcal{T}_g)^S=\frac12 (\nabla^{can}g)^S$.
	
	Notice that $tr_g^{12}\mathcal{S}_g=(1-n)tr_g^{12}\mathcal{T}_g$, so 
	$$tr_g^{12}(\mathcal{S}_g^S)=(1-n)tr_g^{12}\mathcal{T}_g+2tr_g^{13}(\mathcal{S}_g)=2tr_g^{13}(\mathcal{T}_g)-(n+1)tr_g^{12}(\mathcal{T}_g).$$
	If $\{e_i\}$ is a local $g$-orthonormal positive frame, we have 
	$$\psi^{-1}=vol_{can}(e_1,\dots,e_n).
	$$
	Thus
	$-\psi^{-2}X\psi(p)=-\sum_i \mathcal{T}_g(X,e_i,e_i)\psi^{-1},$
	which implies $d\log \psi=tr_g^{13}(\mathcal{T}_g).$ Hence the proposition follows from the identity
	$$
	\mathcal{S}_g^S=\frac{1}{2}\left(\nabla^{can}g-\frac{4}{n+1}g\otimes d\log\psi+\frac{2}{n+1}g\otimes tr_g^{12}(\mathcal{S}_g^S) \right)^S.
	$$

\end{proof}

We have all the elements  to describe the correspondence:
\begin{thm} \label{thmcorrespondenceAfull}
	Let $W$ be an open subset of the sphere $S^n$, $n\geq 2$. Let $g$ be a Riemannian metric on $W$ and denote by $F_g$ the positive function on $W$ so that
	$
		dV_g = F_g^{(n+1)/4} dV_{can}.
	$
	The map $$g \mapsto k_g = \frac{1}{F_g} g$$ is  a bijection between the set of metrics $g$ with minimal $(n-1)$-equators and the set of positive definite  Killing symmetric two-tensors $k$ of $(W,can)$. The metric $g_k$ corresponding to $k$  is given by $g_k = \frac{1}{D_k} k$ where $dV_k = D_k^{(n-1)/4} dV_{can}$.
	
\end{thm}

\begin{proof}
\noindent Using the notation of Proposition \ref{propfundamentalequation}, we have $\psi = F_{g}^{(n+1)/4}$. Hence, the symmetric two-tensor $k_g = (1/F_g)g$, which is clearly positive definite, satisfies
\begin{equation*}
	\nabla^{can}k_g = \frac{1}{F_g}\nabla^{can}g -\frac{1}{F_{g}^2}(g\otimes dF_{g}) =\frac{1}{F_g}\left( \nabla^{can}g - \frac{4}{n+1}g\otimes d\log(\psi)\right).
\end{equation*}
\noindent By Proposition \ref{propfundamentalequation}, $g$ has minimal $(n-1)$-equators if and only if $k_g$ is  Killing symmetric for the canonical metric. The result follows by checking that $k_{g_k}=k$. 
\end{proof}

We observe that the aforementioned classification  of  Killing symmetric two-tensors on $(S^n,can)$ in terms of Killing vector fields implies, using Theorem \ref{thmcorrespondenceAfull},  that 
all metrics on $S^n$ with minimal $(n-1)$-equators are invariant under the antipodal map.

We will now discuss some examples. We  justify that not all metrics on $S^n$ with minimal equators have constant sectional curvature if $n\geq 3$. Let $K_1,K_2$ be nontrivial Killing vector fields on $(S^n,can)$ that are orthogonal to each other everywhere. This choice is possible if $n\geq 3$.
Let $k(t)=can +t K_1 \odot K_2$. Hence $k(t)$ is a symmetric Killing two-tensor on $(S^n,can)$ which is positive definite for
sufficiently small $t$. Then $k(0)=can$ and $k'(0)=K_1 \odot K_2$. It is not difficult to check that 
$tr_{can} k'(0) =2can(K_1,K_2)=0$ and that $div_{can}k'(0)=-\frac12 d(tr_{can}k'(0))=0$. If $g(t)=g_{k(t)}$ is as in Theorem \ref{thmcorrespondenceAfull}, then $g(0)=can$ and $g'(0)=k'(0)$. Since $g'(0)=k'(0)$ is divergence-free, it is $L^2$-orthogonal in the canonical metric to the Lie derivative $\mathcal{L}_Xcan$ for any vector field $X$ on $S^n$. Hence, for sufficiently small positive $t$, $g_{k(t)}$ has minimal equators and does not have constant sectional curvature.

In dimension $n=3$, using the structure of $S^3$ as the Lie group of unit quaternions, we exhibit positive definite symmetric Killing two-tensors $\overline{k}$ that correspond  to Riemannian metrics $g_{\overline{k}}$ on $S^3$ with minimal equators that are moreover arbitrarily close to $can$ in the smooth topology and whose isometry group is discrete. 
These metrics are invariant by the antipodal map hence this would prove Theorem D.

\indent Let $i$, $j$, $k$ be the basic quaternions. Using quaternionic multiplication, define the left-invariant vector fields 
\begin{equation*}
	X_i(p) = p\cdot i, \quad X_j(p) = p\cdot j, \quad X_k(p) = p\cdot k \quad \text{for all} \quad p\in S^3,
\end{equation*}
\noindent and the right-invariant vector fields
\begin{equation*} 
	Y_i(p) = i\cdot p, \quad Y_j(p) = j\cdot p, \quad Y_k(p) = k\cdot p \quad \text{for all} \quad p\in S^3.
\end{equation*}
Note that these are Killing vector fields of $(S^3,can)$, and also
\begin{equation}\label{can.formula}
 	can = \frac{1}{2}(X_i\odot X_i + X_j\odot X_j  + X_k\odot X_k)  = \frac{1}{2}(Y_i\odot Y_i + Y_j\odot Y_j  + Y_k\odot Y_k).
\end{equation}
The vector fields $X_i$, $X_j$, $X_k$, $Y_i$, $Y_j$, $Y_k$ form a basis of the space of Killing vector fields of $(S^3,can)$.

\indent Choose real numbers $\alpha_1>\alpha_2>\alpha_3>\beta_1>\beta_2>\beta_3>0$ with $\alpha_1$, $\alpha_2$, $\alpha_3$ sufficiently close to $1$ and $\beta_1$, $\beta_2$, $\beta_3$ sufficiently close to $0$ to guarantee that 
\begin{equation} \label{eqdiagonal}
	\overline{k} = \frac{1}{2}\sum_{i=1}^{3} \alpha_i X_i\odot X_i+\frac{1}{2}\sum_{i=1}^{3} \beta_i Y_i\odot Y_i
\end{equation}  
is positive definite by virtue of being sufficiently close to $can$. Apply Theorem \ref{thmcorrespondenceAfull} so to obtain an antipodally-invariant Riemannian metric $g_{\overline{k}}$ on $S^3$ with minimal equators. We claim that the isometry group of  $g_{\overline{k}}$ is discrete. 

\begin{prop} \label{propnaturalsymmetries}
	Let $\mathcal{G}$ be the group consisting of diffeomorphisms of the sphere $S^n$ that map $(n-1)$-equators into $(n-1)$-equators. Every element of $\mathcal{G}$ is of the form $\phi(T)(x)=Tx/|Tx|$ for some $T\in GL(n+1,\mathbb{R})$.
\end{prop}
\begin{proof}
 Since the intersection of distinct $k$-equators is a $(k-1)$-equator, a diffeomorphim $\phi$ in $\mathcal{G}$ maps $k$-equators into $k$-equators for every $k=0,\ldots,n-1$. 
Therefore the map that assigns to each proper vector subspace $V$ of $\mathbb{R}^{n+1}$ the proper vector subspace generated by the $(dim(V)-1)$-equator $\phi(V\cap S^n)$ is a collineation, in the sense that it permutes proper vector subspaces of $\mathbb{R}^{n+1}$ while preserving the partial order induced by inclusion. By the Fundamental Theorem of Projective Geometry, there exists $T\in GL(n+1,\mathbb{R})$ such that, for every proper vector subspace $V$ of $\mathbb{R}^{n+1}$, the subspace $T(V)$ is precisely the subspace generated by $\phi(V\cap S^n)$. We conclude that for every $x$ in $S^n$ the point $\phi(x)$ in $S^n$ must be either equal to $T(x)/|T(x)|$ or to $-T(x)/|T(x)|$. Since $T$ is a linear isomorphism, 
 either $\phi=\phi(T)$ or $\phi=\phi(-T)$.
\end{proof}

As in the proof of Theorem B, every minimal sphere of $(S^3,g_{\overline{k}})$ is an equator (we could also have used the uniqueness result of \cite{GalMir}). An isometry $\psi$ of $(S^3,g_{\overline{k}})$ will send a minimal surface into a minimal surface, hence it permutes equators. Therefore by Proposition \ref{propnaturalsymmetries} every isometry of $(S^3,g_{\overline{k}})$ is of the form $\phi(T)$ for some
$T\in GL(4,\mathbb{R})$.

The correspondence from Theorem  \ref{thmcorrespondenceAfull} implies that $\phi(T)^*g_k = g_k$ if and only if $k\cdot T = k$, where the action of $T\in GL(n+1,\mathbb{R})$ on Killing symmetric two-tensors of $(S^n,can)$ is given by
$$
(k\cdot T)_x =\frac{|Tx|^4}{(\det T)^\frac{4}{n+1}}(\phi(T)^*k)_x.
$$
Notice that $(K_1\odot K_2)\cdot T=(\det T)^{-\frac{4}{n+1}} (T^*K_1T)\odot (T^*K_2T)$.

If $\{T_i\} \subset GL(n+1,\mathbb{R})$ is such that $\phi(T_i)$ converges to the identity map as diffeomorphisms of $S^n$, we have $\phi(T_i)^*dV_{can} \rightarrow dV_{can}$ which implies $|T_ix|^{-(n+1)}\det T_i\rightarrow 1$ for every $x\in S^n$. Hence $\det T_i>0$ and 
$
(\det T_i)^{-\frac{1}{n+1}}T_i\rightarrow Id.
$
Since we can suppose by scaling that $\det T_i=1$, we have that the group of isometries of $(S^n,g_k)$ of the form $\phi(T)$ is discrete if and only if the equation
$$
k \cdot \mathfrak{t}=0,
$$
$\mathfrak{t}\in \mathfrak{sl}(4,\mathbb{R})$ (or equivalently $tr \, \mathfrak{t}=0$), admits only the trivial solution $\mathfrak{t}=0$.
Here $k \rightarrow k \cdot \mathfrak{t}$ is the linearization of the action $k \rightarrow k \cdot T$, $T\in SL(n+1,\mathbb{R})$, at the identity. We are going to check that this property holds with $k=\overline{k}$.

Notice that 
$(K_1\odot K_2) \cdot \mathfrak{t} = ( K_1 \cdot  \mathfrak{t})\odot K_2 + K_1\odot ( K_2 \cdot  \mathfrak{t}),
$
where $K \cdot \mathfrak{t}=\mathfrak{t}^* K + K \mathfrak{t}$.
If we set $W_1,\dots,W_6$ equal to $X_1,X_2,X_3,Y_1,Y_2,Y_3$ in that order, as matrices the basis $\{\frac12W_i\}$ of $\mathfrak{so}(4,\mathbb{R})$ is orthonormal for the scalar product $\langle A,B\rangle =tr (A^*B)$. For every $\mathfrak{t}\in \mathfrak{sl}(4,\mathbb{R})$ we have
	\begin{align*}
		W_p\cdot \mathfrak{t} & = \frac{1}{4}\sum_{q=1}^6 tr((W_p\cdot \mathfrak{t})^{*}W_q)W_q = -\frac{1}{4}\sum_{q=1}^6 tr(\mathfrak{t}^* W_pW_q+W_p\mathfrak{t}W_q)W_q \\
		 & = -\frac{1}{4}\sum_{q=1}^6 tr(W_q^* W_p^* \mathfrak{t}+W_qW_p\mathfrak{t})W_q = -\frac{1}{2}\sum_{q=1}^{6}  tr(W_qW_p\mathfrak{t})W_q.
	\end{align*}

	\indent Hence, for every $\mathfrak{t}\in \mathfrak{sl}(4,\mathbb{R})$ and every $\overline{k}$ as in \eqref{eqdiagonal},  by setting $w_1,\dots, w_6$ equal to $\alpha_1,\alpha_2,\alpha_3,\beta_1,\beta_2,\beta_3$ in that order, 
\begin{align*}
	\overline{k}\cdot \mathfrak{t} & = \sum_{p=1}^6 w_p (W_p\cdot \mathfrak{t})\odot W_p = -\frac{1}{2}\sum_{p,q=1}^6 w_ptr(W_qW_p\mathfrak{t})W_q\odot W_p \\ & = -\frac{1}{2}\sum_{1\leq p< q\leq 6} (w_p tr(W_qW_p\mathfrak{t})+ w_q tr(W_pW_q\mathfrak{t}))W_p\odot W_q,
\end{align*}
where we used the symmetry $K_1\odot K_2= K_2\odot K_1$ and the fact that $W_pW_p=-Id$ for every  $p$.

\indent Since $dim \,\mathcal{K}_2(S^3)=20$ and $\mathcal{K}_2(S^3)$ is generated by products of Killing vector fields,
the symmetry $K_1\odot K_2= K_2\odot K_1$ and the identity \eqref{can.formula} imply that  $\{W_p\odot W_q : \, 1\leq p< q\leq 6\}$ is a linearly independent set. Hence
\begin{equation*}
	\overline{k}\cdot \mathfrak{t} =0 \, \Leftrightarrow \, w_p tr(W_qW_p\mathfrak{t})+ w_q tr(W_pW_q\mathfrak{t}) = 0 \,\, \text{for all} \,\, 1\leq p<q\leq 6.
\end{equation*}
The quaternionic algebra gives that if $1\leq p<q\leq 3$ or $4\leq p<q\leq 6$ we have $W_pW_q=-W_qW_p$, and if $1\leq p\leq 3$, $4\leq q\leq 6$, we have
$W_pW_q=W_qW_p$. In any case by the choice of $\{w_p\}$ we have $tr(W_qW_p\mathfrak{t})=0$ for every $1\leq p<q\leq 6$. Hence $\mathfrak{t}$ is orthogonal to $X_i$, $Y_j$ and $X_iY_j$ for every $1\leq i,j\leq 3$. Since the set $\{X_iY_j\}_{1\leq i,j\leq 3}$ is a basis of the space of symmetric matrices with zero trace, it follows that $\mathfrak{t}=0$. This finishes the proof.

\subsection{Proof of Theorem E}\label{proof.thmE}

Let $Sym_2(S^n)$ denote the space of smooth symmetric two-tensors on $S^n$.  We replace $\mathcal{A}(\rho,\Phi)$ and $\mathcal{H}(\rho,\Phi)$ by $\mathcal{A}(g,\Phi)$ and $\mathcal{H}(g,\Phi)$ to indicate the dependence on a general Riemannian metric, rather than on a conformal factor. Similarly as before we have the  identity
\begin{equation*}
	\frac{d}{dt}_{\vert t=0}\mathcal{A}(g,\Phi+t\phi)(v) = \int_{\Sigma_v}\mathcal{H}(g,\Phi)(x,v)\phi(x,v)dA_{can}(x),
\end{equation*}
 for all $\phi \in C^{\infty}_{*,odd}(T_1S^n)$ (see \cite{Whi}, Theorem 1.1).
 
 The derivative $D_2\mathcal{H}(g,\Phi)$ induces (\cite{Whi}, Theorem 1.1) symmetric second order elliptic operators on each equator $\Sigma_v$ as before, with
	\begin{equation*}
		D_2\mathcal{H}(can,0)(\phi)_v = -\Delta_{can}\phi_v-(n-1)\phi_v  
	\end{equation*}
	for every $\phi\in C^{\infty}_{*,odd}(T_1S^n)$. Therefore we also have 
	$$
\mathcal{P}(g,\Phi):C^{\infty}_{0,odd}(T_{1}S^n) \rightarrow C^{\infty}_{0,odd}(T_{1}S^n)$$
as in \eqref{eqP}, and a solution map
\begin{equation*}
	\mathcal{S}(g,\Phi) : C^{\infty}_{0,odd}(T_1S^n) \rightarrow C^{\infty}_{0,odd}(T_1S^n)
\end{equation*}
as in \eqref{eqS}. \\

\indent The constraint equation still holds true:
\begin{equation*}
	\mathcal{K}(\Phi,\mathcal{H}(g,\Phi)) = d\mathcal{A}(g,\Phi),
\end{equation*}
and $\mathcal{H}(g,\Phi)=0$ if and only if
\begin{equation*}
	d\mathcal{A}(g,\Phi)=0 \quad \text{and} \quad \mathcal{H}(g,\Phi)=j \omega \quad \text{for some} \quad \omega\in \Omega^{1}_{even}(S^n).
\end{equation*}

Thus we can define a map $\Lambda$ as in Section \ref{problem.formulation}. We let	$\Lambda=(\Lambda_1,\Lambda_2)$ where 
\begin{equation*}
	\Lambda_1(g,\Phi) = \mathcal{A}(g,\Phi) - \dashint_{{\mathbb{RP}}^n} \mathcal{A}(g,\Phi)(\sigma)dA_{can}(\sigma)
\end{equation*}
and
\begin{equation*}
	\Lambda_{2}(g,\Phi) = \mathcal{H}(g,\Phi) - j C(\mathcal{H}(g,\Phi)).
\end{equation*}
 Hence
\begin{equation*}
	\Lambda(g,\Phi)=0 \quad \Leftrightarrow \quad \mathcal{H}(g,\Phi)=0.
\end{equation*}
 Moreover, $\Lambda(can,0)=(0,0)$ since the  equators are minimal hypersurfaces in $(S^n,can)$. 
 
 We need to find an approximate right-inverse for $D\Lambda$ to apply Theorem \ref{thmift}. We have
 \begin{equation*}
	D_1\Lambda_1(g,\Phi)\cdot h= \mathcal{F}(g,\Phi)(h)-\dashint_{{\mathbb{RP}}^n}\mathcal{F}(g,\Phi)(h)(\sigma)dV_{can}(\sigma),\nonumber 
	\end{equation*}
	where $ \mathcal{F}(g,\Phi)$ is the generalized Funk transform
	\begin{equation*}
	\mathcal{F}(g,\Phi)(h)(\sigma) = \frac{1}{2}\int_{\Sigma_{\sigma}(\Phi)} tr_{(\Sigma_\sigma(\Phi),g)} h (x) dA_{g}(x)
	\end{equation*}
	for a smooth symmetric two-tensor $h$ on $S^n$. Notice that
	$$
	\mathcal{F}(g,\Phi):Sym_2(S^n) \rightarrow C^\infty(\mathbb{RP}^n).
	$$

Then
\begin{equation*}
	D_2\Lambda_1(g,\Phi)\cdot \phi = D_{2}\mathcal{A}(g,\Phi)\cdot \phi-\dashint_{\RP^n}(D_{2}\mathcal{A}(g,\Phi)\cdot \phi)(\sigma)dV_{can}(\sigma),
\end{equation*}
where, by the definition of the Euler-Lagrange operator $\mathcal{H}(g,\Phi)$,
\begin{equation*}
	(D_{2}\mathcal{A}(g,\Phi)\cdot \phi)([v]) = \int_{\Sigma_v} \mathcal{H}(g,\Phi)(x,v)\phi(x,v) dA_{can}(x).
\end{equation*}
For all $\phi\in  C^{\infty}_{0,odd}(T_{1}S^n)$ and $v\in S^n$ we have
\begin{equation*}
	(D_{2}\mathcal{A}(g,\Phi)\cdot \phi)([v]) = \int_{\Sigma_v} \Lambda_2(g,\Phi)(x,v)\phi(x,v) dA_{can}(x).
\end{equation*}
Likewise, we have 
\begin{equation*}
	D_1\Lambda_2(g,\Phi)\cdot h
	= D_1\mathcal{H}(g,\Phi)\cdot h - j C(D_1\mathcal{H}(g,\Phi)\cdot h)
	\end{equation*}
and
\begin{equation*}
	D_2\Lambda_2(g,\Phi)\cdot \phi = \mathcal P(g,\Phi)(\phi).
\end{equation*}

If we can construct a right-inverse 
$$
\mathcal{R}(g,\Phi): C^\infty(\mathbb{RP}^n)\rightarrow Sym_2(S^n)
$$
 for the generalized Funk transform $ \mathcal{F}(g,\Phi)$,  we can use it to define $V$ and $Q$  as in Section \ref{sectionhypotheses}.

Notice that for $\zeta \in C^{\infty}(S^n)$, we have
\begin{equation*}
	\mathcal{F}(g,\Phi)(\zeta g)(\sigma) = \frac{(n-1)}{2}\int_{\Sigma_\sigma(\Phi)} \zeta(x) dA_{g}(x).
\end{equation*}

\indent We define
\begin{equation*}
	\mathcal{F}^\circ(g,\Phi) : C^{\infty}(\mathbb{RP}^n) \rightarrow  Sym_2(S^n)
\end{equation*}
by setting
\begin{equation*}
	\mathcal{F}^\circ(g,\Phi)(f)_x = \frac{2}{n-1} \left(\int_{\Sigma^*_{x}(\Phi)} f(\tau)U(g,\Phi)(x,\tau)dA_{can}(\tau)\right) g_x
\end{equation*}
for all $x\in S^n$, where 
\begin{equation*}
	U(g,\Phi) = \frac{|Jac(\pi_2)|}{|Jac(\pi_1)|} \in C^{\infty}([F(\Phi)])
\end{equation*}
is the quotient of the Jacobians of the projections of $[F(\Phi)]$ onto the factors $(S^n,g)$ and $(\mathbb{RP}^n,can)$. 

Hence
\begin{align*}
& \mathcal{F}(g,\Phi)\left(\mathcal{F}^\circ(g,\Phi)(f)\right)(\sigma) & \\
& \quad \quad \quad =\int_{\Sigma_\sigma(\Phi)} \left(\int_{\Sigma^*_{x}(\Phi)} f(\tau)U(g,\Phi)(x,\tau)dA_{can}(\tau)\right)  dA_{g}(x) &\\
& \quad \quad \quad = \int_{[B_\sigma(\Phi)]} f(\tau)U(g,\Phi)(x,\tau)|Jac(Q_1)|(\sigma,x,\tau)dV_{g \times can}(x,\tau) &
\end{align*}
by the co-area formula, where $Q_1:[B_\sigma(\Phi)] \rightarrow \Sigma_\sigma(\Phi)$ is the projection onto the first coordinate as in the proof of Proposition \ref{map.is.K}. Using the co-area formula again with $Q_2:[B_\sigma(\Phi)] \rightarrow \mathbb{RP}^n$ we get
\begin{multline*}
\mathcal{F}(g,\Phi)\left(\mathcal{F}^\circ(g,\Phi)(f)\right)(\sigma)\\
=\int_{\mathbb{RP}^n}\Big(\int_{\Sigma_\sigma(\Phi)\cap \Sigma_\tau(\Phi)}f(\tau)U(g,\Phi)(x,\tau)\\
 \cdot \frac{|Jac(Q_1)|(\sigma,x,\tau)}{|Jac(Q_2)|(\sigma,x,\tau)}dA_g(x)\Big)dV_{can}(\tau).
\end{multline*}

The proofs of Propositions \ref{propdualfunk} and \ref{map.is.K} give that
$$
\frac{|Jac(\pi_2)|}{|Jac(\pi_1)|}\frac{|Jac(Q_1)|}{|Jac(Q_2)|}=\frac{1}{\sqrt{1-g({\bf N}(g,\Phi)(x,v),{\bf N}(g,\Phi)(x,w))^2}}
$$
for $[v]=\sigma$, $[w]=\tau$, where ${\bf N}(g,\Phi)(x,v),{\bf N}(g,\Phi)(x,w)$ denote unit normals in the $g$-metric to
$\Sigma_\sigma(\Phi), \Sigma_\tau(\Phi)$, respectively, at $x$.

Therefore
\begin{equation*}
\mathcal{F}(g,\Phi)\left(\mathcal{F}^\circ(g,\Phi)(f)\right)(\sigma)
 =\int_{\mathbb{RP}^n}K(g,\Phi)(\sigma,\tau)f(\tau)dV_{can}(\tau),
\end{equation*}
where
\begin{multline*}
K(g,\Phi)(\sigma,\tau)\\
=\int_{\Sigma_\sigma(\Phi)\cap \Sigma_\tau(\Phi)}\frac{1}{\sqrt{1-g({\bf N}(g,\Phi)(x,v),{\bf N}(g,\Phi)(x,w))^2}}dA_g(x).
\end{multline*}

The analysis of this integral operator is the same as before: the kernel is of the form $K(g,\Phi)=k(g,\Phi)/(\eta \circ d_{can})$ where the map $k(g,\Phi)$ extends smoothly to the blow-up $B(\mathbb{RP}^n\times \mathbb{RP}^n)$. The operator
$$
\mathcal{F}(g,\Phi)\circ \mathcal{F}^\circ(g,\Phi)=L(k(g,\Phi)):C^\infty(\mathbb{RP}^n)\rightarrow C^\infty(\mathbb{RP}^n)
$$ 
is a pseudodifferential operator of order $1-n$ that is elliptic and invertible if $(g,\Phi)$ is sufficiently close to $(can,0)$ in the smooth topology. Hence we can define
$
\mathcal{R}(g,\Phi)=\mathcal{F}^\circ(g,\Phi) \circ \left(\mathcal{F}(g,\Phi)\circ \mathcal{F}^\circ(g,\Phi)\right)^{-1}
$
so
$$
\mathcal{F}(g,\Phi) \circ \mathcal{R}(g,\Phi) =id.
$$

As before we can check that the maps $\Lambda, V$ and $Q$ are smooth tame and hence Theorem \ref{thmift} can be applied.
This gives a smooth tame map
\begin{equation*}
	\Gamma : W \subset Ker{D\Lambda(can,0)} \rightarrow \Lambda^{-1}(0,0),
\end{equation*}
defined on an open subset $W$ of $Ker{D\Lambda(can,0)}$, with $\Gamma(0)=0$ and $D\Gamma(0)\cdot v=v$ for every $v\in Ker{D\Lambda(can,0)}$. Theorem E follows as Theorem A once we establish that for each smooth symmetric two-tensor $h$ as in \eqref{eqsomeelementsofthekernel} we can find $\phi \in C^{\infty}_{0,odd}(T_1S^n)$ such that $(h,\phi)\in Ker D\Lambda(can,0)$, and then set $(g,\Phi)=\Gamma(h,t\phi)$. 

This is a consequence of the next proposition.\\

\begin{prop} \label{propkerLambda2}
	The kernel of $D\Lambda(can,0)$ consists of all the pairs $(h,\phi)$ in $Sym_2(S^n)\times C^{\infty}_{0,odd}(T_1S^n)$ such that:
	\begin{itemize}
		\item[$i)$] There exists a constant $c$, a smooth odd function $f$ on $S^n$, a smooth vector field $X$ on $S^n$ and a symmetric transverse-traceless two-tensor tensor $\overline{h}$ on $(S^n,can)$ such that
		\begin{equation*}
			h = (c+f)can + \mathcal{L}_{X}can + \overline{h};
		\end{equation*}
		and 
		\item[$ii)$] The function $\phi$ is such that
		\begin{equation*}
			\Delta_{(\Sigma_v,can)}\phi_v + (n-1)\phi_v = (D_1\mathcal{H}(can,0)\cdot h)_v \quad  \text{on} \quad \Sigma_v
		\end{equation*}
		for every $v\in S^n$.
	\end{itemize}
	Moreover, every $h$ as in $i)$ uniquely determines $\phi$ as in $ii)$.
\end{prop}

A calculation yields
\begin{equation*}
		D_1\mathcal{H}(can,0)(h)(x,v) = -(div_{can} h)_x(v) + \frac{1}{2}dtr_{can}(h)_x(v) + \frac{1}{2} \nabla^{can}_v h_x(v,v)
	\end{equation*}
	for all $(x,v)\in T_1S^n$ and $h\in Sym_2(S^n)$. 
	
	\begin{proof}
	By differentiating the constraint equation, we have
 \begin{align*}
	d(\mathcal{F}(can,0)(h)) & = d(D_{1}\mathcal{A}(can,0)\cdot h) \\ & = \mathcal{K}(0,D_{1}\mathcal{H}(can,0)\cdot h) = - C(D_{1}\mathcal{H}(can,0)\cdot h).
  \end{align*}
Recall also that 
	\begin{equation*}
		D_2\mathcal{H}(can,0)(\phi)_v = -\Delta_{can}\phi_v-(n-1)\phi_v  
	\end{equation*}
is $L^2$-orthogonal to the linear functions on $(\Sigma_v,can)$. This implies that $C(D_2\mathcal{H}(can,0)(\phi))=0$.

\indent From the explicit formulas for $D\Lambda$, we  deduce that 
$$(h,\phi)\in ker D\Lambda(can,0)
$$ 
if and only if $\mathcal{F}(can,0)(h)$ is a constant function on $S^n$ and $\phi_v$ is the unique solution of
 \begin{equation*}
 	\Delta_{(\Sigma_v,can)}\phi_v + (n-1)\phi_v = (D_{1}\mathcal{H}(can,0)\cdot h)_v
 \end{equation*}
 that is $L^2(\Sigma_v,can)$-orthogonal to the linear functions. This solution exists because in this case we have
 $C(D_{1}\mathcal{H}(can,0)\cdot h)=0$. \\
 
 It remains to determine for which symmetric two-tensors $h$ on $S^n$ the function $\mathcal{F}(can,0)(h)$ is constant. Recall that any $h\in Sym_2(S^n)$ can be decomposed as a sum
 \begin{equation*}
 	h = (fcan + \mathcal{L}_{X}can) + \overline{h},
 \end{equation*}
 where $f\in C^{\infty}(S^n)$, $\mathcal{L}_X$ denotes the Lie derivative in the direction of the smooth vector field $X$ on $S^n$, and $\overline{h}$ is transverse-traceless tensor on $(S^n,can)$, that is, a tensor such that $div_{can} \overline{h} = tr_{can} \overline{h} =0$ (\cite{Bes2}, Lemma 4.57). \\
 
 \noindent \textbf{Claim 1}: The Lie derivatives of $can$ are in the kernel of $\mathcal{F}(can,0)$. 
 
\indent If $\psi_t$ denotes the one-parameter family of diffeomorphisms generated by $X\in \mathcal{X}(S^n)$, then 
 $
 	\frac{d}{dt}_{\vert t=0}(\psi_t)^{*}can = \mathcal{L}_Xcan
$
 and, for every $\sigma\in \mathbb{RP}^n$,
 \begin{multline*}
	 	\mathcal{F}(can,0)(\mathcal{L}_Xcan)(\sigma) = D_1\mathcal{A}(can,0)(\mathcal{L}_Xcan)(\sigma) \\
	 	= \frac{d}{dt}_{\vert t=0}area(\Sigma_\sigma,(\psi_{t})^{*} can) = \frac{d}{dt}_{\vert t=0}area(\psi_{-t}(\Sigma_\sigma),(\psi_{t})^{*} can) =0,
\end{multline*}
where we used the minimality of $\Sigma_\sigma$ in $(S^n,can)$. \\
 
\noindent \textbf{Claim 2}: Any transverse-traceless symmetric two-tensor $\overline{h}$ is in the kernel of $\mathcal{F}(can,0)$. 
 
 For $v\in S^n$, the function $f(x)=\langle x,v\rangle$ satisfies $Hess f=-f can$. If $\omega=\overline{h}(\nabla^{can}f,\cdot)$, we get
 $$
 div_{can}\omega=(div_{can}\overline{h})(\nabla^{can}f)+\langle \overline{h},Hess f\rangle=0
 $$
since $\overline{h}$ is divergence-free and traceless. If $\Omega_v$ denotes the hemisphere containing $v$  with $\partial \Omega_v=\Sigma_v$, we have
$$
0=\int_{\Omega_v}(div_{can}\omega) dV_{can}=-\int_{\Sigma_v}\omega_x(v)dA_{can}(x)=-\int_{\Sigma_v}\overline{h}_x(v,v)dA_{can}(x).
$$
But, with $\sigma=[v]$,
\begin{align*}
2\mathcal{F}(can,0)(\overline{h})(\sigma)&=\int_{\Sigma_v}tr_{\Sigma_v,can}\overline{h}(x)dA_{can}(x)\\
&= -\int_{\Sigma_v}\overline{h}_x(v,v)dA_{can}(x)=0
\end{align*}
since $tr_{S^n}\overline{h}=0$.\\

\noindent \textbf{Claim 3}: If $f\in C^{\infty}(S^n)$, then $\mathcal{F}(can,0)(fcan)$ is constant if and only if $f$ is the sum of a constant function and an odd function. 

\indent In fact, since
\begin{equation*}
	\frac{2}{n-1}\mathcal{F}(can,0)(fcan) = \int_{\Sigma_v} f(x)dA_{can}(x)
\end{equation*}
is the standard Funk transform on $(S^n,can)$, the result follows from Lemma A.1. \\
 
\end{proof}


%

\begin{appendix}

\section*{Appendix. The standard Funk transform}\label{appendice}

\indent The following lemma summarizes the most relevant properties of the standard Funk transform $\mathcal{F}$. Its content and method of proof are well-known (see \cite{Gui}, Appendix A and \cite{Str}, \S 4, Lemma 4.3).

\begin{lemA1} \hfill{} \label{Aplemgardingestimate}
	\begin{itemize}
	\item[$i)$] The kernel of $\mathcal{F}$ is the set of odd functions:
	\begin{equation*}
		C^{\infty}_{odd}(S^n)= \{f\in C^{\infty}(S^n):\, \mathcal{F}(f)=0\}.
	\end{equation*}
	\item[$ii)$] For every non-negative integer $k$, $\mathcal{F}$ restricted to $C^{\infty}_{even}(S^n)$ extends to a linear isomorphism
	\begin{equation*}
		\mathcal{F}: H^{k}_{even}(S^n) \mapsto H^{k+\frac{n-1}{2}}_{even}(S^n).
	\end{equation*}
	and for a positive constant $c$,
	\begin{equation*}
		\frac{1}{c}\norm{f}_{k} \leq  \norm{\mathcal{F}(f)}_{k+\frac{n-1}{2}} \leq c\norm{f}_{k} 
	\end{equation*}
			for all  $f\in H^{k}_{even}(S^n)$.

	\end{itemize}
\end{lemA1}

\begin{corA2}\label{Apcorollary}
The standard Funk transform $\mathcal{F}$ of $(S^n,can)$ is such that $\mathcal{F}\circ \mathcal{F}^{*}$ is elliptic,
	{\begin{equation*}
		\mathcal{F}\circ \mathcal{F}^{*} : H^k(\mathbb{RP}^n)\rightarrow H^{k+n-1}(\mathbb{RP}^n)
	\end{equation*}
	is  invertible for all $k\in \N$} and there is a positive constant $c$ such that
	\begin{equation*}
		\frac{1}{c} \norm{g}_{k} \leq \norm{\mathcal{F}\circ\mathcal{F}^*(g)}_{k+n-1} \leq c\norm{g}_{k} 
	\end{equation*}
	 {for all $g\in H^k(\mathbb{RP}^n)$.}
\end{corA2}

\begin{proof}
Since $\mathcal{F}\circ \mathcal{F}^{*}=L(k(0,0))$, it follows from  \eqref{eqkernel} that
$$
(\mathcal{F}\circ \mathcal{F}^{*})(f)(\sigma)=\omega_{n-2}\int_{\mathbb{RP}^n}\frac{1}{\sin d(\sigma,\tau)}f(\tau)dV_{can}(\tau),
$$
where $\omega_{n-2}=area(S^{n-2})$. The ellipticity follows as in \cite{Pal} (Lemma 6.2). Seen as a map from $C^\infty_{even}(S^n)$ to $C^\infty_{even}(S^n)$ we have $\mathcal{F}^*=\frac12 \mathcal{F}$. Hence the invertibility of $\mathcal{F}\circ\mathcal{F}^*$ follows from Lemma  A.1.
\end{proof}

\end{appendix}

\end{document}